\documentclass[11bpt]{amsart}
 \usepackage[utf8]{inputenc}
 \usepackage[margin=1.3in]{geometry}
 \usepackage{url,float,tikz,graphicx,tikz-cd}
 \usepackage{amsmath,amsthm,amsfonts,amssymb,amscd,latexsym,bm,enumitem,mathrsfs,mathtools}
\usepackage[english]{babel}
 \usepackage{csquotes}
 \usepgflibrary{patterns}
 \usetikzlibrary{patterns}

\usepackage{eepic,epsfig,eufrak}
\usepackage[english]{babel}
\usepackage[colorlinks=true]{hyperref}
\hypersetup{colorlinks=true,  linkcolor=blue,  filecolor=blue,  citecolor=blue, urlcolor=blue}

\theoremstyle{plain}
\newtheorem{theorem}[equation]{Theorem}
\newtheorem{lemma}[equation]{Lemma}
\newtheorem{proposition}[equation]{Proposition}
\newtheorem{proposition-definition}[equation]{Proposition-Definition}

\newtheorem{construction}[equation]{Construction}
\newtheorem{corollary}[equation]{Corollary}

\theoremstyle{definition}
\newtheorem{block}[equation]{}
\newtheorem{definition}[equation]{Definition}
\newtheorem{example}[equation]{Example}
\newtheorem{remark}[equation]{Remark}

\newcommand{\spec}{\mathrm{Spec}\,}

\newcommand{\an}{\mathrm{an}}
\newcommand{\sat}{\mathrm{sat}}

\newcommand{\cont}{\mathrm{cont}\,}
\renewcommand{\log}{\mathrm{log}}
\newcommand{\ord}{\mathrm{ord}}

\newcommand{\Hom}{\mathrm{Hom}}

\newcommand{\id}{\mathrm{Id}}
\newcommand{\add}[1]{\add \emph{#1}}

\newcommand{\val}{\mathrm{val}}
\renewcommand{\div}{\mathrm{div}\,}

\newcommand{\wt}{\mathrm{wt}}
\renewcommand{\sp}{\mathrm{sp}}
\newcommand{\tor}{\mathrm{tor}}
\newcommand{\coker}{\mathrm{coker}}

\newcommand{\sk}{\mathrm{sk}}

\newcommand{\lra}{\longrightarrow}

\newcommand{\gp}{\mathrm{gp}}
\newcommand{\mult}{\mathrm{mult}}
\newcommand{\colim}{\mathrm{colim}}
\newcommand{\codim}{\mathrm{codim}}
\newcommand{\vol}{\mathrm{vol}}
\newcommand{\rec}{\mathrm{rec}}
\newcommand{\cyc}{\mathrm{cyc}}
\renewcommand{\star}{\mathrm{star}}
\newcommand{\link}{\mathrm{link}}

\newcommand{\A}{\mathbb{A}}

\newcommand{\Q}{\mathbb{Q}}

\newcommand{\N}{\mathbb{N}}

\newcommand{\R}{\mathbb{R}}
\newcommand{\Z}{\mathbb{Z}}

\newcommand{\CC}{\mathscr{C}}

\renewcommand{\O}{\mathscr{O}}
\renewcommand{\H}{\mathscr{H}}

\newcommand{\I}{\mathcal{I}}
\newcommand{\DD}{\mathscr{D}}
\newcommand{\XX}{{X}}

\newcommand{\M}{\mathscr{M}}
\newcommand{\YY}{X'}

\newcommand{\p}{\mathfrak{p}}

\newcommand{\E}{\mathfrak{E}}

\newcommand{\FF}{\mathscr{F}}

\newcommand{\D}{\mathcal{D}}
\renewcommand{\add}[1]{$\bigstar\bigstar\bigstar\bigstar\bigstar$\emph{\footnotesize{#1}}}
\numberwithin{equation}{section}

\begin{document}
\title{Harmonic covers of skeleta}

\date{\today}
\author[Art Waeterschoot]{Art Waeterschoot}
\address{KU Leuven\\
Department of Mathematics\\
Celestijnenlaan 200B\\3001 Heverlee \\
Belgium}
\email{art.waeterschoot@kuleuven.be}
\thanks{The author was supported by the Fund for Scientific Research (FWO) Flanders, grants 11F0123N and G0B1721N. The results in this paper are part of my PhD thesis supervised by Johannes Nicaise, I am very grateful for his guidance and advice over the years. I also thank Stefan Wewers, Lars Halvard Halle, Michael Temkin, Yanbo Fang, Antoine Ducros, Micha\"el Maex, Alejandro Vargas and Soham Karwa for valuable discussions.}
\urladdr{\url{https://sites.google.com/view/artwaeterschoot}}



\begin{abstract}
The geometry of a toroidal scheme over a DVR is encoded in a $\Z$-PL space known as the dual polyhedral complex.
Any such dual complex is a skeleton, i.e. a nonarchimedean analytic retract, and admits a combinatorial divisor theory via specialisation. These structures on the dual complex interact via a Poincaré-Lelong slope formula, which interprets specialisations of divisors as Laplacians of PL functions. The main result presented here shows that finite covers of toroidal schemes give harmonic morphisms of dual complexes, i.e. morphisms that preserve the tropical Laplace equation. A crucial ingredient is a balancing condition which is a variant of the tropical multiplicity formula for dual complexes. 

We apply these results to obtain a Riemann-Hurwitz formula for covers of skeleta in any dimension: the Laplacian of the different function is the tropical relative canonical divisor.
\end{abstract}
\maketitle
\tableofcontents
\section{Introduction}
\begin{block}[\textbf{Motivation}] Harmonic functions play an important role in latest developments in nonarchimedean geometry. See for instance the recent work \cite{GJR}, 
where it is shown (\S7.6 in loc. cit.) that general morphisms of Berkovich analytic spaces are \emph{harmonic morphisms} in the sense that harmonic functions pull back to harmonic functions.  
Since it is well-known that Berkovich spaces retract onto piecewise linear subspaces called \emph{skeleta}, one is naturally led to contemplate harmonic morphisms of PL spaces in order to study Berkovich spaces from a polyhedral viewpoint, leading to a combinatorial or ``tropical'' approach to analytic geometry.

Skeleta are commonly constructed as dual complexes of toroidal schemes. We will focus on finite covers of such skeleta, working over discretely valued ground fields of arbitrary residue characteristic.
\end{block}
\begin{block}[\textbf{Summary of the results.}]
		We present four Theorems on skeleta and finite covers.
		\begin{itemize}
			\item First is a \emph{balancing condition} (Theorem~\ref{intro2}) which gives a combinatorial notion of local degree. 
			\item Second is a variant of the \emph{Poincar\'e-Lelong formula} for skeleta (Theorem~\ref{intro pll}). 
			\item Thirdly we show that finite covers of skeleta are \emph{harmonic morphisms} (Theorem~\ref{intro4}). 
			\item Fourth is a \emph{Riemann-Hurwitz formula} for skeleta (Theorem~\ref{intro5}).
		\end{itemize}
\end{block}
		\noindent Let us now give an overview of this work.
\begin{block}[\textbf{Toroidal schemes}]
	
\label{intro dual complexes}
Throughout we work over a discretely valued ground field $k$ whose ring of integers will be denoted by $k^{\circ}$. The residue field $\tilde{k}$ is assumed algebraically closed. 
Log-structures are defined with respect to the Zariski site, and we equip $S=\spec k^{\circ}$ with the ``standard'' log-structure $S^\dagger$ defined via $k^{\circ}\setminus 0\to k^{\circ}$.

In this paper a \emph{toroidal scheme over $S$} is defined as an integral separated finite type $S$-scheme $\XX$ equipped with a Zariski log-structure $$\XX^\dagger=(\XX,\M_\XX)$$ over $S^\dagger$ which is log-regular \cite{K94}, the last condition means $\XX^\dagger$ has some mild ``toric'' singularities that are adequately controlled by the log-structure -- basic properties will be recalled in Section~\ref{sec:log}. Then the log-structure of $\XX^{\dagger}$ is divisorial with respect to the support $$D_{\XX}=\{x\in \XX:\M_{\XX,x}\ne\O_{\XX,x}^\times\},$$ so that $\XX^\dagger$ is equivalently determined by the pair $(\XX,D_{\XX})$ for some ``boundary'' divisor $D_{\XX}$. Toroidal pairs generalise both:
\begin{itemize}
	\item \emph{snc pairs}: $\XX$ is regular and $D_{\XX}$ is strict normal crossings (containing $\XX_s$)
	\item \emph{toric pairs}: $\XX$ is a toric $S$-scheme andµ $D_{\XX}$ is torus-invariant (containing $\XX_s$).
\end{itemize}
We emphasise that a toroidal $S$-scheme $\XX^{\dagger}$ is \emph{not necessarily} (strictly) semistable, meaning that the special fiber $\XX_s$ is possibly not a reduced scheme. 
It is possible to study $S$-schemes which are toroidal with respect to the \'etale topology, but we do not pursue this here.\end{block}

\begin{block}[\textbf{Dual complexes}] A toroidal $S$-scheme $\XX^{\dagger}$ is stratified by the intersections of components of $D_{\XX}$, also called the log-stratification. The \emph{dual complex} $\D(\XX^\dagger)$ is obtained as the dual intersection complex of the log-stratification and is a polyhedral complex which may possibly have some unbounded polyhedra corresponding to horizontal log-strata. An important fact is that the dual complex $\D(\XX^\dagger)$ admits a canonical topological embedding $$\D(\XX^\dagger)\to \XX^\an$$ whose image is called the \emph{skeleton} $\sk(\XX^\dagger)$ of $\XX^\dagger$, see \cite{BM} and precursors \cite{B99,KS,MN,GRW}. The closure $\overline{\sk}(\XX^{\dagger})$ of $\sk(\XX^{\dagger})$ is a retract 
of $\XX^{\an}$ (\ref{retract to compact skeleton}).
\end{block}
\begin{block}[\textbf{Toroidal covers}]
Throughout the paper we denote by $$f:\YY^\dagger\lra \XX^\dagger$$ a finite morphism of toroidal $S^{\dagger}$-schemes or short a \emph{toroidal cover}, and the associated cover of dual complexes is denoted by $$\phi:\D(\YY^\dagger)\to\D(\XX^\dagger)\quad \text{ or sometimes }\quad \phi:\Sigma'\to\Sigma.$$ 

We are primarily interested in such covers because they provide an adequate ``monomialised'' setting to study analytifications of covers of $k$-varieties via degeneration techniques. There is an advantage of working with toroidal schemes as opposed to \emph{snc} pairs: finite covers of \emph{snc} pairs hardly arise in nature -- the so-called ``impossibility of simultaneous resolution'' according to Abhyankar \cite[\S12]{K94}; also see the negative results of \cite{LL}. In contrast, if in the above only $f_\eta$ and $\XX^\dagger$ are given then the normalisation $f:\YY^{\dagger}\to\XX^{\dagger}$ is a toroidal cover assuming $f$ is tamely ramified along $D_\XX$. Kato raised the question whether in wildly ramified situations toroidal covers also exist (see \S\ref{examples f}). 
\end{block}
\begin{block}[\textbf{$\Z$-PL complexes}]\label{intro z-pl} In this paper \emph{$\Z$-PL complex} is defined as an integral $k$-affine strictly convex rational polyhedral complex. Roughly speaking a $\Z$-PL complex is glued from polyhedra defined by systems of inequalities of the form $a_1x_1+\dots+a_nx_n\ge b$ where $b,a_1,\dots,a_n\in\Z$.  
Each polyhedron $\sigma$ of a $\Z$-PL complex $\Sigma$ is equipped with a lattice $M_{\sigma}$ of $\Z$-affine functions and for each pair of faces $\tau\le \sigma$ we have compatible restriction maps $M_{\sigma}\to M_{\tau}$. For toroidal $S$-schemes $\XX^\dagger$ the dual complex $\D(\XX^\dagger)$ comes canonically equipped with a $\Z$-PL structure where the lattices $M_{\sigma}$ consist of toric coordinates on $\sigma$. More precisely every polyhedron $\sigma$ of $\D(\XX^\dagger)$ corresponds to a unique 
generic point $x$ of a log-stratum of $\XX^\dagger$ and the lattice of $\Z$-affine functions $M_{\sigma}$ is then given by $$M_{\sigma}=\CC^{\gp}_{\XX,x},$$ this is the groupification of the characteristic monoid $\CC_x=\M_{\XX,x}^\sharp$ at $x$; details in Section~\ref{sec:dual complex}. Geometrically $M_{\sigma}$ is the lattice of Cartier divisors on $\spec\O_{\XX,x}$ that are supported on $D_{\XX,x}$, the inverse image of $D_{\XX}$ in $\O_{\XX,x}$. This leads to an isomorphism (Lemma~\ref{lem correspondence cartier and pl functions}) between the group of Cartier divisors supported on $D_{\XX}$ and the group of $\Z$-affine functions on $\Sigma$, defined as $$M_{\Sigma}\coloneqq\lim_{\sigma\in \Sigma} M_\sigma.$$
\end{block}

\begin{block}[\textbf{Local degrees and the lattice index}] Let $\Sigma$ be a polyhedral complex of dimension $n$. By a \emph{facet-ridge pair} $(\sigma,\tau)$ of $\Sigma$ we mean a pair of a \emph{facet} $\sigma$ (a $n$-dimensional polyhedron) of $\Sigma$ and a \emph{ridge} $\tau$ (a $(n-1)$-dimensional polyhedron) of $\Sigma$ such that $\tau$ is a face of $\sigma$. 

Now let $$\phi:\Sigma'\to \Sigma$$ be a finite cover of $n$-dimensional $\Z$-PL complexes.  Consider maps of a facet-ridge pair $(\tau',\sigma')$ of $\Sigma'$ to a given facet-ridge pair $(\tau,\sigma)$ of $\Sigma$ -- see Figure~\ref{fig: cover pl spaces} for a picture. Then, fixing $\tau'$, we call $\phi:\Sigma'\to\Sigma$ \emph{balanced at $\tau'$} if the following sum of lattice indices, called the \emph{local degree} at $\tau'$, is independent of the choice of $\sigma$ and thus well-defined :
\begin{equation}
	\label{intro eqn: deg}
\deg_{\tau'}\phi\coloneqq\sum_{\substack{\text{facets }\, \sigma'\\(\sigma',\tau')\mapsto (\sigma,\tau)}}[M_{\sigma'}:M_{\sigma}].\end{equation} 

The sum is over all facets $\sigma'$ such that $(\sigma',\tau')$ is a facet-ridge pair above a fixed pair $(\sigma,\tau)$. We call $\phi$ a \emph{balanced cover} if it is balanced at every ridge $\tau'$, and then assuming $\Sigma$ is connected we can define the \emph{degree} of $\phi$ as \begin{equation}
	\label{intro eqn: degree}
\deg\phi\coloneqq\sum_{\tau'\mapsto \tau}\deg_{\tau'}\phi,
\end{equation}
independently of the choice of ridge $\tau$ of $\Sigma$.
\end{block}
\refstepcounter{equation}
\begin{figure}[ht!]
		\centering
		\label{fig: cover pl spaces}
		\includegraphics[width=0.4\textwidth]{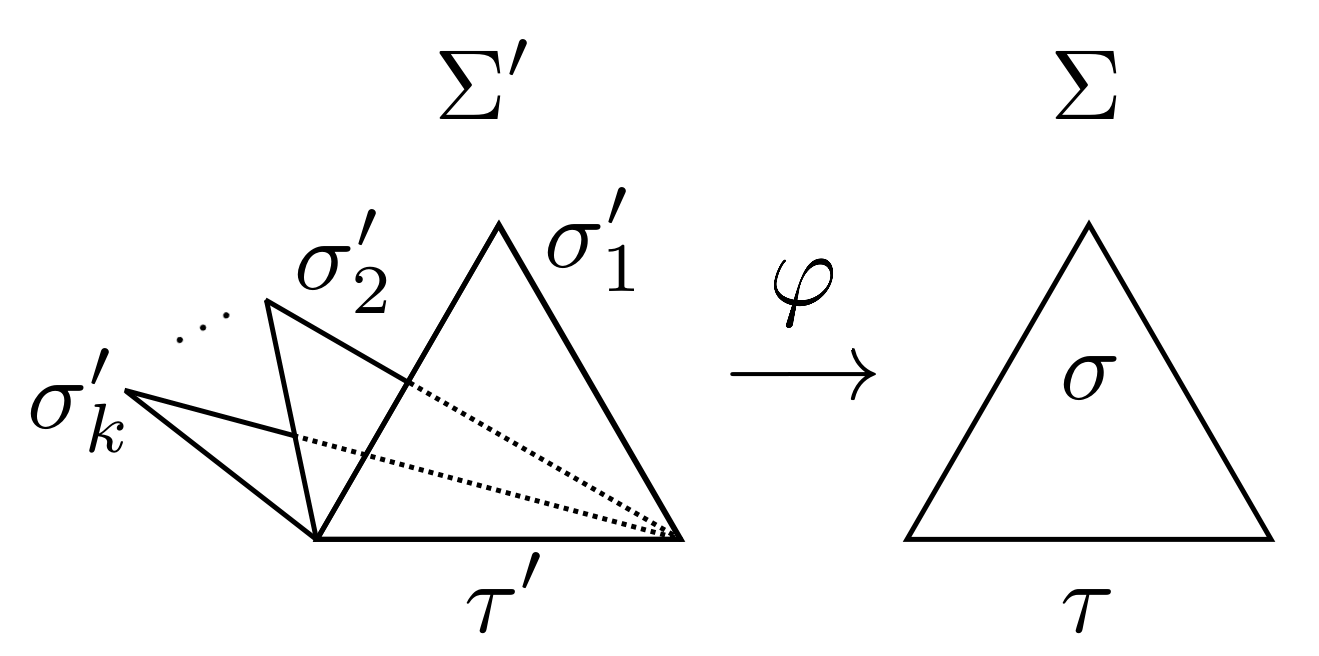}
		\caption{A finite cover of $\Z$-PL complexes is called \emph{balanced} if the local degree $\deg_{\tau'}\varphi\coloneqq \sum_{(\sigma',\tau')\mapsto (\sigma,\tau)}[M_{\sigma'}:M_{\sigma}]$ is independent of $\sigma$ (Definition~\ref{balanced cover Z-pl spaces}).}
	\end{figure}
\begin{theorem}[\textbf{Balancing condition}, See Theorem~\ref{thm: balanced}] \label{intro2} 
 Let $f:\YY^\dagger\to \XX^\dagger$ be a finite morphism of toroidal $S^\dagger$-schemes. Then $$\phi:\D(\YY^\dagger)\to\D(\XX^\dagger)$$ is a balanced cover of $\Z$-PL complexes and $$\deg \phi=\deg f.$$
\end{theorem}
\begin{block}[\textbf{Metrisation of $\Z$-PL complexes}] Any $\Z$-PL complex $\Sigma$ can be equipped with a canonical piecewise Lebesgue measure $\mu_{\Sigma}$ that we call the \emph{conformal measure} such that isometries equal isomorphisms (i.e. unimodular bijections). It is determined by the following \emph{conformality property}, see Lemma~\ref{lemma z-pl mu on polyhedron}: for any finite cover of bounded $\Z$-PL polyhedra $\sigma'\to \sigma$, we have \begin{equation}
	\label{eq intro conformality}\vol(\sigma)/\vol(\sigma')=[M_{\sigma'}:M_{\sigma}].
\end{equation}
\end{block}
\begin{block} For the remainder of the introduction we consider \emph{proper} toroidal $S$-schemes $\XX$ for which the dual complex $\D(\XX^{\dagger})$ is \emph{simplicial}, equivalently Weil divisors supported on $D_\XX$ are $\Q$-Cartier (\ref{rmk qfactoriality}).\end{block}
\begin{block}[\textbf{Specialisation of Cartier divisors}] Given a $\Q$-Cartier divisor $\DD$ on $\XX$ which is supported on $D_\XX$ we have intersection numbers $\DD\cdot C\in\Q$ for any proper regular curve $C\subset \XX$. We define the \emph{specialisation} $\sp_*\DD$ as the combinatorial divisor 

	 \begin{equation}	
\label{intro dfn spec}	
\sp_*\DD\coloneqq \sum_{\tau}\frac{\DD\cdot E_{\tau}}{\mathrm{vol}(\tau^-)}[\tau],
	\end{equation}
	summing over all vertical ridges $\tau$ of $\Sigma$; here $E_{\tau}$ is the $1$-dimensional log-stratum associated to $\tau$ and $\tau^{-}$ is a bounded subset of $\tau$ (cf. \ref{def -}). In fact $\sp_*\DD$ only depends on the linear equivalence class of $\DD$.
\end{block}
\begin{block}[\textbf{Laplacians of PL functions}] 

 For the definition of slopes of PL functions, we need to take into account additional structure next to the $\Z$-PL structure of the dual complex $\Sigma=\D(\XX^{\dagger})$. In the spirit of the works \cite{GRW,Car} we additionally decorate $\Sigma$ with some intersection numbers called \emph{$\alpha$-constants} -- this gives rise to the notion of a \emph{$\Z$-PL tropical complex} 
 (Definition~\ref{defn zpl trop}). Informally speaking the slope $\partial_{\sigma/\tau}F$ of a $\Z$-affine PL function $F:\Sigma\to\R$ on a $\Z$-PL tropical complex $(\Sigma,\alpha)$ along a facet-ridge pair $(\sigma,\tau)$ should be thought of as an appropriate directional derivative of $F$ in a direction normal to $\tau$ that is distinguished by the $\alpha$-constants.  . 
The precise definition is given in~\ref{defn slope}. 

A \emph{combinatorial divisor} on $\Sigma$ is a formal sum of ridges with rational coefficients, and the Laplacian of $F$ is defined as the combinatorial divisor $$\Delta(F)=\sum_{(\sigma,\tau)}\partial_{\sigma/\tau}F[\tau],$$ 
summing over facet-ridge pairs $(\sigma,\tau)$ with $\tau$ vertical; and we call $F$ \emph{harmonic} if $\Delta(F)=0$.  
\end{block}

\begin{theorem}[\textbf{Poincar\'e-Lelong slope formula}, see Theorem~\ref{thm: pll}]\label{intro pll}
Let $\DD$ be a $\Q$-Cartier divisor on $\XX$ supported on $D_{\XX}$, and let $F_{\DD}\in M_{\D(\XX^{\dagger})}$ be the associated $\Q$-PL function, as in~\ref{intro z-pl}. Then \begin{equation}
	\label{eq intro pll}
\Delta(F_{\DD})={\sp}_*\DD.
\end{equation}
\end{theorem}
\begin{block} By the slope formula \eqref{eq intro pll} it follows that $F_{\DD}$ is harmonic if and only if $\DD$ is numerically trivial on the vertical $1$-dimensional log-strata. This is for instance the case if $\DD$ is a principal divisor \footnote{In this case the associated functions $F_{\DD}$ are usually called \emph{smooth} in the literature.}.

Given a balanced cover $\phi:\Sigma'\to\Sigma$ of $\Z$-PL tropical complexes (Definition~\ref{defn zpl trop}) we define the pullback of a combinatorial divisor  
via the rule  $\phi^*[\tau]=\sum_{\tau'\mapsto \tau}\deg_{\tau'}\phi[\tau']$ and 
extending linearly.
\end{block}
	\begin{theorem}[\textbf{Harmonic morphisms}, see Theorem~\ref{thm: harmonic morphisms}]\label{intro4} 
Let $\phi:\Sigma'\to\Sigma$ be a balanced cover of $\Z$-PL tropical complexes. Then for any $\Q$-PL function $F$ we have
\begin{equation}\label{eqn change lapl}
	{\Delta}(\phi^*F)=\phi^*{\Delta}(F).
\end{equation}
in particular if $F$ is harmonic then so is $\phi^*F$, that is $\phi$ is a harmonic morphism.
	\end{theorem}
	\begin{block}[An extension to compactified skeleta]
Let $\overline{\Sigma}=\overline{\sk}(\XX^\dagger)$ denote the closure of $\sk(\XX^\dagger)$ in $\XX^{\an}$ -- this is an analog of the \emph{compactified skeleton} introduced in \cite{GRW}. Then there exists a continuous retraction $\overline{\rho}:\XX^{\an}\to \overline{\Sigma}$. One can extend the definition of specialisation to combinatorial divisors on $\overline{\Sigma}$, as well as an extended laplacian, and Theorem~\ref{intro pll} extends to an equality $\overline{\Delta}(F_{\DD})=\overline{\sp}_*\DD$, see~\ref{block compact pll}. 
		
	\end{block}
	\begin{block}[Reminders on Riemann-Hurwitz formula] The classical Riemann-Hurwitz formula can be reformulated as follows. If $g:X\to Y$ is a finite seperable morphism of proper smooth $k$-curves, then the pullback map $$g^*\Omega_{Y/k}\to \Omega_{X/k}$$ yields a global section $\tau_{X/Y}$, called the \emph{trace} section, of the relative canonical bundle $$\omega_{X/Y}\coloneqq \det \Omega_{X/Y}\cong \underline{\Hom}(g^*\Omega_{Y/k},\Omega_{X/k}).$$ The Riemann-Hurwitz formula is then obtained by computing the degree of the relative canonical divisor $K_{X/Y}\coloneqq \div\tau_{X/Y}$ in two ways: $$-\chi(X)+(\deg g)\chi(Y)=\deg \omega_{X/Y}=\sum_{x\in X} \mathrm{length}_{\O_{X,x}}{\Omega_{X/Y,x}}.$$ We also write $\delta_g(x)\coloneqq\mathrm{length}_{\O_{X,x}}{\Omega_{X/Y,x}}$ for the \emph{different} of $g$ at $x$.
	\end{block}
	\begin{theorem}[\textbf{Riemann-Hurwitz formula for skeleta}, see Theorem~\ref{thm:rh}] \label{intro5}
		Let $f:\YY^{\dagger}\to\XX^{\dagger}$ be a finite seperable cover of proper toroidal $S$-schemes with simplicial dual complexes. Denote $\phi:\Sigma'\to\Sigma$ for the assocated cover of skeleta. Then there exists a $\Z$-affine function $\delta:\Sigma'\to\R_{\ge0}$, called the \emph{different function}, such that for each divisorial point $x'\in \Sigma'$ and $x=\phi(x')$ we have $$\delta(x)=\frac{1}{e(\H(x')/k)}\mathrm{length}_{\H(x')^{\circ}}{\Omega^{\log}_{\H(x')^{\circ}/\H(x)^{\circ}}}$$ and moreover \begin{equation}
			\Delta(\delta)=K_{\Sigma'/\Sigma}
			\label{eq intro rh}
		\end{equation}
		where $K_{\Sigma'/\Sigma}=\sum_{\tau'}-\chi_{\Sigma'}(\tau'/\tau)[\tau']$ 
		is the tropical relative canonical divisor (Definition~\ref{def trop rel can}).
	\end{theorem}
	\begin{block}[Weight functions]
		The method of the proof of Theorem~\ref{thm:rh} also gives the following results, further detailed in Section \ref{sec:rh}. 
		\begin{itemize}
			\item Suppose $\omega\in \Gamma(\XX_k,\omega_{\XX_k/k})$ is a nonzero holomorphic canonical form and let $\wt_\omega:\Sigma\to \R_{\ge0}$ denote the associated weight function of \cite{MN}. Then $$\Delta(\wt_{\omega})=K_{\Sigma}-\rho_*\div\omega;$$ this extends a Theorem of Baker-Nicaise \cite{BN} to arbitrary dimension, see \ref{wt}.
			\item The function $\lambda:X^{\div}\to \R_{\ge0}x\mapsto\frac{1}{e(\H(x)/k)}\mathrm{length}_{\H(x')^{\circ}} \Omega^{\log}_{\H(x)^{\circ}/k^{\circ}}$ extends to a $\Z$-affine function on the faces of dual complexes. This gives some evidence for an expectation of Temkin on the piecewise linearity of K\"ahler norms (Remark \ref{rmk: kahler pl}). 
		\end{itemize}		 
	\end{block}
	\begin{block}{(\textbf{Companion paper})} The companion paper \cite{different} contains some examples in the case of curves and gives applications to reduction types of arithmetic curves.

	\end{block}

\begin{block}[\textbf{Related work}]
The balancing formula of~\ref{intro2} is strongly analogous to tropical multiplicity formulas \cite{ST,BPR16,GRW,GJR}, where similar sums of lattice indices play the role of the degree. Our result is different since we consider covers of a skeleton by a skeleton and not covers of a tropicalisation by a skeleton.
Roughly speaking our case corresponds to the setting of tropical multiplicity one. A similar notion of a balanced cover of complexes in all dimensions appears in \cite{vargas}. 

The notion of a harmonic morphism presented in~\ref{intro4} 
closely mimics the eponymous notion in Riemannian geometry \cite{fuglede,ishihara,BW,eells fuglede}, but is not directly related. 
In dimension one discrete analogues were well-studied,
see the works \cite{urakawa,BN09,ABBR}. 

Lebesgue measures on dual complexes were studied in \cite{JN}, where the conformal behaviour with respect to tamely ramified base change was already observed. Our construction of the conformal measures was also influenced by the available theory of metrisation of Berkovich curves \cite{BN,BPR,CTT,ducros}. 

The main inspiration for the slope formula~\ref{intro pll} is the prior work \cite{GRW}, which deals with the strictly semistable case over algebraically closed ground fields and builds on earlier work in \cite{Car,BPR16}. After the paper was written, Yanbo Fang kindly showed the work \cite{gross}, where tropical intersection theory is developed on the dual complex of a toroidal variety over a field. There is some overlap, but his results do not adress mixed characteristic situations. It is an interesting question whether his results on tropical chow groups carry over. Similarly one can wonder how the Poincar\'e-Lelong formula presented here relates to the ones developed in \cite{thuillier,CD,GJR,GK,Mih}, and to the theory of tropical intersection theory.


The different function and skeletal Riemann-Hurwitz formula was introduced and studied systematically in a series of works \cite{T16,T17,CTT,BT} for covers of curves over algebraically closed ground fields, see the references cited in these papers for earlier work on wild ramification of analytic curves. In parallel upcoming work H\"ubner-Temkin \cite{HT} study analytic properties of the different function over any nonarchimedean field, but not the Riemann-Hurwitz formula. 

\end{block}

\begin{block}[\textbf{Overview of the paper}] 

Section~\ref{sec:log} covers preliminaries on logarithmic regularity. For convenience we give a self-contained exposition of the tools from logarithmic geometry that we use. We have put all necessary background on polyhedral geometry and fans in a separate Appendix~\ref{sec: polyhedral}. 
In Section~\ref{sec:dual complex} we review dual complexes of toroidal $S$-schemes, describe their $\Z$-PL structure and prove Theorem~\ref{intro2}.
In Section~\ref{sec: conformal} we construct the conformal measures on $\Z$-PL complexes and give explicit descriptions for dual complexes. Section~\ref{sec: harmonicity} studies specialisation of Cartier divisors and some intersection numbers are computed. 
Section~\ref{sec: tropical harmonicity} introduces slopes and the Laplacian of PL functions and contains the proofs of Theorems~\ref{intro pll} and ~\ref{intro4}. Section \ref{sec:log diff} contains some preparations on logarithmic differentials, our comparison of canonical sheaves and log-canonical sheaves may be of independent interest. In Section \ref{sec:rh} we discuss potential theory of differential forms and prove Theorem \ref{intro5}.
\end{block}

\begin{block}[\textbf{Conventions and notation}]\label{conventions}
We admit a modest familiarity with logarithmic structures \cite{Kato,K94}, geometry of toric fans \cite{oda}, and Berkovich analytifications \cite{B90,BPR}.

Throughout we denote by $k$ a discretely valued field. Its ring of integers is denoted by $k^{\circ}$ and its residue field, assumed algebraically closed, is denoted by $\tilde{k}$. We let $p\ge 1$ denote the characteristic exponent of $\tilde{k}$ and $\varpi$ a choice of uniformiser of $k^{\circ}$. We write $S=\spec k^{\circ}$. The generic and closed point of $S$ are denoted by $\eta$ and $s$ respectively, so $S=\{\eta,s\}$ as a set.

All rings and monoids are assumed to be commutative. All schemes are assumed to be locally Noetherian. Logarithmic structures are denoted by $(\cdot)^{\dagger}$ and defined with respect to the Zariski site. 
If $X$ is a geometric object, then $|X|$ denotes the underlying topological space. Underlined objects like $\underline{\Hom}$ denote sheaves.

Conventions on monoids, cones, fans, polyhedra, and $\Z$-PL spaces are discussed in appendix~\ref{sec: polyhedral}, in summary:
\begin{itemize}
	\item All cones and polyhedra are assumed convex rational polyhedral and $\Z$-PL (i.e. integral $k$-affine, see \ref{def PL space}), they are denoted by $\sigma=(|\sigma|,N_{\sigma})$ where $N_{\sigma}$ is the lattice of integral points of the real linear span $(N_{\sigma})_{\R}$ of $|\sigma|$. We write $M_{\sigma}=N_{\sigma}^*=\Hom(N_{\sigma},\Z)$ for the dual lattice of $\Z$-affine functions on $\sigma$. 
	\item Given a polyhedral complex $\Sigma$ we write $\Sigma(n)$ for the set of $n$-dimensional polyhedra of $\Sigma$.
 \item The letter $M$ is used frequently, so we use the following convention: straight $M$ refers to dual lattice of $\Z$-affine functions and curly $\M$ refers to log-structure. 
 \item If $\M$ is a monoid, we write $\M^\times$ for its group of invertible elements, $\M^{\sharp}\coloneqq \M/\M^\times$ for its sharpification and $\M^{\gp}$ for its groupification. Most monoids we encounter will be \emph{fs}, i.e. fine and saturated.
\end{itemize}


Berkovich $k$-analytification is denoted by $(\cdot)^{\an}$; it will only be used in the end of Section \ref{sec: harmonicity} and in Section \ref{sec:rh}. 
Completed residue fields are denoted by $\mathscr{H}(\cdot)$. 
\end{block}
\refstepcounter{equation}
\begin{table}[ht!]
\begin{tabular}{c|l}
Shorthand & Meaning\\\hline
 \emph{fs} monoid & fine saturated monoid\\
 \emph{snc} divisor & strict normal crossings divisor\\ 
 log-structure & logarithmic structure with respect to Zariski site\\
 $\Z$-PL & $(\Z,|k^\times|)$-piecewise linear\\
$S^\dagger$ & $S=\spec k^{\circ}$ equipped with standard log-structure $k^{\circ}\setminus \{0\}\to k^{\circ}$\\
 $\XX^\dagger$ & a toroidal $S$-scheme\\
 $D_{\XX}$ & support of the log-structure of $\XX^\dagger$, viewed as a reduced divisor.\\
  $F(\XX^\dagger)$ & Kato fan associated to $\XX^\dagger$\\
   $\CC_x=\CC_{\XX,x}$ & characteristic monoid at a point $x\in F(\XX^\dagger)$.\\
  $\Sigma(\XX^\dagger)$ & cone complex associated to $F(\XX^\dagger)$ \\
  $C(x)=\R_{\ge0}\CC_x^\vee$ & cone of $\Sigma(\XX^{\dagger})$ associated to $x\in F(\XX^{\dagger})$, i.e. dual cone of $\CC_{\XX,x}$ \\
 $M_x=M_{C(x)}=\CC_{\XX,x}^{\gp}$ & lattice of $\Z$-linear functions on $C(x)$,\\ 
 $N_x=N_{C(x)}=(\CC_{\XX,x}^\vee)^{\gp}$ & lattice of integral points on $C(x)$\\
  $\langle\cdot,\cdot\rangle:M_x\times N_x\to \Z$ & canonical pairing \\
 $\D(\XX^\dagger)$ & dual polyhedral complex associated to $\XX^\dagger$, i.e. ``$\left[\varpi\right]\mapsto 1$''-slice of $\Sigma(\XX^\dagger)$.\\
 $\sigma_x$ & face of $\D(\XX^\dagger)$ associated to $x\in F(\XX^\dagger)$, i.e. ``$\left[\varpi\right]\mapsto 1$''-slice of $C(x)$.\\ 
 $M_{\sigma_x}=M_{x}$ & lattice of $\Z$-affine functions on $\sigma_x$\\
 $E_{\sigma}$ & closed stratum associated to a polyhedron $\sigma$ of $\D(\XX^{\dagger})$. \\
 $\sk(\XX^\dagger)$ & skeleton of $\XX^\dagger$, i.e. homeomorphic image of $\D(\XX^\dagger)$ in $\XX_\eta^{\an}$.\\
$M_{\sk(\XX^{\dagger})}$ & group of $\Z$-affine functions on $\sk(\sk(\XX^{\dagger})$\\
 $\mu=\mu_{\Sigma}$ & conformal measure on a $\Z$-PL complex $\Sigma$\\
 $(\sigma,\tau)$ & a facet-ridge pair of a PL-space\\
 $f:\YY^\dagger\to\XX^\dagger$ & A toroidal cover, i.e. a finite morphism of toroidal $S$-schemes\\
 $\phi:\D(\YY^\dagger)\to\D(\XX^\dagger)$ & cover of dual complexes associated to $f$\\

 \end{tabular}
 \vspace{5pt}
\caption{List of abbreviations and notations used throughout the paper}
\end{table}
\section{Background on logarithmic regularity}
\label{sec:log}
 \begin{block} In this section we give a short background exposition of log-regular log-schemes \cite{K94}. Some properties of monoids, fans and polyhedral geometry are recalled in Appendix~\ref{sec: polyhedral}. Experienced readers may wish to skip to Section \ref{sec:dual complex} and refer back as needed. Comprehensive references on log-geometry include \cite{ogus,Kato,K94,Niziol,GR,T24}. 
 \end{block}

\begin{block}[Log-schemes] 
Let $X^\dagger=(X,\O_X,u:\M_X\to \O_X)$ be a log-scheme, that is a scheme $(X,\O_X)$ equipped with a log-structure $u:\M_X\to \O_X$. If $X$ is a scheme and $D\subset X$ is a closed subset we write $X(\log D)$ for the log-scheme $(X,\O_X,\O_X\cap i_*\O^\times_{X\setminus D}\to \O_X)$ where $i:X\setminus D\to X$ is the open immersion. If $u:\M_X\to \O_X$ is a pre-log-structure we denote by $u^a:\M_X^a\coloneqq\M_X\oplus_{u^{-1}\O_X^\times}\O_X^\times\to\O_X$ the associated log-structure. Recall that a log-scheme $X^\dagger$ is called \emph{coherent} if $X$ is covered by charts $P_U\lra \Gamma(U,\M_U)$ for some finitely generated monoids $P_U$. Recall that this means that if $\underline{P_U}$ denotes the constant sheaf on $U$, we have an isomorphism $\underline{P_U}^a\overset{\sim}{\lra}\M_U$, in other words there is a \emph{strict} map $U^\dagger\to \A_P\coloneqq (\spec(\Z[P]),\underline{P}^a)$. 
We call $X^\dagger$ \emph{fine} (resp. \emph{fs}) if it is coherent and $\M_X$ has integral (resp. saturated) stalks. Convention: if we call $X^{\dagger}$ integral, then we mean the underlying scheme $X$ is integral.
  The forgetful functor from the category of fine (resp. \emph{fs}) log-schemes to coherent log-schemes admits a right adjoint namely \emph{integralisation} $(\cdot)^{\mathrm{int}}$ (resp. \emph{saturation} $(\cdot)^{\mathrm{fs}}$). Beware that integralisation or saturation may change the underlying scheme. Base change in the category of fine (resp. \emph{fs}) log-schemes is obtained by applying $(\cdot)^{\mathrm{int}}$ (resp. $(\cdot)^{\mathrm{fs}}$) to the naive base change. All \emph{fs} log-schemes $X^\dagger=(X,\M_X)$ admit a chart near $x$ by the \emph{characteristic monoid} $\CC_{X,x}\coloneqq \M_{X,x}^\sharp$, also known as a \emph{neat} chart, see for instance \cite[\S II.2.3]{ogus}. 
   If $X^{\dagger}$ is coherent then by \cite[\S12.2.21]{GR} $\CC_X$ is constructible and the rank function $x\mapsto \mathrm{rank}(x)\coloneqq \mathrm{rank}_{\Z}\CC_{X,x}^{\gp}$ is upper-semicontinuous; so we obtain a \emph{rank stratification} of $X=\sqcup_{r\ge 0}X_r$ by locally closed subsets of equal rank. \label{log-sch}
\end{block}
\begin{block}[log-parameters]\label{log parameters}
	Let $X^{\dagger}$ be a log-scheme, and let $\M^+_{X}\coloneqq\M_X\setminus \M_X^\times$ be the maximal \emph{monoid ideal} of $\M_X$. Let $\I_X$ be the (ring) ideal sheaf of $\O_X$ generated by $\M^+_{X}$ -- we also call $\I_X$ the sheaf of \emph{logarithmic parameters}. We write $$D_X=\mathrm{supp}\,\I_X$$ for the support of $\I_X$. Note $D_X$ consists of the points of $X$ where the log-structure is non-trivial, so we will also call $D_X$ the \emph{support} of the log-structure. 
\end{block}
\begin{block}[Log-regularity] 
  \label{defn: log reg} 
  A log-scheme $X^\dagger$ is defined to be \emph{log-regular} at a point $x\in X$ if $X^\dagger$ is \emph{fs} at $x$ and $V(\I_{X,x})$ defines a regular closed subscheme of $\spec\O_{X,x}$ of codimension equal to $\mathrm{rank}(x)$. 

\label{log regular discussion} 
Any log-regular scheme is normal \cite[\S4.1]{K94} and a \emph{fs} log-scheme log-smooth\footnote{
Recall that a morphism $f^\dagger:X^\dagger\lra Y^\dagger$ of coherent log-schemes is called \emph{log-smooth} if it is locally of finite type and formally log-smooth (defined via infinitesimal liftings of log-thickenings in \cite[\S IV.3.1]{ogus}). Log-smooth morphisms are preserved by composition and \emph{fs} base change \cite[IV.3.2]{ogus}. In practice one uses Kato's criterion \label{kato criterion}
  \cite[\S3.5]{Kato}or \cite[\S3.1.13]{ogus}: 
  if $\phi:P\to Q$ is a coherent chart at $x\in X$ such that $\#\ker\phi^{\gp}$ and $\#\left(\mathrm{coker}\phi^{\gp}\right)_{\mathrm{tor}}$ are invertible in $k(x)$, and $X\to Y\times_{\spec\Z[P]}\spec\Z[Q]$ is smooth (in the usual sense), then $f$ is log-smooth at $x$; and analogously for log-\'etale morphisms.} over a log-regular scheme is log-regular \cite[\S8.2]{K94}. 
\end{block}
\begin{block}[Properties of log-regular log-schemes]
\label{regular log-regular}
 Suppose $X^\dagger$ is a log-regular log-scheme. It follows from \cite[\S11.6]{K94}  
 that the support $D_X$ is of pure codimension one and $X^\dagger=X(\log D_X)$. Also $V(\I_{X,x})\subset \spec \O_{X,x}$ is the intersection of the $\mathrm{rank}(x)$ components of $D_{X,x}$ that contain $x$, where we write $D_{X,x}$ for the inverse image of $D_X$ along $\spec\O_{X,x}\to X$. 
 Observe that $\CC_{X,x}$ is isomorphic to the monoid of effective Cartier divisors on $\spec\O_{X,x}$ supported on $D_{X,x}$.

\end{block}
\begin{example}[\emph{snc} divisors]\label{example snc}
  Suppose that $X$ is a regular integral scheme of dimension $d$ and $D_X$ defines a (reduced) \emph{snc} divisor on $X$ \cite[0CBN]{stacks}; recall this means that at every point $x\in X$ there exists a system of regular parameters $t_1,\dots,t_d\in \O_{X,x}$ 
  such that $D_{X,x}$ is defined by $\prod_{i=1}^rt_i=0$ in $\O_{X,x}$ for some $r\le d$. 
  Then $X^\dagger=X(\log D_X)$ is log-regular: indeed, near $x$ we have a sharp \emph{fs} chart $\N^r\to \O_X:(n_1,\dots,n_r)\mapsto t_1^{n_1}\cdots t_r^{n_r}$ so $\CC_{X,x}\cong\N^r$ and $\mathrm{rank}(x)=r$, and $\I_{X,x}$ is generated by $t_1,\dots,t_r$ with $\codim V(\I_{X,x})=r$.

Conversely if $X^\dagger$ is a log-regular log-scheme then by~\ref{regular log-regular} we have $X^\dagger=X(\log D_X))$ and if the underlying scheme $X$ is regular then $D_X$ is a \emph{snc} divisor: indeed, at each $x\in X$, the quotient of the regular ring $\O_{X,x}$ by $\I_{X,x}$ is regular, which implies that $\I_{X,x}$ is generated by regular parameters by \cite[00NR]{stacks}.
 \end{example}
\begin{block}[Fan of a log-regular log-scheme] \label{subs: kato fan}
  Let $X^\dagger$ be any log-regular log-scheme. Then we construct \emph{the fan $F(X^\dagger)$ associated to $X^\dagger$} as follows. As a topological space, $F(X^\dagger)$ is defined as the subspace $\{x\in X:\, \I_{X,x}=\mathfrak{m}_{X,x}\}$, with $\I_{X}$ as defined in~\ref{log parameters}. We denote by $\CC_{F(X^\dagger)}$ the restriction of $\CC_X$ along the canonical inclusion $F(X^\dagger)\to X$. Then $(F(X^\dagger),\CC_{F(X^\dagger)})$ is a \emph{fan} in the sense of Kato (see Definition~\ref{defn fan}) -- indeed: suppose $U\to X$ is an affine open neighbourhood of $x$ such that $P\to \Gamma(U,\M_U)$ is an \emph{fs} chart of $X^\dagger$ near $x$, then  \cite[12.6.9]{GR} shows that after possibly further shrinking $U$ we have $$F(U^\dagger)\cong (\spec P)^{\sharp}.$$ 
 There is a canonical \emph{strict} morphism of sharp monoidal spaces $$\pi_{X^\dagger}:(X,\CC_X)\to F(X^\dagger),$$ called the \emph{characteristic morphism} of $X^\dagger$, which sends a point $x$ to the prime ideal defined by $I_{X^\dagger,x}$ in $\spec\O_{X,x}$, it is open and in fact initial among all open maps of $(X,\M_X^\sharp)$ to a sharp monoidal space (see op.cit. or \cite[10.2]{K94}), this gives a categorical\footnote{informally, one can think of $\pi$ as a categorical quotient, though unfortunately this does not make sense since fans and log-schemes live in different categories. It is plausible one can remedy this by using Artin stacks instead of Kato fans -- see the work of Ulirsch \cite{uli}.} definition of $F(X^\dagger)$. For every morphism $f:X^\dagger\to Y^\dagger$ of log-regular log-schemes we have an induced map $F(f):F(X^\dagger)\to F(Y^\dagger)$ of fans, obtained via the composition $F(X^\dagger)\to (X,\CC_X)\to (Y,\CC_Y)\to F(Y^\dagger)$. \end{block}

 \begin{block}[log-stratification]\label{log-strata}
Let $X^\dagger$ be log-regular log-scheme. Each point of $F(X^\dagger)$ is locally closed. Since $\pi=\pi_{X^\dagger}:(X,\M_X)\to F(X^\dagger)$ is open, the fibers of $\pi$ define a stratification of $X$ into irreducible locally closed subsets, here called the \emph{log-stratification}, which by \cite[Lemma 2.2.3]{BM} coincides with the stratification defined by the support $D_{X^\dagger}$. Hence the points of $F(X^{\dagger})$ parametrise log-strata. For a point $x\in F(X^\dagger)$ we denote by $\pi^{-1}{x}=E^{\circ}_x$ the associated log-stratum and $E_x=\overline{\{x\}}$ for its schematic closure. We call $E_x^\circ$ and $E_x$ the log-stratum associated to $x$ and closed log-stratum associated to $x$ respectively. The log-stratification refines the rank stratification of \ref{log-sch}. 
\end{block}

\begin{lemma} \label{lemma balanced cospec}  Let $X^\dagger$ be a log-regular log-scheme and let $x\in F(X^{\dagger})$. Let $E_x=\overline{\{x\}}$ be the associated log-stratum. Equip $E_x$ with the log-structure $E_x^\dagger=E_x(\log D_x)$ where $D_x=E_x\setminus E_x^{\circ}$. Then we have the following.
\begin{enumerate}[label=(\roman*)]
	\item The points of $\spec \CC_{X,x}$ parametrise log-strata containing $E_x$.
	\item The log-scheme $E_x^\dagger$ is log-regular
	\item The natural map $\iota:F(E_x^\dagger)\to F(X^\dagger)$ is a homeomorphism onto $\overline{\{x\}}$.
	\item Let $y\in \overline{\{x\}}$. Then the cospecialisation map $\tau_{y,x}:\CC_{X,y}\to \CC_{X,x}$ is surjective 
	\item For all $y\in \overline{\{x\}}$, \begin{equation}
		\label{eq ker tau}
	\ker\tau_{y,x}\cong \CC_{E_x^\dagger,\iota^{-1}(y)}.
	\end{equation}
\end{enumerate}
\end{lemma}
\begin{proof} 

Part (i) follows from the existence of neat charts (\ref{log-sch}) and the descriptions of $F(X^{\dagger})$ in \ref{subs: kato fan} and \ref{log-strata}.

Part (ii): For simplicity write $E=E_x$ and $D=D_{X}$. Let $D_1,\dots, D_m$ be the irreducible components of $D$. Suppose $E=D_1\cap \dots \cap D_r$ where $r=\mathrm{rank}_{X^{\dagger}}(x)$. Then $E^\dagger=E(\log(D_x))$ and $D_x=E\cap (D_{r+1}\cup \dots \cup D_m)$. The strata of $D_x$ are the strata of $D$ contained in $E$, in particular these are regular. For any $y\in E$, the rank $\mathrm{rank}_{E^\dagger}(y)$ is bounded by the number of components $D_i$ containing $E_y$, where $r<i\le m$; this implies \begin{equation}
	\mathrm{rank}_{E^\dagger}(y)\le \mathrm{rank}_{X^\dagger}(y)-r.\label{eqn: ranks add up}\end{equation}
Since $X^\dagger$ is log-regular, we have $$\mathrm{rank}_{X^\dagger}(y)=\mathrm{codim}_{X}(E_y)=\mathrm{codim}_{E}(E_y)+\codim_{X}(E)=\mathrm{codim}_{E}(E_y)+r.$$ Together, this gives $\mathrm{rank}_{E^\dagger}(y)\le \mathrm{codim}_{E}(E_y)$. By \cite[\S2.3]{K94} it follows that equality holds, therefore $E^\dagger$ is log-regular.

Part (iii): The natural map $\iota:F(E^\dagger)\to F(X^\dagger)$ corresponds to the inclusion of the generic points of log-strata of $X^\dagger$ contained in $E$. Since $\iota$ 
 preserves specialisations it is a homeomorphism onto its image. 

Part (iv): Recall the notion of prime ideals and faces of a monoid from~\ref{ideals}. The cospecialisation $\M_{X,y}\to \M_{X,x}$ is obtained by localising $\M_{X,y}$ at the prime ideal corresponding to $x$, which is $\p=\{s\in \M_{X,y}: s(x)=0\}$. 
 So $$\CC_{X,x}\cong\left(\left(\M_{X,y}\right)_\p\right)^\#\cong\M_{X,y}/(\M_{X,y}\setminus \p)$$ and this implies $\tau_{y,x}$ is surjective. In fact, $\tau_{y,x}$ is the quotient of $\CC_{X,y}$ by the submonoid $\CC_{X,y}\setminus \p^\#$, i.e. the face of $\CC_{X,y}$ associated to $x$. 


Part (v): the monoid $\CC_{X,x}$ (resp. $\CC_{X,y}$) is naturally isomorphic with the monoid of effective Cartier divisors on $\spec \O_{X,x}$ (resp. $\spec \O_{X,y}$) supported on the inverse image of $D$. It follows that $\ker\tau_{y,x}$ is isomorphic to the monoid of effective Cartier divisors $\spec \O_{X,y}$ supported on the inverse image of $D_{r+1}\cup\dots\cup D_{m}$, borrowing notation from (ii). On the other hand $\CC_{E,\iota^{-1}(y)}$ is the monoid of effective Cartiers divisors on $\spec\O_{E,y}$ supported on the preimage of $D_y$. 
Restriction of local equations yields a natural surjective monoid morphism $\psi:\ker\tau_{y,x}\to \CC_{E,\iota^{-1}(y)}$. A rank count, using (iii) and \eqref{eqn: ranks add up} then shows that $\psi^{\gp}$ is injective, and hence $\psi$ is injective too. 

\end{proof}
\begin{remark}\label{snc analogy}
	To bring out the analogy with \emph{snc} divisors even more, one can use ideal quotients of monoids (see~\ref{app ideal quotient}) and rephrase \eqref{eq ker tau} as
 $ \CC_{X,y}/\p^\sharp\cong \CC_{E,\iota^{-1}(y)}^0$. We found this viewpoint helpful for describing compactifications of skeleta (see \ref{retract to compact skeleton}), but it will not be used in an essential way.
\end{remark}


\begin{proposition}[Toroidal resolutions]
  \label{toroidal resolutions}
Let $X^\dagger$ be a log-regular log-scheme with fan $F=F(X^\dagger)$, and let $F'\to F$ be a proper subdivision of $F$ (Definition \ref{def subdivision}) with $F'$ locally fs. Then there exists a final object $X'^\dagger$ among all log-schemes above $X^\dagger$ that admit a morphism of monoidal spaces $(X'^\dagger)^\sharp\to F'$ such that the square of characteristic morphisms $$\begin{tikzcd}
(X',\M_{X'}^\sharp) \arrow[r, ""] \arrow[d, ""] & \arrow[d, ""]F' \\
(X,\M_X^\sharp) \arrow[r]    & F(X^\dagger).                                                 
\end{tikzcd} $$ commutes, we suggestively\footnote{Beware that strictly speaking, this notation is incorrect: such a fibered product does not exist because fans and log-schemes live in other categories.} write $$X'^\dagger=X^\dagger\times_FF'.$$ 
Moreover $X'^\dagger$ is log-regular and $(X'^\dagger)^\sharp\to F'$ is the characteristic morphism. 

If additionally $F'\to F$ is a regular subdivision as in~\ref{existence nice proper subdivision} then $X'\to X$ is a resolution of singularities (i.e. a proper birational morphism such that $h^{-1}(X_{\mathrm{reg}})\overset{\sim}{\lra} X_{\mathrm{reg}}$); In that case we will also call $X'^\dagger\to X^\dagger$ a \emph{toroidal resolution}. 
\end{proposition}
\begin{proof} Omitted, this is the main result of \cite[\S9-10]{K94}, also see \cite[\S6.6]{GR} for a more detailed proof. 
  For later reference we recall how to construct $(X')^\dagger$: 
   suppose that the proper subdivision $F'\to F$ is affine-locally given by $\spec\phi:\spec P'\to \spec P$ for some morphism $\phi:P\to P'$ of sharp \emph{fs} monoids. Then $X'$ is locally constructed as $$X\times_{\spec \Z[P]}\spec \Z[Q]$$ where $$Q=\{a\in P^{\gp}:\phi^{\gp}(a)\in P'\},$$ and $Q$ provides a chart for the log-structure on $X'$. An application of Kato's criterion 
   shows that $(X')^\dagger\to X$ is log-smooth, so it follows that $(X')^\dagger$ is log-regular. The morphism $X'\to X$ is proper since $F'\to F$ is a proper subdivision. If $F'$ is regular then $X'$ is regular \cite[12.5.35]{GR}. 
   For more details, consult the works cited above.
\end{proof}
\begin{remark} 
 Also see the work of Niziol \cite[\S5.8]{Niziol} where log-blowups are discussed and it is shown that one can always toroidally resolve Zariski log-regular schemes via a log-blowup -- we won't use this. 
\end{remark}
\begin{block}[The dual cone complex $\Sigma(\XX^\dagger)$.] Recall that locally \emph{fs} (Kato) fans correspond to cone complexes, see Section~\ref{fs fan and cone complex}. Given a log-regular log-scheme $X^\dagger$ we write $\Sigma(X^\dagger)=\Sigma(F(X^{\dagger}))$ for the cone complex associated to $F(X^{\dagger})$ -- its construction is recalled in~\ref{cone complex of a fan}. In summary, it is the geometric realisation of $F(\XX^{\dagger})$ defined by glueing, for all $x\in F(\XX^\dagger)$, the dual cones $$\sigma_{\CC_{F(\XX^\dagger),x}}\coloneqq \Hom(\CC_{F(\XX^\dagger),x},\R_{\ge0})$$ along the cospecialisation maps as face embeddings, see \eqref{eqn: cone maps}. For brevity we will also write $C(x)\coloneqq \sigma_{\CC_{F(\XX^\dagger),x}}\subset \Sigma(\XX^{\dagger})$ for the cone associated to $x\in F(\XX^{\dagger})$. Note that $$\Sigma(\XX^\dagger)=\colim_{x\in F(\XX^\dagger)}C(x)$$ and for any two $x_1,x_2\in F(\XX^\dagger)$ either $C(x_1)\cap C(x_2)=\emptyset$ or we have $C(x_1)\cap C(x_2)=C(x)$ where $x$ is the generic point of $\overline{\{x_1\}}\cap \overline{\{x_2\}}=E_{x_1}\cap E_{x_2}$.

By construction, $\Sigma(\XX^{\dagger})$ is a strictly convex rational polyhedral cone complex. In particular $\Sigma(\XX^{\dagger})$ is a $\Z$-PL complex (see Definition~\ref{def PL space}).  
Without danger of confusion we will identify  $\Sigma(\XX^\dagger)$ with its underlying topological space $|\Sigma(\XX^\dagger)|=F(\XX^\dagger)(\R_{\ge0})$.

Because the construction is functorial, we obtain for any morphism $\YY^{\dagger}\to\XX^{\dagger}$ of log-regular schemes a morphism of cone complexes $\Sigma(\YY^\dagger)\to\Sigma(\XX^\dagger)$.
\end{block}

The following lemma shows that proper subdivisions can be pulled back along morphisms of log-regular schemes. 
\begin{lemma}[pullback of a subdivision] \label{lemma pullback subdivision} Let $f:X'^\dagger\lra X^\dagger$ be a morphism of log-regular log-schemes. Then for any proper subdivision $\Sigma_0\to \Sigma(X^{\dagger})$ (Definition \ref{def cone complex}), there exists a cartesian square (in the \emph{fs} category) of log-regular log-schemes  \begin{equation}
		\label{diagram: pullback subdivision}\begin{tikzcd}
	X_0'^\dagger\arrow[r]\arrow[d]&X_0^\dagger\arrow[d]\\
	X'^\dagger\arrow[r]&X^\dagger
\end{tikzcd}\end{equation}
such that $\Sigma(X_0^\dagger)=\Sigma_0$ and $\Sigma(X_0'^\dagger)\to \Sigma(X'^{\dagger})$ is a proper subdivision. 
\end{lemma}
\begin{proof} By~\ref{defn fan} locally fs fans admit fibered products, so there is a cartesian square
$$	\begin{tikzcd}
	F_0'\coloneqq F(X'^\dagger)\times^{\mathrm{fs}}_{F(X^\dagger)}F_0\arrow[r]\arrow[d]&F_0\arrow[d]\\
	F(X'^\dagger)\arrow[r]&F(X^\dagger)
\end{tikzcd}$$
where $F_0=F(\Sigma_0)$ is the Kato fan associated to $\Sigma_0$. Let $\Sigma'_0$ denote the cone complex associated to $F_0'$ and let $\varphi:\Sigma(X'^{\dagger})\to \Sigma(X^{\dagger})$ denote the associated morphism of cone complexes.
For each cone $\sigma'$ of $\Sigma(X^{\dagger})$, the cone $\varphi(\sigma')$ is a finite union $\cup \sigma_i$ of cones $\sigma_i$ in $\Sigma_0$, because $\Sigma_0\to \Sigma(X^{\dagger})$ is a proper subdivision. By construction $\Sigma'_0$ is obtained by glueing the cones $\phi^{-1}\sigma_i$. So by definition (see \ref{def cone complex}) $\Sigma_0'\to\Sigma(X'^{\dagger})$ is a proper subdivision and so is $F_0'\to F'$. 


Now by~\ref{toroidal resolutions} we have log-regular schemes $X_0'^\dagger=X'^\dagger\times_{F}F_0'$ and $X_0^\dagger=X^\dagger\times_{F}F_0$ with fans $F_0'$ and $F_0$ respectively. By the universal property of~\ref{toroidal resolutions} we obtain a morphism $(X_0')^\dagger\to X^\dagger$ such that~\ref{diagram: pullback subdivision} is cartesian.
\end{proof}
\section{Balanced covers of dual complexes}
\label{sec:dual complex}
\begin{block}
In this section, we study dual complexes of toroidal $S$-schemes. We view these as $\Z$-PL complexes and and prove that finite covers are balanced (Theorem~\ref{thm: balanced}).
\end{block}
\begin{block}[Setup] 
\label{setup dual complex}
Recall the running assumptions of Conventions~\ref{conventions}. The scheme $S=\spec k^{\circ}$ is endowed with the standard log-structure $S^\dagger$ given by the closed point.\end{block}
\begin{definition}[Toroidal $S$-schemes]
	\label{def toroidal scheme}
We define a \emph{toroidal $S$-scheme} as a log-scheme $\XX^{\dagger}=(\XX,\M_{\XX})$ such that $\XX$ is an integral flat separated finite type $S$-scheme and $\XX^{\dagger}$ is log-regular.

The flatness assumption is in fact superfluous; indeed, an integral log-scheme $\XX^{\dagger}$ over $S^{\dagger}$ is log-regular at the generic point if and only if the log-structure is generically trivial, and any log-scheme $X^{\dagger}$ over $S^{\dagger}$ satisfies $D_X\subset \XX_s$, where $D_\XX$ denotes the support (see \ref{log parameters}).
\end{definition}
\begin{block}[Assumption]\label{assumption toroidal cover} Throughout, unless explicitly stated otherwise, $$f:\YY^\dagger\to\XX^\dagger$$ denotes a finite morphism of toroidal $S$-schemes, which we will also call a \emph{toroidal cover}. Note this implies $f^{-1}D_{\XX}\subset D_{\YY}$. \label{assume f finite cover toroidal scheme}
\end{block}
\begin{example} \label{examples f}To justify our choice of setting in \ref{assumption toroidal cover}, let us show that in nature, one can often encounter morphisms $f:\YY^\dagger\to\XX^\dagger$ as in~\ref{assume f finite cover toroidal scheme}; we give four situations below. Let $\XX^\dagger$ be a toroidal $S$-scheme.
\begin{enumerate}
	\item Suppose $\XX^{\dagger}\to S^{\dagger}$ is log-smooth, and $S'^{\dagger}\to S^{\dagger}$ is a finite morphism of log-schemes, then $\YY^\dagger\coloneqq\XX^\dagger\times^{\mathrm{fs}}S'^\dagger\to \XX^\dagger,$ is log-smooth, and hence $\YY^{\dagger}$ is log-regular. In particular the underlying scheme $\YY$ is normal and coincides with the normalisation of the base change $\XX_{S'}=\XX\times_SS'$.
	\item Let $S'^\dagger\to S^\dagger$ be a finite log-\'etale morphism of log-schemes. Then the \emph{fs} base change $f^{\dagger}:\YY^\dagger\coloneqq\XX^\dagger\times_{S^\dagger}^{\mathrm{fs}}S'^\dagger\to \XX^\dagger$ is also log-\'etale. Therefore, since $X^{\dagger}$ is log-regular, so is $\YY^\dagger$. Similarly the underlying scheme $\YY$ is the normalisation of $\XX_{S'}$. 
	\item Suppose $\YY^\dagger$ is a fine log-scheme equipped with a finite morphism $f:\YY^\dagger\to\XX^{\dagger}$ such that $f$ is \'etale away from $D_{\XX}$ and tamely ramified over the generic points of $D_{\XX}$. 
	Then $f$ is log-\'etale by the log-purity theorem of Kato-Mochizuki \cite[Theorem B]{Mo} (also see \cite[\S13.3.43]{GR} for a refinement), and hence $\YY^{\dagger}$ is log-regular. 
	
	\item Suppose $\XX$ is regular \emph{of dimension $2$} and write $U=\XX\setminus D_{\XX}$. Suppose $U'\to U$ is a finite \'etale cover which is \emph{cyclic of degree $p$}. Then by Abhyankar's theorem  \cite{abhyankar} on $p$-cyclic covers of Jungian domains -- see Kato's comments in \cite[\S12.4]{K94} and the references cited there -- it follows that there exists a modification $g:\XX_0\to\XX$ such that the following conditions (i)-(iii) hold:
	\begin{enumerate}[label=(\roman*)]
		\item Let $D_{\XX_0}=g^{-1}D_\XX$. Then the log-scheme $\XX_0^\dagger=\XX_0(\log\,D_{\XX_0})$ is log-regular.
		\item The modification $g$ is an isomorphism when restricted to $\XX_0\setminus D_{\XX_0}$.
		\item Let $f:\YY\to\XX_0$ be the normalisation of $\XX_0$ in the function field of $U'$. Then $\YY^\dagger=\YY(\log\,f^{-1}D_{\XX_0})$ is log-regular. 
	\end{enumerate} 
If (i)-(iii) hold then $f$ is a finite morphism of $2$-dimensional toroidal $S$-schemes (the morphism $f$ is finite because $\XX_0$ is excellent).
	
\end{enumerate}
\end{example}

\begin{remark}[Kato's question on simultaneous toroidal resolution]
We mention that Kato \cite[\S12.4.1]{K94} has asked the question whether the assertion of \ref{examples f} (4) remains true for \emph{any finite \'etale cover $U'\to U$ and in arbitrary dimension}. By the existence of toroidal resolutions (\ref{toroidal resolutions}) and Lemma \ref{lemma pullback subdivision}, we may additionally require $\XX_0$ to be regular (as Kato does) without changing the outcome of this question. As far as the author is aware Kato's question is still wide open. 
\end{remark}

\begin{block}[Fans] For the remainder of the paper $f:\XX^{\dagger}\to \YY^{\dagger}$ is as in Assumption~\ref{assume f finite cover toroidal scheme}, sometimes additional assumptions are imposed and then this is indicated in the beginning of a new section. 
The constructions of Section~\ref{subs: kato fan} yield a morphism of Kato fans $F(\YY^\dagger)\to F(\XX^\dagger)$. Since the morphism $\XX^\dagger\to S^\dagger$ provides structural maps $\N\to \CC_{F(\XX^\dagger),x}$, sending $1$ to the class of the uniformiser $[\varpi]\in\CC_{F(\XX^\dagger),x}$, we can and will view $F(\XX^\dagger)$ as a fan over $\spec\N$, and $F(\YY^\dagger)\to F(\XX^\dagger)$ as a morphism of fans over $\spec\N$. To make the notation somewhat less heavy, we will also write $$M_x\coloneqq \CC_{F(\XX^\dagger),x}^{\gp}.$$ By definition, $M_x$ is a lattice and $\mathrm{rank}(M_x)=\mathrm{rank}(x)$.
\end{block}
\begin{lemma} 
	\label{lem: inclusion lattices of pl-functions}
Let $x'\in F(\YY^\dagger)$ and write $x=f(x')$. Then $M_x\to M_{x'}$
is injective and $\mathrm{rank}(M_x)=\mathrm{rank}(M_{x'})$.
 \end{lemma} 
 \begin{proof}
 	The map $\CC_{F(\XX^\dagger),x}\to\CC_{F(\YY^\dagger),x'}$ is \emph{local} morphism of sharp monoids, which by sharpness is the same as an injective monoid morphism (\ref{local morphisms of monoids}).  
 	 By the assumption that $x'\in F(\YY^\dagger)$ and $f:\YY\to\XX$ is finite, we have a finite morphism of log-strata $$E_{x'}=\overline{\{x'\}}\to E_x=\overline{\{x\}}$$ of dimension $\dim \O_{\YY,x'}=\dim \O_{\XX,f(x)}$, which by log-regularity gives $\mathrm{rank}(x')=\mathrm{rank}(x)$. 
 \end{proof}


\begin{block}[The dual complex $\D(\XX^\dagger)$] \label{construction dual complex}

 The \emph{dual complex} $\D(\XX^\dagger)$ of $\XX^\dagger$ is defined as the $\Z$-PL subspace of $\Sigma(\XX^\dagger)$ defined by glueing the polyhedra $$\sigma_x\coloneqq\{a\in\sigma_{\CC_{F(\XX^\dagger),x}}:\alpha([\varpi])=1\}\subset C(x)$$ by restricting the face maps of $\Sigma(\XX^\dagger)$. 
 This defines functor from the category of log-regular log-schemes over $S^\dagger$ to $\Z$-PL complexes and thus functor factors through the category of locally \emph{fs} fans. Throughout we denote by $$\phi:\D(\YY^\dagger)\to\D(\XX^\dagger)$$ the induced morphism of $\Z$-PL complexes. 
The following are equivalent:
\begin{enumerate}[label=(\roman*)]
 	\item The polyhedron $\sigma_x$ is nonempty.
 	\item The element $[\varpi]\in \CC_{F(\XX^{\dagger}),x}$ is nonzero.
 	\item The point $x$ lies in the special fiber $\XX_s$.
 \end{enumerate} 
 Moreover, $\sigma_x$ is bounded if and only if the closed stratum $E_x$ associated to $x$ (see~\ref{log-strata}) is an intersection of \emph{vertical components} of $D_{\XX}$, i.e. components of $\XX_s$

\end{block}
\begin{remark} Informally speaking, the points of $\sigma_x$ are ``$k$-valuations'' on the monoid $\CC_{F(\XX^\dagger),x}$, in the sense that they define monoid morphisms $\CC_{F(\XX^{\dagger}),x}\to\R_{\ge0}$ which are compatible with the normalised valuation of $k$. 
\label{rmk monoid valuations}
\end{remark}

\begin{proposition}[Multiplicity formula]\label{proposition: multiplicity formula} Let $f:\YY^{\dagger}\to \XX^{\dagger}$ be a toroidal cover (\ref{assumption toroidal cover}. 
	Let $y'\in F(\YY^\dagger)$, write $y=f(y)$ and let $x\in \overline{\{y\}}$ be a specialisation of $y$. Then \begin{equation}
		\label{eqn:mu}
\left[M_{y'}:M_{y}\right]\cdot\left[\kappa(y'):\kappa(y)\right]=\sum_{x'\in \overline{\{y'\}}\cap f^{-1}(x)}\left[M_{{x'}}:M_{x}\right]\cdot\left[\kappa(x'):\kappa(x)\right].
	\end{equation}
\end{proposition}
\begin{proof} By induction, we can assume that $\dim_x\XX=\dim_{y}\XX+1$. We have induced morphisms of closed strata 
$$\begin{tikzcd}E_{x'}\arrow[r]\arrow[d]&\arrow[d]E_x\\E_{y'}\arrow[r]&E_{y}	
\end{tikzcd},$$
and we can view $x'$ and $x$ as codimension $1$ points of $E_{y'}$ and $E_y$ respectively. 
	By the usual multiplicity formula, see for instance \cite[\S7.1.38]{Liu}, it follows that \begin{equation}\label{eqn:ag}
		\left[\kappa(y'):\kappa(y)\right]=\sum_{x'\in \overline{\{y'\}}\cap f^{-1}(x)}e(\O_{E_{y'},x'}/\O_{E_{y},x})\cdot\left[\kappa(x'):\kappa(x)\right].
	\end{equation}

	By Lemma~\ref{lemma balanced cospec} (iv)-(v) we have a diagram with exact rows


$$
\begin{tikzcd}
	0 \arrow[r] &\CC_{E_y,x}^{\gp}\arrow[r]\arrow[d] &\CC_{\XX,x}^{\gp}\arrow[r]\arrow[d] & \CC_{\XX,y}^{\gp}\arrow[r]\arrow[d] & 0\\
		0 \arrow[r] &\CC_{E_{y'},x'}^{\gp}\arrow[r] &\CC_{\YY,x'}^{\gp}\arrow[r] & \CC_{\YY,y'}^{\gp}\arrow[r] & 0\\
\end{tikzcd}
$$

where $E_y$ and $E_{y'}$ are equipped with the log-structure obtained by restricting the log-structure of $\XX^\dagger$ and $\YY^\dagger$ respectively. By Lemma \ref{lem: inclusion lattices of pl-functions} and the snake lemma we obtain \begin{equation}
	\left[M_{{x'}}:M_{x}\right]=\left[M_{{y'}}:M_{{y}}\right]\cdot\left[\CC_{E_{y'},x'}^{\gp}:\CC_{E_y,x}^{\gp}\right].\label{eqn:l}
\end{equation}
The desired formula now follows from Lemma \ref{lemma balanced cospec} (ii), Equations \eqref{eqn:ag} and \eqref{eqn:l} and Lemma~\ref{lemma extension dvr balanced} below.
\end{proof}
\begin{lemma}
	\label{lemma extension dvr balanced} Let $R'/R$ be a finite extension of DVRs with ramification index $e(R'/R)$. Equip $S'=\spec R'$ and $S=\spec R$ with the standard log-structures $S'^\dagger$ and $S^\dagger$. Then $$[\CC^{\gp}_{S,s}:\CC_{S,s'}^{\gp}]=e(R'/R).$$
\end{lemma}
\begin{proof} Pick a uniformiser $\xi'$ of $R'$, then $\xi\coloneqq\xi'^{e(R'/R)}$ uniformises $R$, and $\CC_s$ (resp. $\CC_{s'}$) is generated by the class of $\xi$ (resp. $\xi'$).
\end{proof}
\begin{definition}[Finite covers of PL spaces]
\label{cover pl spaces}
	Let $\varphi:\Sigma'\to \Sigma$ be a morphism of PL spaces. 
	We call $\varphi$ a (branched) \emph{finite cover} if $\varphi$ is surjective with finite fibers and for each polyhedron $\sigma'$ of $\Sigma'$ we have $\dim\varphi(\sigma')=\dim\sigma'$. 
	\end{definition}
	\begin{block}[Lattice index] \label{lattice index} Suppose $\varphi:\Sigma'\to\Sigma$ is a finite cover of $\Z$-PL complexes. Let $\sigma'$ be a polyhedron of $\Sigma'$ and let $\sigma=\varphi(\sigma')$. Then by assumption that $\dim \sigma=\dim\sigma'$ it follows that $$N_{\sigma'}\to N_{\sigma}$$ is injective, where recall $N_{\sigma}$ is our notation for the lattice of integral points of $\R\sigma$ (as in \ref{def polyhedral complex}). The \emph{lattice index} of $\varphi$ at $\sigma'$ is defined as $$\left[N_{\sigma}:N_{\sigma'}\right].$$


A computation\footnote{For any injective group morphism of same-rank lattices $\psi:L'\to L$, the dual $\psi^*:L^*\to(L')^*$ is also injective and $\mathrm{coker}\psi^*\cong\mathrm{coker}\psi$ canonically by the long exact sequence $$0=\mathrm{coker}(\psi)^*\to L^*\to(L')^*\to\mathrm{Ext}^1(\mathrm{coker}(\psi),\Z)=\mathrm{coker}(\psi)\to \mathrm{Ext}^1(L,\Z)=0.$$ 
So it follows that $[L:L']=\#\mathrm{coker}\,\psi=\#\mathrm{coker}\,\psi^*=[L'^*:L^*]$.} shows that \begin{equation}\label{eq: dual lattice index}
\left[M_{\sigma'}:M_{\sigma}\right]=\left[N_{\sigma}:N_{\sigma'}\right].
\end{equation}
Informally speaking, the lattice index measures the change of $\Z$-PL functions along $\sigma'\mapsto \sigma$. 

Note that $\sigma$ is strictly convex if and only if $\sigma'$ is so. In that case $(\CC_{\sigma})^{\gp}=M_{\sigma}$ where $\CC_{\sigma}$ denotes the monoid of non-negative $\Z$-affine functions on $\sigma$ (see~\ref{lemma: props Z-pl function on polyhedra}). 
	\end{block}
	\begin{definition}[facet-ridge pairs $(\sigma,\tau)$]\label{defn facet-ridge}
	Let $\sigma,\tau$ be polyhedra of a PL-space $\Sigma$. Then we write $\sigma>\tau$ if and only if $\tau$ is a proper face of $\sigma$. We will also call $\sigma$ an \emph{adjacent} polyhedron of $\tau$.
Suppose $n=\dim\Sigma\ge 1$. A \emph{facet} is defined as a $n$-dimensional polyhedron and a \emph{ridge} is defined as a $(n-1)$-dimensional polyhedron. By a \emph{facet-ridge pair} $(\sigma,\tau)$ we mean a pair of a facet $\sigma\in \Sigma(n)$ and a {ridge} $\tau\in\Sigma(n-1)$ such that $\sigma>\tau$. 
\end{definition}

	\begin{definition}[Balanced cover of $\Z$-PL complexes]
		\label{balanced cover Z-pl spaces}
		Suppose $\varphi:\Sigma'\to\Sigma$ is a finite cover of $\Z$-PL complexes. 
	Let $\tau'$ be a ridge of $\Sigma'$. Then we call $\varphi$ \emph{balanced at $\tau'$} if the following condition holds. 
		For each facet $\sigma$ adjacent to $\tau$, the following expression is independent of the choice of $\sigma$:
		\begin{equation}\label{eqn: local degree}
\deg_{\tau'}\varphi\coloneqq \sum_{\substack{\text{facets }\, \sigma'\\(\sigma',\tau')\mapsto (\sigma,\tau)}}[M_{\sigma'}:M_{\sigma}],
		\end{equation}
		the sum is taken over all facets $\sigma'$ adjacent to $\tau'$ such that $\varphi:(\sigma',\tau')\mapsto(\sigma,\tau)$, see Figure~\ref{fig: cover pl spaces} for a picture. We call $\varphi$ a balanced finite cover of $\Z$-PL complexes if it is balanced at every ridge $\tau'$ of $\Sigma'$.
		
		If $\varphi$ is balanced at every ridge $\tau'$ of $\Sigma'$ and $\Sigma$ is connected then we define the \emph{degree} $\deg\phi$ of $\phi$ as follows. Choose any ridge $\tau$ of $\Sigma$ and set \begin{equation}
	\label{intro eqn: degree}
\deg\varphi\coloneqq\sum_{\tau'\mapsto \tau}\deg_{\tau'}\varphi,
\end{equation}
where we sum over ridges $\tau'$ above $\tau$. It follows from the definitions and the assumption that $\Sigma$ is connected that $\deg\phi$ is independent of the choice of $\tau$.
\end{definition} 
\begin{block} The balancing condition of \ref{balanced cover Z-pl spaces} behaves well with respect to proper subdivisions in the following sense. If $\Sigma'\to \Sigma$ is a balanced finite cover of $\Z$-PL complexes and $\Sigma_0\to \Sigma$ is a proper subdivision, then $\Sigma'_0=\Sigma'\times_{\Sigma}\Sigma_0\to \Sigma_0$ is a balanced finite cover of $\Z$-PL complexes and $\Sigma'_0\to\Sigma'$ is a proper subdivision.
\end{block}

\begin{theorem}[Balancing condition] \label{thm: balanced} Let $f:\YY^{\dagger}\to\XX^{\dagger}$ be a toroidal cover. Then $\phi:\D(\YY^\dagger)\to\D(\XX^\dagger)$ is a balanced finite cover of $\Z$-PL complexes and $\deg\phi=\deg f$.
\end{theorem}
\begin{proof} The map $\phi$ is a finite cover of $\Z$-PL complexes by Lemma \ref{lem: inclusion lattices of pl-functions}.  
Let $(\sigma,\tau)$ be a facet-ridge pair of $\D(\XX^\dagger)$, then $(\sigma,\tau)=(\sigma_x,\sigma_y)$ for points $x,y\in F(\XX^\dagger)$ with $x\in\overline{\{y\}}$. By Proposition~\ref{proposition: multiplicity formula} it follows that for any ridge $\tau'$ above $\tau$ we have $\deg_{\tau'}\phi=[L_{\tau'}:L_{\tau}][\kappa(\tau'):\kappa(\tau)]$, hence $\phi$ is balanced. 

It remains to prove that $\deg\phi=\deg f$. For any vertex $v\in \D(\XX^\dagger)(0)$, Lemma~\ref{lemma extension dvr balanced} and equation \eqref{eqn:ag} (applied to the generic point of $\YY$) yield the formula $$\deg f=\sum_{v'\in \phi^{-1}(v)}[M_{v'}:M_v]\cdot[\kappa(v'):\kappa(v)].$$ Proposition~\ref{proposition: multiplicity formula} then shows that the right-hand side equals $\deg \phi$.
\end{proof}
\begin{remark} The assumption that $\tilde{k}$ is algebraically closed is not essential in the proof of  Theorem~\ref{thm: balanced}. If we omit this hypothesis then one additionally decorates each facet $\sigma=\sigma_x$ of $\D(\XX^\dagger)$ with the degree $\deg_{\widetilde{k}}{\sigma}\coloneqq\left[\kappa(x):\widetilde{k}\right]$ and the balancing condition is changed to $\deg_{\tau'}\phi=\sum \deg_{\sigma'}\phi$ where $\deg_{\sigma'}\phi\coloneqq[M_{\sigma'}:M_{\sigma}]\deg_{\widetilde{k}}{\sigma'}/\deg_{\widetilde{k}}{\sigma}$, in accordance with Proposition~\ref{proposition: multiplicity formula}. 
\end{remark}

\section{Metrisation of $\Z$-PL complexes}
\label{sec: metrisation of pl}
\begin{block}[Overview]
	In this section we show that any $\Z$-PL complex $\Sigma$ is endowed with a canonical piecewise Lebesgue measure $\mu_{\Sigma}$, that we call the \emph{conformal} measure, such that isometries are isomorphisms. We give explicit descriptions of $\mu_{\Sigma}$ in the case where $\Sigma$ is a dual complex. 

\end{block}
\begin{block}[First attempt] Let $\sigma$ be a polyhedron; recall that by convention $\sigma$ carries a $\Z$-affine structure: $\sigma=(|\sigma|,N_{\sigma})$ where $N_{\sigma}$ is the lattice of $\Z$-affine functions.
	Let $H_0$ be the real vector space spanned by the differences of elements of $\sigma$. Then $H_0$ contains the maximal lattice $N_{\sigma}\cap H_0$ and therefore $\sigma$ inherits an integral structure  by translation .
For many purposes this is already a good procedure to equip $\sigma$ with a Lebesgue measure which only depends on the $\Z$-affine structure, it is for instance used in the works \cite{CD,JN}. However for our purposes it is more natural to work with a slightly different construction that can distinguish isomorphism types of polyhedra: see Example~\ref{example zpl} below.
\end{block}
\begin{lemma} Suppose $A\in \Z^{n\times m}$.
	\label{integral linear algebra}
\begin{enumerate}[label=(\roman*)]
	\item $\{x\in\Z^n:Ax=0\}$ is a full-rank lattice in $\{x\in\R^n:Ax=0\}$. 
	\item Let $b\in \Q^n\setminus \{0\}$ and suppose $\{x\in \R^m:Ax=b\}$ is nonempty. Then there exists a minimal positive rational number $\rho\in\Q_{>0}$ such that $Ax=\rho b$ admits an integral solution $x\in \Z^n$.  
\end{enumerate}
\end{lemma}
\begin{proof} 
(i) By Gaussian elimination we can assume $A$ is upper triangular and so there exists a basis of rational solutions of $Ax=0$. 
(ii) By Gaussian elimination the system $Ax=b$ admits a rational solution $x\in \Q^m$, so by clearing the denominators of $x$ there exists at least one such $\rho\in \Q_{>0}$. Since $\rho b\in \Z^n$ the set of such $\rho$ is well-ordered, so there exists a minimum. 
\end{proof}
\begin{construction}[Conformal measure] \label{construction: z-pl measure} Let $\sigma=(|\sigma|,N_{\sigma})$ be a polyhedron and suppose $0\not\in\sigma$. let $$H_\sigma=\{\lambda(x-y):x,y\in|\sigma|,\lambda\in\R\}$$ denote the $\Z$-affine plane in $N_{\R}$ spanned by $|\sigma|$. By Lemma~\ref{integral linear algebra} (ii) there exists a minimal positive rational number $\rho=\rho_{\sigma}\in \Q_{>0}$ such that $$H_{\rho\sigma}\cap N\ne\emptyset,$$ where we write $\rho\sigma=\{\rho x: x\in\sigma\}$ for the $\rho$-scaling of $\sigma$. Then $H_{\rho\sigma}\cap N$ is a full-rank lattice in $H_{\rho\sigma}$ by Lemma~\ref{integral linear algebra} (i). 
Let $\mu_{H_{\rho\sigma}}$ be the unique Lebesgue measure on $H_{\rho\sigma}$ such that $N\cap H_{\rho\sigma}$ has covolume $1$. We write $\mu_{\rho\sigma}$ for the restriction to $\rho\sigma$. Then we define the \emph{conformal measure} on $\sigma$ as \begin{equation}
	\mu_\sigma\coloneqq \frac{\mu_{\rho\sigma}}{\rho^{\dim\sigma+1}};\label{def zpl measure}
\end{equation}

If $0\in\sigma$ we define $\mu_{\sigma}$ as the restriction of the unique Lebesgue measure on $(N_{\sigma})_{\R}$ such that $N_{\sigma}$ has covolume $1$.
 The name ``conformal measure'' is justified by Lemma \ref{lemma z-pl mu on polyhedron} below.
\end{construction}
\begin{example}\label{example zpl}
	For instance let $I\subset \R^n$ be the line segment given by the equations $2x+2y=1$ and  $x,y\ge 0$. Then $I$ is $\Z$-PL and the length of $I$ with respect to $\mu_I$ is $1/4$.

	For another example, the line segment $J\subset\R^n$ given by $x+2y=1$ and $x,y\ge 0$ has length $1/2$ with respect to $\mu_J$. The map $\phi:(x,y)\mapsto (2x,y)$ sends $I$ to $J$, and rescales the length by a dilation factor $2$, which is precisely the determinant of $\phi$. This will be generalised in Lemma~\ref{lemma z-pl mu on polyhedron} below and motivates our choice of the exponent $\dim\sigma+1$ in Equation \eqref{def zpl measure}. Indeed, if instead we had chosen the perhaps more natural-looking exponent $\dim\sigma$, then in the above we obtain something that we do not want: $I\to J$ would be an isometry, but $\phi$ is not unimodular (i.e. of determinant 1), so not a $\Z$-PL isomorphism. 
\end{example}
\begin{lemma} \label{lemma z-pl mu on polyhedron}Let $\phi:\sigma'\to\sigma$ be a morphism of polyhedra (in the sense of Appendix~\ref{def polyhedral complex}) with $\dim\sigma=\dim\sigma'$, then we have that $$\phi_*\mu_{\sigma'}=\frac{1}{[M_{\sigma'}:M_{\sigma}]}\mu_{\sigma}.$$
In particular $\phi$ is a $\Z$-PL isomorphism (i.e. a unimodular bijection) if and only if it is an isometry with respect to the conformal measure, and it is strictly expanding otherwise.
\end{lemma}
\begin{proof} We can assume $0\not\in\sigma$. 
Let $\rho$ (resp. $\rho'$) be the minimal $\rho\in \Q_{>0}$ such that $H_{\rho\sigma}\cap N_{\sigma}\ne\emptyset$ (resp. $H_{\rho'\sigma'}\cap N_{\sigma'}\ne\emptyset$). 
 Then we have $$\left(\frac{\rho'}{\rho}\right)^{\dim\sigma+1}\mu_{H_{\rho\sigma}}=\frac{\rho'}{\rho}\mu_{H_{\rho'\sigma}}\overset{\mathrm{def}}{=}\frac{\rho'}{\rho}[N\cap H_{\rho'\sigma}:N'\cap H_{\rho'\sigma'}]\phi_*\mu_{H_{\rho'\sigma'}}=[N:N']\phi_*\mu_{H_{\rho'\sigma'}}.$$ The result now follows by rescaling both sides and using Equation \eqref{eq: dual lattice index}.
\end{proof}
\begin{block}[The conformal measures $\mu_{\Sigma}$] 
Let $\Sigma$ be a $\Z$-PL space. By Lemma~\ref{lemma z-pl mu on polyhedron} the construction of~\ref{construction: z-pl measure} is compatible with isomorphisms and in particular proper subdivisions. Therefore we can glue the measures $\mu_{\sigma}$ for facets $\sigma$ of $\Sigma$ to the \emph{conformal measure} $\mu_{\Sigma}$ on $\Sigma$.
\end{block}

\label{sec: conformal}
\begin{block}
For the remainder of this section we explicitise the construction of $\mu_{\Sigma}$ in the case where $\Sigma=D(\XX^\dagger)$ is a dual complex of a toroidal $S$-scheme $\XX^{\dagger}$. 
In summary, for any $x\in F(\XX^{\dagger})$ the conformal measure $\mu_x$ equals $\frac{1}{\rho_x}\lambda_x$ where $\lambda_x$ is the integral Lebesgue measure of Jonsson-Nicaise \cite[\S5.1]{JN} and the scaling constant $\rho_x$ is the root index (see \ref{defn root index}).
\end{block}
\begin{block}[Notation] 
 \label{notation dual complex} Recall the notation from Section \ref{sec:dual complex}: for $x\in F(\XX^\dagger)$, we write $\CC_{x}\coloneqq \CC_{F(\XX^\dagger),x}=\M_{\XX,x}^\#$ for the characteristic monoid, $M_x=\CC_x^{\gp}$ for the characteristic lattice  
 $$N_x\coloneqq M_x^*=(\M_{\XX,x}^\vee)^{\gp}$$ for the dual lattice, $r(x)\coloneqq \mathrm{rank}(x)=\mathrm{rank}_{\Z}M_x$ for the rank, $$C(x)\coloneqq \sigma_{\CC_{x}}=\R_{\ge0}\CC_{x}^\vee$$ for the dual cone 
 and $$\sigma_{x}=\{\alpha\in C(x):\alpha([\varpi])=1\}$$ for the polyhedron of $\D(\XX^\dagger)$ associated to $x$. Note $\dim\sigma_x=r(x)-1$ and $M_{\sigma_x}=M_x$. For any $\rho\in\R$ we write $H_{\rho,x}$ for the affine hyperplane in $(N_x)_{\R}$ defined by $[\varpi]\mapsto \rho$, so by definition $\sigma_x=\sigma_{\CC_{x}}\cap H_{1,x}$. We write $N_{\rho,x}=N_x\cap H_{\rho,x}$. See Figures~\ref{fig:slice} and~\ref{fig:slice2} for pictures. We have the following properties:
 \begin{enumerate}[label=(\roman*)]
  	\item The hyperplane $H_{0,x}$ is the orthogonal complement of $[\varpi]$ under the natural pairing $$\langle\cdot,\cdot\rangle: \left(M_x\right)_\R\times(N_x)_{\R}\to \R.$$
  	\item The hyperplanes $H_{\rho,x}$ are parallel translates for varying $\rho\in \R$.
  	\item The hyperplane $H_{\rho,x}$ is $\Z$-affine with respect to $N_x$ if and only if $\rho\in\Z$. 
\end{enumerate}
  	\end{block}
\begin{block}[Root index] \label{defn root index}The \emph{root index} $\rho_x$ at $x$ is defined as the largest positive integer $\rho$ such that the image $[\varpi]$ of $\varpi$ in $\CC_{x}$ (equivalently, in $M_x$) is divisible by $\rho$. It was introduced in \cite[\S2]{BN20}.

A computation using B\'ezout-Bachet shows there is a short exact sequence
\begin{equation}
	0\lra N_{x,0}\lra N_x\overset{\langle[\varpi],\cdot\rangle}{\lra}\rho_x\Z\lra 0\label{eq: bezout}
\end{equation}
So $N_{\rho,x}\ne \emptyset$ if and only if $\rho\in \rho_x\Z$.

If $x,y\in F(X^{\dagger})$ and $y$ generalises $x$ then we have $\rho_x\mid \rho_y$ because $\CC_{x}\to\CC_{y}$ preserves $[\varpi]$. 
\label{props root index}
\end{block}
\begin{block}[Recession cone]
We denote by $\rec(\sigma_x)$ the \emph{recession cone} of $\sigma_x$, i.e. the cone of vectors $v\in (N_x)_\R$ for which $\sigma_x+\R_{\ge0}v\subset \sigma_x$ (also see \ref{block recession fan})
The following properties are immediate to verify.
\begin{enumerate}[label=(\roman*)]
	\item The cone $\rec(\sigma_x)$ is a rational polyhedral cone with associated lattice $N_{0,x}\cap |\rec(\sigma_x)|_{\R}$. 
	\item The cone $\rec(\sigma_x)$ is trivial if and only if $\sigma_x$ is bounded.
	\item The rays of $\rec(\sigma_x)$ are in one-to-one-correspondence with the rays of $C(x)$ that lie in $H_{0,x}$.
\end{enumerate}
\end{block}
\begin{block}[Rays of $C(x)$]  \label{rays discussion}
Let $x\in F(X^{\dagger})$. Let $r\in C(x)(1)$ be a ray of $C(x)$. Then $r$ corresponds to a unique component $E_r$ of $D_{\XX}$ containing $E_x$, by Lemma~\ref{lemma balanced cospec} (i). We have two cases: 
\begin{enumerate}[label=(\roman*)]
	\item Either $r$ does not lie in $H_{0,x}$, in which case $E_r$ is vertical (i.e. a component of $\XX_s$), and $r$ defines a ray in $\rec(\sigma_x)$
	\item Or $r$ defines a unique vertex $v_r\coloneqq H_{0,x}\cap r\in \sigma_x(0)$ and $E_r$ is a horizontal component (i.e. not a vertical component, equivalently $E_r$ is flat over $k^{\circ}$). We also write $E_{v_r}=E_r$. 
     The \emph{multiplicity} $m_{v_r}\coloneqq \mathrm{mult}_{E_r}\XX_s$ of $E_r$ equals $\langle [\varpi],n_r\rangle$, where $n_r$ is the primitive ray generator of $r$, because $\CC_{E_{r}}\cong \N$ is generated by the class of a local equation of $E_r$.
\end{enumerate}\end{block}
\begin{block}
We picture the situation in Figure \ref{fig:slice} and Figure \ref{fig:slice2} below, where we write $s(x)\coloneqq\#C(x)(1)$ for the number of rays of the cone $C(x)$. Note $s(x)\ge r(x)$ with equality if and only if $C(x)$ is simplicial. Let $n_1,\dots,n_{s(x)}\in C(x)(1)$ denote the primitive ray generators of $C(x)$ and let $b(x)=\#\sigma_x(0)$ denote the number of vertices of $\sigma_x$. After reordering we can assume $n_1,\dots,n_{b(x)} \not\in H_{0,x}$ (corresponding to the vertices of $\sigma_x$) and $n_{b(x)+1},\dots,n_{s(x)}\in H_{0,x}$ (corresponding to the rays of $\rec(\sigma_x)$). 
 \end{block}

\refstepcounter{equation}
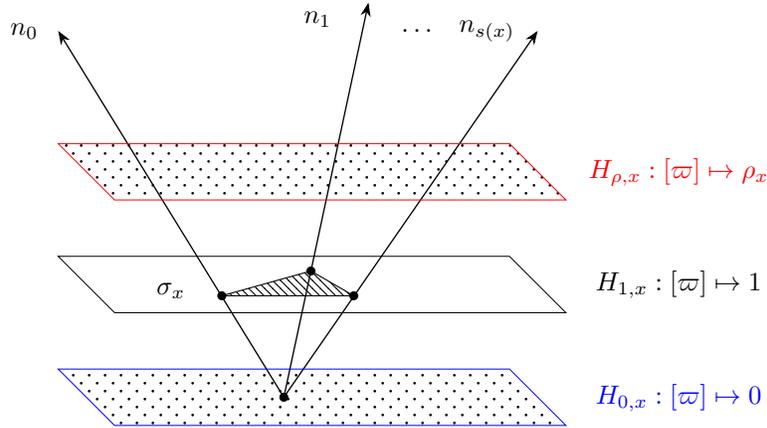
\begin{figure}[!ht]
\centering
\begin{tikzpicture}[scale=1.5]
 \usetikzlibrary{patterns}

\pgfdeclarepatternformonly{my crosshatch dots}{\pgfqpoint{-1pt}{-1pt}}{\pgfqpoint{5pt}{5pt}}{\pgfqpoint{6pt}{6pt}}%
{
    \pgfpathcircle{\pgfqpoint{0pt}{0pt}}{.5pt}
    \pgfpathcircle{\pgfqpoint{3pt}{3pt}}{.5pt}
    \pgfusepath{fill}
}
\draw [line width=0.5pt, ->, >=Stealth] (7,9.25) -- (7.75,12.75);
\draw [line width=0.5pt, ->, >=Stealth] (7,9.25) -- (5,12.5);
\draw [line width=0.5pt, ->, >=Stealth] (7,9.25) -- (9.25,12.5);

\node at (7.3,12.6) {$n_1$};
\node at (4.7,12.5) {$n_0$};
\node at (8.2,12.5) {$\dots$};
\node at (8.8,12.5) {$n_{s(x)}$};
\draw [ line width=0.2pt ] (5,10.5) -- (9,10.5) -- (9.5,10) -- (5.5,10) -- cycle;
\draw [color=blue ,pattern=my crosshatch dots,  line width=0.2pt ] (5,9.5) -- (9,9.5) -- (9.5,9) -- (5.5,9) -- cycle;
\draw [color=red, pattern=my crosshatch dots, line width=0.2pt]  (5,11.5) -- (9,11.5) -- (9.5,11) -- (5.5,11) -- cycle;

\draw[fill] (6.45,10.15) circle (1pt);
\draw[fill] (7.24,10.37) circle (1pt);
\draw[fill] (7.62,10.15) circle (1pt);
\draw [pattern=north west lines] (6.45,10.15) -- (7.24,10.37) -- (7.62,10.15) -- cycle;
\node at  (6,10.2) {$\sigma_x$};
\draw[fill] (7,9.25) circle (1pt);
\node [color=blue] at (10.5,9.25)  {$H_{0,x}:[\varpi]\mapsto 0$};

\node 				at (10.5,10.25)  {$H_{1,x}:[\varpi]\mapsto 1$};
\node [color=red] at (10.5,11.25)  {$H_{\rho,x}:[\varpi]\mapsto \rho_x$};

\end{tikzpicture}
\label{fig:slice}
\caption{Picture of a \emph{bounded} face $\sigma_x$ as the intersection of the dual cone $C(x)$ and the plane $H_{1,x}$.}
\end{figure}
\refstepcounter{equation}
\begin{figure}[!ht]
\centering
\begin{tikzpicture}[scale=1.5]
 \usetikzlibrary{patterns}
\draw [line width=0.5pt, ->, >=Stealth] (7,9.25) -- (7.75,12.75);
\draw [line width=0.5pt, ->, >=Stealth] (7,9.25) -- (5,12.5);
\draw [line width=0.5pt, ->, >=Stealth] (7,9.25) -- (9.6,9.25);

\node at (7.3,12.6) {$n_{b(x)}$};
\node at (4.7,12.5) {$n_1$};
\node at (6.2,12.5) {$\dots$};
\node at (9,9.7) {$\ddots$};
\node at (9.4,9.55) {$n_{s(x)}$};
\draw [ line width=0.2pt ] (5,10.5) -- (9,10.5) -- (9.5,10) -- (5.5,10) -- cycle;
\draw [color=blue ,pattern=my crosshatch dots,  line width=0.2pt ] (5,9.5) -- (9,9.5) -- (9.5,9) -- (5.5,9) -- cycle;
\draw [color=red, pattern=my crosshatch dots, line width=0.2pt]  (5,11.5) -- (9,11.5) -- (9.5,11) -- (5.5,11) -- cycle;

\draw[fill] (6.45,10.15) circle (1pt);
\draw[fill] (7.24,10.37) circle (1pt);
\draw [pattern=north west lines] (6.45,10.15) -- (7.24,10.37) -- (9.12,10.37) -- (9.35,10.15) -- cycle;
\node at  (6,10.2) {$\sigma_x$};
\draw[fill] (7,9.25) circle (1pt);
\node [color=blue] at (10.5,9.25)  {$H_{0,x}:[\varpi]\mapsto 0$};

\node 				at (10.5,10.25)  {$H_{1,x}:[\varpi]\mapsto 1$};
\node [color=red] at (10.5,11.25)  {$H_{\rho_x,x}:[\varpi]\mapsto \rho_x$};

\end{tikzpicture}
\label{fig:slice2}
\caption{Picture of an \emph{unbounded face} $\sigma_x$ as the intersection of the dual cone $C(x)$ and the plane $H_{1,x}$. The number of vertices of $\sigma_x$ is $b(x)$, for some $1\le b(x)<s(x)$.}
\end{figure}
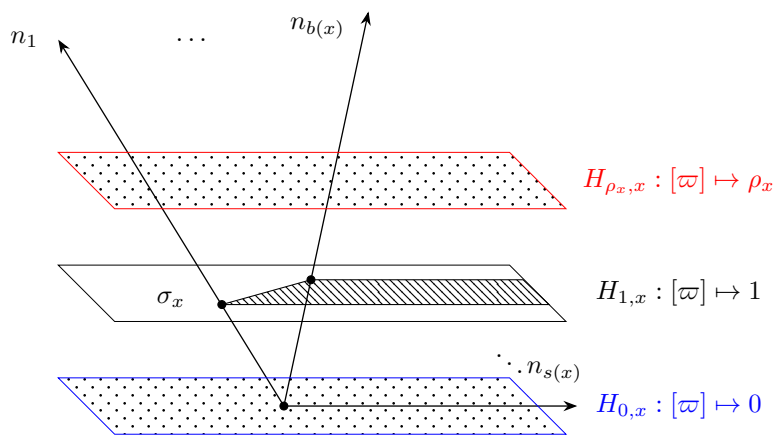

\begin{example} It follows from the definitions that $\rho_x$ divides $\gcd(m_1,\dots,m_{b(x)})$. If $\XX$ is moreover regular, then $\sigma_{\CC_{x}}$ is a regular cone and therefore $\rho_x=\gcd(m_1,\dots,m_{b(x)})$. In general, $\rho_x$ can be a strict divisor of $\gcd(m_1,\dots,m_{b(x)})$ as the following example\footnote{The example given in \cite[\S4.3.7]{JN} is unfortunately not entirely correct: the multiplicites are equal to 1 instead of $2$ and the monoid is generated by $(1,0), (1,1)$ and $(1,2)$. So in the example of loc. cit. $\rho_x=\gcd(m_1,\dots,m_{b(x)})$ contrary to what is claimed.} shows. Suppose $\XX^\dagger$ contains a subscheme isomorphic to $(\spec R)^{\dagger}$ where $$R=k^{\circ}[s,t,u,v]/(s\varpi=t^2,v\varpi=u^2,ut=sv=\varpi^2)$$ and the log-structure is the standard one induced by $\varpi$. A computation shows that $\spec R$ is of dimension $2$ and the special fiber consists of $2$ components given by $(s,t,\varpi)$ and $(u,v,\varpi)$ and these intersect in the singular point $x=(s,t,u,v,\varpi)$. The monoid $\CC_x$ at $x$ is isomorphic to the monoid generated by $[s]=(1,2),[t]=(1,1),[\varpi]=(1,0),[u]=(1,-1),[v]=(1,-2)$ in $\Z^2$, in particular $\XX^\dagger$ is log-regular at $x$. The dual cone $C(x)$ is isomorphic to the cone in $\R^2$ generated by $[s]^\vee=(2,-1)$ and $[v]^\vee=(2,1)$. The affine line $H_{1,x}$ is given by $y=1$. It follows that $\rho_x=1$ and $m_1=(2,-1)\cdot(1,0)=2,m_2=(2,1)\cdot(1,0)=2$, so that $\rho_x=1\ne 2=\gcd(m_1,m_2)$.
\end{example}
\begin{example} In general $C(x)$ need not be simplicial: consider for instance the scheme $\XX=\spec k^{\circ}[t,u,v,w]/(\varpi=tu=vw)$ equipped with the standard log-structure, then a computation shows that for the point $x=(t,u,v,w)$ we have $s(x)=4>r(x)=3$ and $\sigma_x$ is a rectangle.
\end{example}
\begin{definition}[determinant of $\sigma_x$] Suppose $C(x)$ is simplicial. Then we define the \emph{determinant} of $\sigma_{x}$ as the integer $\det(\sigma_x)=|\det(n_1,\dots,n_{r(x)})|$. 
In the literature on toric geometry \cite{fulton,kkmsd,oda} it is also known as the \emph{index} or \emph{multiplicity} of $C(x)$.  \label{det sigma}
\end{definition}
\begin{definition}[$\sigma_x^-$]\label{def -} For every component $E_i$ of $D_{\XX}$, choose a generator $\gamma_i$ of $\CC_{\XX,E_i}\cong \N$ (i.e. the class of a local equation of $E_i$). Then we let $$\sigma_x^-=\{\alpha\in \sigma_x:\langle \gamma_i,\alpha\rangle \le 1\}.$$ In other words $$\sigma_x^{-}=\left\{\sum_{1\le i \le b(x)}a_in_i+\sum_{i<b(x)\le s(x)}b_in_i\ :\ a_i\in\R_{\ge0},\,\sum_{1\le i\le b(x)}a_i=1,\ b_i\in [0,1]\right\}.$$
\label{defn sigma -}
By construction, $\sigma_x^{-}$ is a polytope and we have $\sigma_x^{-}\subset \sigma_x$ with equality if and only if $\sigma_x$ is bounded. Informally speaking $\sigma_x^{-}$ is obtained by ``cutting'' $\sigma_x$ with the unique hyperplane parallel to $v_1,\dots,v_{b(x)}$ which passes through $n_{b(x)+1},\dots,n_{s(x)}$
\end{definition}

\begin{proposition}\label{prop computation normalised volume via determinant} Suppose $C(x)$ is simplicial. Then we have \begin{equation}
  	\label{eqn: volume det}
  \mathrm{vol}_{\mu_x}\sigma_x^-=\frac{\det(\sigma_x)}{m_1\dots m_{b(x)}}.
  \end{equation}
\end{proposition}
\begin{proof} Let $\rho=\rho_x$ be the root index and equip $H_{\rho,x}$ with the Lebesgue measure $\lambda_{\rho,x}$ induced by $N_{\rho,x}=N_x\cap H_{\rho,x}$. The vertices of the polyhedron $$\sigma_{\rho,x}\coloneqq C(x)\cap H_{\rho_x,x}$$ coincide with the vectors $n_i'\coloneqq \rho n_i/m_i\in N_{\rho,x}$ for $i=1,\dots,b(x)$. The polytope $\sigma_{\rho,x}^-$ is obtained by scaling $\sigma_x^-$ along the cone $C(x)$ with scaling factor $\rho$. Hence we compute that $$\vol_{\mu_{x}}\sigma_x^-={\rho^{-\dim\sigma_x-1}}\vol_{\lambda_{\rho,x}}(\sigma_{\rho,x}^-)={\rho^{-r(x)}}|\det(n_1',\dots,n_{b(x)}',\rho n_{b(x)+1},\dots,\rho n_{r(x)})|=\frac{|\det(n_1,\dots,n_{r(x)})|}{ m_1\dots m_{b(x)}}.$$
\end{proof}
\begin{example}
 \label{example regular vol}		In case $\XX^{\dagger}$ is regular, each of the cones $C(x)$ is regular, so Proposition~\ref{prop computation normalised volume via determinant} reads $$\vol(\sigma_x^-)=\left(m_1\dots m_{b(x)}\right)^{-1}.$$

		In the case of relative curves ($\dim\XX=2$) it follows that the conformal measure induces the potential metric introduced in \cite[\S2]{BN}. 
\end{example}
\section{Specialisation of Cartier divisors}
\label{sec: harmonicity}

\begin{block}[Overview]
In this section we study Cartier divisors $\DD$ on $\XX$ that are supported on $D_{\XX}$, these can be studied via the PL functions on $\D(\XX^{\dagger})$ they induce. For later use we compute some intersection numbers (Proposition \ref{prop: computation ev.et}) and we introduce the specialisation $\mathrm{sp}_*\DD$. The remainder of the section is devoted to an analytic view on specialisations, by means of model metrics. This leads to a natural compactification $\overline{\D(\XX^{\dagger})}$ of the dual complex, which we first introduce in an ad hoc fashion.
\end{block}

\begin{block}[Assumption] Throughout this section fix a \emph{proper} toroidal $S$-scheme $\XX^{\dagger}$.
\end{block}
\begin{block}[Support functions] \label{M Sigma}
Recall that for each $x\in F(\XX^{\dagger})$, the lattice $M_{x}=M_{\sigma_x}=\CC_{\XX,x}^{\gp}$ is naturally identified (see~\ref{regular log-regular}) with the lattice of Cartier divisors on $\spec \O_{\XX,x}$ supported on $D_{\XX}$. On the other hand $M_{\sigma_x}$ is the lattice of $\Z$-affine functions on the dual cone $C(x)$. As in~\ref{def PL space} let us write $$M_{\Sigma(\XX^{\dagger})}=\lim_{x\in F(\XX^{\dagger})}M_{\sigma_x},$$ this group coincides with the group of $\Z$-PL functions $\Sigma(\XX^{\dagger})$ that are linear on the cones of $\Sigma(\XX^{\dagger})$. For simplicity elements of $M_{\Sigma(\XX^{\dagger})}$ are also called \emph{$\Z$-linear functions} on $\Sigma(\XX^{\dagger})$. If $\DD$ is a Cartier divisor supported on $D_{\XX}$, then the associated $\Z$-linearfunction is denoted by $F_{\DD}$ and is called the \emph{support function} of $\DD$. More explicitly it is described as follows: locally near a stratum $x\in F(\XX^{\dagger})$ there exists a $m\in M_{x}$ such that $\DD=\div m$, and so we have $$F_\DD\vert_{\sigma_x}=\langle [m],\cdot\rangle$$ on $\sigma_{\CC_{x}}$. Conversely we can reconstruct $\DD$ near $x$ from the image of $[m]$ in $M_x$ because a choice of section $M_x=\CC_x^{\gp}\to \M_{\XX,x}^{\gp}$ is canonical up to a unit in $\O_{\XX,x}^{\times}$ and therefore produces the same divisor $\div m$.
\end{block}
\begin{lemma}[Cartier divisors and support functions] \label{lem correspondence cartier and pl functions}
		The map $\DD\mapsto F_\DD$ establishes a group isomorphism between Cartier divisors supported on $D_{\XX}$ (respectively $\Q$-Cartier divisors supported on $D_{\XX}$) and the group $M_{\Sigma(\XX^{\dagger})}$ (respectively $(M_{\Sigma(\XX^{\dagger})})_{\Q}$).

		Moreover, effective Cartier divisors (respectively effective $\Q$-cartier divisors) correspond to elements of $M_{\Sigma(\XX^{\dagger})}^+\coloneqq \lim_{x\in F(\XX^{\dagger})}\CC_{\XX,x}$ (respectively $(M_{\Sigma(\XX^{\dagger})}^+)_{\Q_{>0}}=\lim_{x\in F(\XX^{\dagger})}\CC_{\XX,x}^{\Q_{>0}}$).
\end{lemma}
\begin{proof}
		This follows from the discussion in~\ref{M Sigma}.
\end{proof}
\begin{remark} See \cite[\S11]{K94} for a similar correspondence as in Lemma \ref{lem correspondence cartier and pl functions} via coherent fractional ideals on fans.
\end{remark}
\begin{block}[Notation] \label{qlaql}By definition $\D(\XX^{\dagger})$ and $\Sigma(\XX^{\dagger})$ have the same sets of $\Z$-PL functions and so we let $M_{\D(\XX^{\dagger})}=M_{\Sigma(\XX^{\dagger})}$, this is the group of functions on $\D(\XX^{\dagger})$ that are $\Z$-affine on polyhedra. 

Given an element $F\in M_{\D(\XX^{\dagger})}$ we write $\DD_F$ for the associated Cartier divisor as in Lemma \ref{lem correspondence cartier and pl functions}. For all $F,F'\in M_{\D(\XX^{\dagger})}$ we have $\DD_{F+F'}=\DD_{F}+\DD_{F'}$ by Lemma \ref{lem correspondence cartier and pl functions}.
\end{block}

\begin{lemma} 
		\label{lem computation weil divisor}
Suppose $F\in M_{\D(\XX^{\dagger})}$ and let $\DD_F$ as in \ref{qlaql}.
Then we have an equality of Weil divisors 

$$\mathrm{cyc}\,\DD_F=\sum_{r\in \Sigma(\XX^{\dagger})(1)} F_{\DD}(n_r) [E_r]=\sum_{v\in \D(\XX^{\dagger})(0)} m_vF_{\DD}(v)[E_v]+\sum_{r\in \mathrm{rec}(\D(\XX^{\dagger}))(1)} \partial_rF[E_r],$$

 where we write:
 \begin{itemize}
 \item $\mathrm{cyc}\DD_F$ for the Weil divisor associated to $\DD_F$,
  	\item  $E_r$ for the unique component of $D_{\XX}$ associated to a ray $r\in \Sigma(\XX^{\dagger})(1)$ (see \ref{rays discussion}) 
  	\item $n_r\in N_r\cap r$ for the primitive ray generator of a ray $r\in \Sigma(\XX^{\dagger})(1)$,
  	\item $E_v$ for the unique component associated to a vertex $v\in \D(\XX^{\dagger})(0)$,
  	\item $m_v\coloneqq \mult_{E_v}\XX_s$ for the \emph{multiplicity} of a vertex $v\in \D(\XX^{\dagger})(0)$,
  	\item $\partial_rF\coloneqq F(n_r)$ for any ray ${r\in \mathrm{rec}(\D(\XX^{\dagger}))(1)}$ of the recession cone complex. 
  \end{itemize} 
In addition the same assertion holds for $(M_{\D(\XX^\dagger)})_{\Q}$ and $\Q$-Weil divisors.
\end{lemma}
\begin{proof}
Let $r$ be a ray of $\Sigma(\XX^{\dagger})$. Choose a local equation $\DD_F=\div m$ near $E_r$, for some $m\in M_r$. Since $M_r$ is a $1$-dimensional lattice dual to $n_r\Z$, the description of~\ref{M Sigma} implies $F(n_r)=\langle m,n_r\rangle=\mult_{E_r}\DD_F$. This gives the first equality. 

For the second equality: by \ref{rays discussion} any ray $r$ of $\Sigma(\XX^{\dagger})$ is either contained in $[\varpi]^\perp$ or not, giving rise to a ray of the recession cone complex of $\D(\XX^{\dagger})$ or a vertex $v\in \D(\XX^{\dagger})$ respectively. In the second case $[E_r]=[E_v]$ and $F(n_r)=F(v)\langle [\varpi],n_r\rangle=F(v)\mult_{E_v}\div[\varpi]=F(v)m_v$.
\end{proof}
\begin{block} 
By Lemma \ref{lem correspondence cartier and pl functions} the following assertions are equivalent: 
\begin{enumerate}[label=(\roman*)]
	\item The cone complex $\Sigma(\XX^{\dagger})$ is \emph{simplicial}.
	\item For every set of rational numbers $\{q_r\}_{r\in \Sigma(\XX^{\dagger})(1)}$ there exists a function $F\in (M_{\D(\XX^{\dagger})})_{\Q}$ such that $F(n_r)=q_r$.
	\item Any $\Q$-Weil divisor supported on $D_{\XX}$ is $\Q$-Cartier. 
\end{enumerate}
	In particular if $\XX$ is $\Q$-factorial then $\Sigma(\XX^{\dagger})$ is simplicial. \label{rmk qfactoriality}
\end{block}
\begin{block}[Reminder on intersection products] Suppose $V$ is any normal Cohen-Macaulay scheme, $\DD$ is a Cartier divisor and $C$ is a normal curve in $V$ defined over a field $\kappa$. Then $V$ contains no associated points \cite[\S8.2.15]{Liu} and so the restriction of Cartier divisors $D\vert_{C}$ is well-defined \cite[\S7.1.29]{Liu}. We can therefore define the intersection\footnote{One can develop a general refined intersection theory of ($\Q$-)Cartier divisors and cycles \cite[\S20]{fulton intersection}, but not intersection products of two arbitrary cycles unless $\XX$ is smooth. We won't need this.} product  $\DD\cdot C\coloneqq \deg_{\kappa} \DD\vert_C\in \Z$, and if $\DD$ is only assumed $\Q$-Cartier instead of Cartier we let $\DD\cdot C=\frac{1}{e}(e\DD)\cdot C$ where $e\in\Z_{\ge1}$ is chosen such that $e\DD$ is Cartier. 
We have the usual rules $(\DD+\DD')\cdot C=\DD\cdot C+\DD'\cdot C$, and $\div(f)\cdot C=0$, so that $\DD\cdot C$ only depends on the linear equivalence class of $\DD$. 

We will mainly be interested in the case where $V=\XX$ is a toroidal $S$-scheme (which is Cohen-Macaulay by \cite[\S4.1]{K94}), $\kappa=\tilde{k}$ and $C=E_\tau$ for some vertical ridge $\tau$ of $\D(\XX^{\dagger})$. Note that $\DD$ intersects $E_\tau$ properly if and only if $F_\DD\vert_\tau$ is the zero function. Since $\tau$ is vertical and $\XX$ is assumed proper, $E_{\tau}$ is a proper $\tilde{k}$-curve  and $\DD\cdot E_\tau$ only depends on the linear equivalence class of $\DD$. 

We will later need the following \label{proj formula} outcome of the projection formula (see \cite[\S20]{fulton intersection}): let $g:V'\to V$ be a proper flat dominant morphism of integral normal $S$-schemes, let $\DD$ be a $\Q$-cartier divisor. Suppose $c'\in Z_1(V')$ is a \emph{vertical} $1$-cycle on $V'$ (i.e. a linear combination of curves contained in $V'_s$). Then using that $\tilde{k}$ is algebraically closed we have $$g^*{\DD}\cdot c'=\DD \cdot g_*c'.$$ 

\end{block}
\begin{example}\label{eg special fiber trivial divisor}
For instance if $\DD=\div(\varpi)$ we have $\cyc\DD=[X_s]=\sum_{v\in \D(\XX^{\dagger})(0)} m_v[E_v]$ by~\ref{lem computation weil divisor}. If additionally $\XX$ is $\Q$-factorial then the divisors $[E_v]$ for $v\in \D(\XX^{\dagger})(0)$ are $\Q$-Cartier, so for every vertical 1-cycle $c\in Z_1(\XX)$ we obtain by the projection formula discussed in \ref{proj formula} the useful relation $$0=\sum_{v\in \D(\XX^{\dagger})(0)} m_v [E_v]\cdot c.$$

\end{example}
\begin{block}[Normal vector] \label{normal vector} Let $(\sigma,\tau)$ be a facet-ridge pair of a $\Z$-PL complex $\Sigma$. Write $m_{\sigma/\tau}$ for the unique generator of $\ker(M_{\sigma}\to M_{\tau})$ such that $\langle m_{\sigma/\tau},\cdot\rangle$ is positive on $\sigma\setminus \tau$. Since $m_{\sigma/\tau}$ kills $N_{\tau}$, the map $\langle m_{\sigma/\tau},\cdot\rangle$ descends to a map $N_{\sigma}/N_{\tau}\to \Z$. Let $n_{\sigma/\tau}$ denote the unique generator of $N_{\sigma}/N_{\tau}$ such that $\langle m_{\sigma/\tau},n_{\sigma/\tau}\rangle >0$. We will also view $n_{\sigma/\tau}$ as a vector of $N_{\tau}$, well-defined up to an element of $N_{\tau}$, and call it the \emph{primitive normal vector} of $\sigma$ relative to $\tau$. 
\end{block}
\begin{block}[Notation]\label{tau+r}
Recall the notion of stars and links from \ref{stars and links}.
Suppose $\tau$ is a ridge of $\D(\XX^{\dagger})$ and $r\in \overline{\star}(C(\tau))(1)$ is a ray of $\Sigma(\XX^{\dagger})$ adjacent to $C(\tau)$, then we write $\sigma\coloneqq \tau+r$ for the unique facet adjacent to $\tau$ such that $r\in C(\sigma)(1)$.
\end{block}
\begin{proposition} 
		Suppose $\tau$ is a ridge of $\D(\XX^{\dagger})$. Let $r\in \link(C(\tau))(1)$ and $\sigma=\tau+r$. Suppose that the component $E_r$ associated to $r$ defines a $\Q$-Cartier divisor (eg. $\star(r)$ is simplicial, see Remark~\ref{rmk qfactoriality}). Then
		\begin{equation}
		[E_r]\cdot E_\tau=\frac{1}{[N_{\sigma}:N_{\tau}+N_r]}.
			\label{eqn:ev.et}
		\end{equation}
		Additionally, if $\sigma$ is simplicial, this also equals $\det\tau/\det\sigma$.
		\label{prop: computation ev.et}
\end{proposition}
\begin{proof} This resembles a classical fact of toric geometry, see for instance \cite[\S5.1, p. 98,(2)]{fulton}.

Let $n_{\sigma/\tau}$ be the normal vector as in~\ref{normal vector}. Then $N_{\sigma}=n_{\sigma/\tau}\Z+N_{\tau}$, so \begin{equation}
	\label{Msigma :Mtau +Mr}
[N_{\sigma}:N_{\tau}+N_r]n_{\sigma/\tau}\equiv n_r \bmod{N_{\tau}}.
\end{equation}

Suppose $\DD=l[E_r]$ is Cartier for some $l\in \Z_{>0}$. Note that $\mathrm{supp}\DD\cap E_\tau=E_{\sigma}$ set-theoretically.
So $l[E_r]\cdot E_{\tau}=\DD\cdot E_{\tau}=\ord_{E_{\sigma}}\O(\DD)\vert_{E_{\tau}}$.
 The support function $F_{\DD}$ is given on $\sigma$ by $\langle m,\cdot\rangle$ for some $m\in \ker((M_{\sigma})_\Q\to (M_{\tau})_\Q)$, and we have $\langle m,n_r\rangle =l$. It follows from Lemma \ref{lemma balanced cospec} that $E_{\tau}^{\dagger}$ is log-regular with $\CC_{E_{\tau},E_{\sigma}}^{\gp}\cong N_{\sigma}/N_{\tau}=n_{\sigma/\tau}\Z$.
So by restriction of local equations we obtain that $$\ord_{E_{\sigma}}\O(\DD)\vert_{E_{\tau}}=\langle m,n_{\sigma/\tau}\rangle.$$ 
Hence by the previous points we can compute $$[E_r]\cdot E_\tau=\frac{\langle m,n_{\sigma/\tau}\rangle}{l}=\frac{\langle m,n_r\rangle}{l[M_{\sigma}:M_{\tau}+M_r]}=\frac{1}{[M_{\sigma}:M_{\tau}+M_r]}.$$ 
 The final assertion follows immediately from Equation \eqref{Msigma :Mtau +Mr} and Definition \ref{det sigma}. 

\end{proof}
\begin{block}[Pullback of support functions] \label{pullback support}Let $f:\YY^{\dagger}\to\XX^{\dagger}$ be a finite morphism of \emph{proper} toroidal $S$-schemes and write $\phi:\D(\YY^{\dagger})\to\D(\XX^{\dagger})$ for the associated morphism of dual complexes. Let $\DD$ a $\Q$-Cartier divisor supported on $D_{\XX}$ with support function $F_{\DD}:\DD(\XX^{\dagger})\to\R$. Then $f^*\DD$ is supported on $D_{\YY}$ and has support function $$\phi^*F_{\DD}\coloneqq F_{f^*\DD}.$$ 
\end{block}
\begin{lemma} \label{lem pullback support}
	Keep notations and assumptions of~\ref{pullback support}, then $\phi^*F_{\DD}=F_{\DD}\circ \phi$.
\end{lemma}
\begin{proof}
Using Lemma~\ref{lem computation weil divisor} we have $$\mathrm{cyc}(f^*\DD)=f^*\mathrm{cyc}(\DD)=\sum_{r\in \Sigma(\XX^\dagger)(1)}F(n_{r})f^*[E_{r'}]=\sum_{r'\in \Sigma(\YY^\dagger)(1)}F(n_{r})e(E_{r'}/E_r)[E_{r'}].$$
By Lemma \ref{lemma extension dvr balanced} shows $n_{r'}=e(E_{r'}/E_r)n_r$, and so it follows that $f^*\DD$ is the $\Q$-Cartier divisor with support function $F_{\DD}\circ\phi$.
\end{proof}

\begin{definition}[harmonic and convex functions]\label{def harmonic and convex f}
	Let $\DD$ be a $\Q$-cartier divisor on $\XX^{\dagger}$. We call the support function $F_\DD$ \emph{harmonic}, \emph{resp. convex (or subharmonic)} or \emph{resp. strongly convex} if for all ridges $\tau$ of $\D(\XX^{\dagger})$ we have $\DD\cdot E_{\tau}=0$, $\DD\cdot E_{\tau}\ge 0$ or $\DD\cdot E_{\tau}>0$ respectively.
This is for instance the case if $\DD$ is numerically trivial, nef or ample respectively. The terminology is justified in Section~\ref{sec: tropical harmonicity}.


\end{definition}
\begin{block} With assumptions as in \ref{pullback support}, suppose $F_{\DD}$ is harmonic (resp. convex, resp. strongly convex), then $\phi^*F_{\DD}$ is harmonic (resp. convex, resp. strongly convex): this follows from the projection formula of \ref{proj formula}. \label{pullback support harmonic is harmonic}
\end{block}
\begin{block}[Cycles on $\Z$-PL complexes]\label{cycles}
		Let $\Sigma$ be a $\Z$-PL complex. We define the group of codimension 1 cycles $Z^1(\Sigma)$ as the group of formal sums of ridges $[\tau]$, and effective cycles are those with positive coefficients.  
		Now suppose $\phi:\Sigma'\to \Sigma$ is a balanced cover of $\Z$-PL complexes. Then we define pullbacks of cycles via the rule $f^*[\tau]=\sum_{\tau'}\deg_{\tau'}\phi[\tau']$ summing over ridges $\tau'$ above $\tau$, and extending linearly. Note that for any cycle $W\in Z^1(\Sigma)$ we have $\deg f^*W=\deg f\deg W$.
\end{block}

\begin{definition}[Specialisation] Let $\DD$ be a $\Q$-Cartier divisor supported on $D_{\XX}$. Then we define the \emph{specialisation} of $\DD$ as the cycle \begin{equation}
		\sp_*\DD\coloneqq \sum_{\tau \text{ vertical}} \frac{\DD\cdot [E_{\tau}]}{\vol(\tau)}[\tau] \in Z^1(\D(\XX^{\dagger}))\label{defn sp},
	\end{equation}
	summing over \emph{vertical} ridges $\tau$ of $\D(\XX^{\dagger})$.

\end{definition}
\begin{block} Note that $\sp_*\DD$ only depends on the linear equivalence class of $\DD$. Also, $\sp_*\DD$ is trivial (resp. effective) if and only if $F_\DD$ is harmonic (resp. convex).
	 \end{block}	 
\begin{proposition} \label{prop pullback commutes with specialisation} Let $\DD\in M_{\D(\XX^{\dagger})}$ and $\phi:\D(\YY^{\dagger})\to\D(\XX^{\dagger})$ as in \ref{pullback support}. Then pullback commutes with specialisation: \begin{equation}
	\phi^*\sp_*\DD=\sp_*f^*\DD\end{equation}
\end{proposition}
\begin{proof} Let $\tau'$ be a ridge of $\D(\YY'^{\dagger})$ and let $\tau=\phi(\tau')$. By Lemma \ref{lemma z-pl mu on polyhedron} we have $\vol(\tau)/\vol(\tau')=[M_{\tau'}:M_{\tau'}]$. For any facet $\sigma$ adjacent to $\tau$ we have $$\mult_{E_{\sigma'}}f^*\DD\vert_{E_{\tau'}}=e(E_{\sigma'}/E_{\sigma})\mult_{E_{\sigma}}\DD\vert_{E_\tau}.$$ Therefore by \eqref{eqn:l} we have $$\frac{\mult_{E_{\sigma'}}f^*\DD\vert_{E_{\tau'}}}{\vol\tau'}=[M_{\sigma'}:M_{\sigma}]\frac{\mult_{E_{\sigma}}\DD\vert_{E_{\tau}}}{\vol\tau}.$$ Theorem~\ref{thm: balanced} then implies $$\phi^*\sp_*\DD=\sum_{(\sigma',\tau')} [M_{\sigma'}:M_{\sigma}]\frac{\mult_{E_{\sigma}}\DD\vert_{E_\tau}}{\vol(\tau)}[\tau']=\sum_{(\sigma',\tau')} \frac{\mult_{E_{\sigma'}}f^*\DD\vert_{E_{\tau'}}}{\vol(\tau')}[\tau']=\sp_*f^*\DD$$ where we sum over all facet-ridge pairs $(\sigma',\tau')$ of $\D(\YY^\dagger)$ with $\tau'$ bounded.
\end{proof}
\begin{block}[Horizontal strata]	
Suppose that $x\in F(\XX^\dagger)$. Recall from~\ref{construction dual complex} that $\sigma_x$ is empty if and only if $x$ lies in the special fiber $\XX_s$. Suppose $x\not\in \XX_s$, we also say $x$ is \emph{horizontal} and write $F(\XX^{\dagger})^{\mathrm{hor}}$ for the horizontal points. The generic point of $\XX$ is also considered to be horizontal. By~\ref{lemma balanced cospec} it follows that for all $x\in F(\XX^{\dagger})^{\mathrm{hor}}$ we have a strict closed immersion $E^{\dagger}_x=E_x(\log D_x)\to X^{\dagger}$ where $D_x=E_x\setminus E_x^{\circ}$ and $E^\dagger_x/S^\dagger$ is again a proper toroidal $S^{\dagger}$-scheme \footnote{if $x$ is vertical this is not the case, because if we equip $E_x$ with the pullback log-structure then we obtain a so-called hollow log-scheme over the log-point $\tilde{k}$.}. In particular we have an associated dual complex $\D(E^{\dagger}_x)$.  
\end{block}
\begin{block}[compactified dual complex]
	We write $$\overline{\D}(X^{\dagger})\coloneqq \sqcup_{x\in F(X^{\dagger})^{\mathrm{hor}}}\D(E_x^{\dagger}).$$ 
	We will see below in \ref{retract to compact skeleton} that $\overline{\D}(X^{\dagger})$ admits a natural compact topology, such that each $\D(E_x^{\dagger})\subset \overline{\D}(X^{\dagger})$ is a subspace, so we call $\overline{\D}(X^{\dagger})$ the \emph{compactified dual complex} -- beware it is not a $\Z$-PL complex but rather a $\Z$-PM complex (see Remark \ref{zpm} below). 

	The specialisation map extends to $\overline{\D}(\XX^{\dagger})$ by amalgamating the various specialisation maps for $x\in F(X^{\dagger})^{\mathrm{hor}}$: \begin{equation}
	\overline{\sp}_*\DD\coloneqq \bigsqcup_{x\in F(\XX^{\dagger})^{\mathrm{hor}}}\overline{\sp}_*(\DD\vert_{E_x})\in Z^1(\overline{\D}(X^{\dagger}))\coloneqq \bigsqcup_{x\in F(\XX^{\dagger})^{\mathrm{hor}}}Z^1(\D(E_x^{\dagger})).
\end{equation}
\end{block}

\begin{block}[Notation for analytifications] 
Let $Z$ be a finite type $k$-scheme. Then recall the Berkovich analytification $Z^{\an}$ is canonically identified with the set of bounded $k$-valuations $|\cdot(z)|:k(Z)\to \R_{\ge 0}$ of rank one  (written multiplicatively) with the topology of pointwise convergence \cite{B90}, and we have a canonical map $\ker:Z^{\an}\to Z$ that admits a continuous section that we view as an inclusion $Z\subset Z^{\an}$. 
\end{block}

\begin{block}[Skeleta]\label{skeleta}
Let $\XX^{\dagger}$ be a proper toroidal $S$-scheme. Write $X=\XX_k$ for the generic fiber, note it is a proper $k$-variety that is regular away from $(D_{\XX})_{k}$. 

By the theory of admissible expansions \cite[\S3]{BM}, \cite[\S2.4]{MN} (not reviewed here) there exists a canonical topological embedding $$\iota_{\XX}:\D(\XX^{\dagger})\to (\XX\setminus D_{\XX})^{\an}.$$ 
We call the image of $\iota_{\XX}$ the \emph{skeleton} $\sk(\XX^{\dagger})$ of $\XX^{\dagger}$. By loc. cit. the embedding $\iota_{\XX}$ admits a continuous retraction\footnote{It is plausible that $\rho_{\XX}$ is a strong deformation retract, but we do not study this.} $$\rho_{\XX}:(\XX\setminus D_{\XX})^{\an}\to \sk(\XX^{\dagger}),$$ which is described as follows: let $\sp_{\XX}:X^{\an}\to \XX_s$ be the specialisation map, and let $$\xi:(\XX\setminus D_{\XX})^{\an}\to F(\XX^{\dagger})$$ be the map which sends $z$ to the unique minimal stratum containing $\sp_{\XX}(z)$, one verifies that $\xi$ is continuous. 
Then an element $z=|\cdot(z)|\in (\XX\setminus D_{\XX})^{\an}$ gives an additive monoid $k$-valuation $$\rho(z):\CC_{\XX,\xi(z)}\to \R_{\ge0}: m\mapsto \log_{|\varpi|}|m(z)|,$$ that is, an element $$\rho(z)\in \sigma_{\xi(z)}\subset \D(\XX^\dagger)\overset{i_\XX}{\cong} \sk(\XX^{\dagger}),$$ also see Remark~\ref{rmk monoid valuations}. Since we have removed the divisor $D_{\XX}$ it follows that $m(z)$ is never zero and so $\rho(z)$ is well-defined.

Recall that by Section~\ref{balanced cover Z-pl spaces} the dual complex $\D(\XX^{\dagger})$, and thus $\sk(\XX^{\dagger})$ carries a $\Z$-PL structure, and thus by Section~\ref{sec: conformal} also a canonical conformal measure $\mu_{\sk(\XX^{\dagger})}$. Volumes will always be taken with respect to this measure. If $x\in F(\XX^\dagger)$ is a log-stratum of $\XX^\dagger$, by a slight abuse we also denote by $\sigma_x$ the associated polyhedron of $\sk(\XX^\dagger)$. Recall that the lattice of $\Z$-PL functions $M_{\sigma_x}$ coincides with the group of Cartier divisors on $\O_{\XX,x}$ supported on $D_{\XX,x}$.\end{block}
\begin{block}[Compactified skeleta]\label{retract to compact skeleton}
The open strata $E^{\circ}_x$ for $x\in F(\XX^{\dagger})^{\mathrm{hor}}$ are disjoint and cover the generic fiber $X$, and by \ref{skeleta} we have continuous skeleton retraction maps $$\rho_{E_x^\dagger}:(E_x^{\circ})^{\an}\to \sk(E^{\dagger}_x)\subset (E_x^{\circ})^{\an}.$$
For all horizontal $x$ we have closed embeddings $E_x^{\an}\to X^{\an}$, and so we obtain a retraction
\begin{equation}
	\overline{\rho}=\overline{\rho}_{\XX^\dagger}:X^{\an}\lra \overline{{\sk}}(\XX^\dagger)\coloneqq \bigsqcup_{x\in F(\XX^{\dagger})^{\mathrm{hor}}}\sk(E^{\dagger}_x).\label{eq: cr}
\end{equation}
We call $\overline{\rho}$ the \emph{compactified} retraction to the \emph{compactified skeleton} $\overline{\sk}(\XX^{\dagger})$.  Since the topological boundary of $\sk(\XX^{\dagger})$ is contained in $D_{\XX}^{\an}$, 
 it follows by induction that $\overline{\sk}(\XX^{\dagger})$ is the closure of $\sk(\XX^{\dagger})$ and $\overline{\rho}$ is continuous.

In what follows we give a more intrinsic description of $\overline{\rho}$ when restricted to $X^{\an}$.
Let $z\in X^{\an}$.
Write $$\p_z=\{m\in \CC_{\XX,\xi(z)}: |m(z)|=0\}.$$ Then either $\p_z=\CC_{\XX,\xi(z)}$ or $\p_z$ is a prime ideal of the monoid $\CC_{\XX,\xi(z)}$. 
Recall that prime ideals of $\CC_{\xi(z)}$ correspond to strata containing $\xi(z)$. In fact, $\p_z$ corresponds to the minimal stratum containing $\ker(z)$, which is necessarily horizontal, we denote this stratum by $E_{z}$. By~\ref{lemma balanced cospec} (v) (also see Remark~\ref{snc analogy}) we have $$\CC_{\XX,\xi(z)}/\p_z\cong \CC_{E_z,\xi(z)}^{0}.$$ Therefore, the multiplicative valuation $|\cdot(z)|$ induces a well-defined additive valuation \begin{equation}
	\label{eq: retr} \CC_{E_z,\xi(z)}\to \R_{\ge 0}:m\mapsto \log_{|\varpi|}|m(z)|;
\end{equation}
which defines the element $\overline{\rho}(z)$ of $\D(E_z^\dagger)$. The point is that $|m(z)|\ne 0$ for all $m\in \CC_{E_z,\xi(z)}$ by construction of $E_z$. 

It follows from the constructions that $\rho$ and $\overline{\rho}$ are functorial with respect to toroidal covers.
 \end{block}

\begin{remark}[$\Z$-PM spaces]\label{zpm}
	The compactified skeleton $\overline{sk}(\XX^\dagger)$ is an example of an integral piecewise monomial space (short: $\Z$-PM space), these will be studied systematically in upcoming work of Ducros-Thuillier (in preparation). In short, $\Z$-PM spaces are multiplicative extensions of $\Z$-PL spaces: they are locally cut out in $\R_{\ge 0}$ by equations of the form $cx_1^{a_1}\dots x_n^{a_n}\le 1$ where $c\in |k^\times|$ and $a_i\in\Z$, where the coordinates are possibly $0$ -- informally speaking this allows to study the boundary of $\Z$-PL spaces ``at infinity''. 
\end{remark}
\begin{block}[Support functions as model metrics]\label{an view supports} Let $\DD$ be a Cartier divisor on $\XX$ supported on $D_{\XX}$. Then we can extend the definition of the support function $F_{\DD}$ to all of $\XX^\an$ as follows: $$F_\DD:\XX^\an\to\R_{\ge0}\cup \{\infty\}:z\mapsto \log_{|\varpi|}\|1\|_{\DD,z},$$ here $\|\cdot\|_{\DD}$ denotes the model metric associated to the line bundle $\O(\DD)$ -- it is characterised as follows: if $f$ is local equation of $\DD$ near $\xi(z)\in F(\XX^\dagger)$, then we have $$F_\DD(z)=\log_{|\varpi|}|f(z)|.$$ The previous expression only depends on the image of $f$ in $\CC_{\XX,\xi(z)}$, and so it follows that the support function $F_{\DD}$ factors through the retraction map $\overline{\rho}$. 
\end{block}

\section{Harmonic morphism of $\Z$-PL tropical complexes}\label{sec: tropical harmonicity}
\begin{block}[Overview] 
In this section we introduce $\Z$-PL tropical complexes, these are, as the name suggests, a $\Z$-PL enhancement of the weak\footnote{The tropical complexes of \cite{Car} are required to satisfy a positivity condition for the local Hodge index, this condition will be of no importance for the current work, so we will omit the adjective ``weak'' in the sequel.} tropical complexes of \cite{Car}. In short, these are simplicial $\Z$-PL complexes additionally decorated with some structure constants called $\alpha$-numbers, modeled after the intersection numbers $[E_r]\cdot E_\tau$ for a ridge $\tau$ and a ray $r\in C(\tau)(1)$. Any dual complex $\D(\XX^{\dagger})$ of a $\Q$-factorial toroidal $S$-scheme is a $\Z$-PL tropical complex. 

Given a PL function $F$ on a $\Z$-PL tropical complex $\Sigma$, we give a definition of the slopes $\partial_{\sigma/\tau}$ for any facet-ridge pair $(\sigma,\tau)$, and the Laplacian of $F$ is defined as $\Delta(F)=\sum_{(\sigma,\tau)}\partial_{\sigma/\tau}[\tau]$, summing over facet-ridge pairs with $\tau$ vertical. We prove that balanced covers are harmonic morphisms in the sense that harmonic functions (those satisfying $\Delta(F)=0$) are preserved via pullback (Theorem \ref{thm: harmonic morphisms}).
 \end{block}
 \begin{block}[Assumptions for this section]\label{assumptions harmonicity} In this section let $f:\XX^{\dagger}\to \YY^{\dagger}$ denote a finite cover of proper $\Q$-factorial toroidal $S$-schemes.
 \end{block}
\begin{block}[Computations]\label{computationes} As before let $\Sigma(\XX^{\dagger})$ denote the cone complex associated to $\XX^{\dagger}$, note that $\Sigma(\XX^{\dagger})$ is simplicial by the hypothesis that $\XX$ is $\Q$-factorial, see Remark \ref{rmk qfactoriality}. Let $\tau$ be a ridge of $\D(\XX^{\dagger})$ and let $C(\tau)$ be the cone generated by $\tau$ in $\Sigma(\XX^\dagger)$. For any ray $r\in C(\tau)(1)$ write $\alpha_{r,t}=-[E_r]\cdot E_\tau\in \Q$. 
 For any vertex $v\in \tau(0)$ write $\alpha_{v,\tau}=\alpha_{r_v,\tau}$ where $r_v$ is the ray in $\Sigma(\XX^{\dagger})$ generated by $v$. We define the \emph{bounded multiplicity} of $\tau$ as $$\mathrm{mult}_b\tau\coloneqq\sum_{v\in\tau(0)}m_v\alpha_{v,\tau}. $$ By Example~\ref{eg special fiber trivial divisor} and Proposition~\ref{prop: computation ev.et} we then have \begin{equation}
		\label{eq:alfa}\mathrm{mult}_b\tau=\sum_{v\in \mathrm{link}(\tau)(0)}\frac{m_v\det(\tau^-)}{\det(\sigma^-)}
\end{equation}
For every $F\in M_{\Sigma(\XX^{\dagger})}$ we compute, using Proposition~\ref{prop: computation ev.et} again, that

 \begin{align}\label{c} \DD_F\cdot E_{\tau}&=\sum_{r\in\Sigma(1)}F(n_r)[E_r]\cdot E_\tau=\sum_{r\in (\mathrm{link}_{\Sigma}C(\tau))(1)}F(n_r) [E_r]\cdot E_\tau+\sum_{r\in C(\tau)(1)}F(n_r) [E_r]\cdot E_\tau
		\end{align} 

		We can split the last sum in two parts $(1)+(2)$, the first part being 
\begin{align}
 (1)&=\sum_{r\in (\mathrm{link}_{\Sigma}C(\tau))(1)}F(n_r) [E_r]\cdot E_\tau\\&=\sum_{v\in (\mathrm{link}_\D\tau)(0)} m_vF(v)[E_v]\cdot E_{\tau}-\sum_{v\in \tau(0)} m_v\alpha_{v,\tau}[E_v]\\
 &=\sum_{v\in (\mathrm{link}_\D\tau)(0)} m_v[E_v]\cdot E_{\tau}\left(F(v)-\frac{\sum_{v\in \tau(0)} m_v\alpha_{v,\tau}[E_v]}{\mult_b\tau}\right)
		\end{align}

		And the second part is given by
		\begin{align}
 (2)=&\sum_{r\in\mathrm{link}_{\mathrm{rec}(\D)}\rec(\tau)(1)}\partial_rF[E_r]\cdot E_{\tau}-\sum_{r\in\mathrm{rec}(\tau)(1)} \partial_rF\alpha_{r,\tau}\\
 &=\sum_{r\in \mathrm{link}_{\mathrm{rec}(\D)}\mathrm{rec}(\tau)(1)}[E_r]\cdot E_{\tau} \left(\partial_rF-\frac{\sum_{r\in\mathrm{rec}(\tau)(1)} \partial_rF\alpha_{r,\tau}}{\mult_u\tau}\right),
		\end{align}
		where we define the \emph{unbounded multiplicity} of $\tau$ as $$\mult_u\tau\coloneqq \sum_{r\in \mathrm{link}_{\mathrm{rec}(\D)}\mathrm{rec}(\tau)(1)} [E_r]\cdot E_\tau=\sum_{r\in \mathrm{link}_{\mathrm{rec}(\D)}\mathrm{rec}(\tau)(1)} \det((\tau+r)^-)/\det(\tau^-).$$
\end{block}
\begin{remark} Our definition of $\mult_b(\tau)$ corresponds to $\deg_b(\Delta_\tau)$ in \cite[6.4]{GRW} - we use $\mult$ rather than $\deg$ to avoid confusion with the local degree of a finite cover.
\end{remark}
\begin{definition}[$\Z$-PL tropical complex]\label{defn zpl trop} A $\Z$-PL tropical complex $(\Sigma,\alpha)$ is a simplicial $\Z$-PL complex $\Sigma$ endowed with structure constants $\alpha_{r,\tau}\in \Q$ for any ray-ridge pair $(r,\tau)$ with $r\in C(\tau)(1)$ such that~\ref{eq:alfa} holds for any ridge $\tau$.

A \emph{balanced finite cover} of $\Z$-PL tropical complexes  $\phi:(\Sigma',\alpha')\to(\Sigma,\alpha)$ is a balanced finite cover of $\Z$-PL complexes such that additionally for each ridge $\tau'$ of $\Sigma'$ the conclusion of the Lemma \ref{lem atv.atv} below holds (where we let $[E_{\tau'}:E_{\tau}]=\deg_{\tau'}\phi/[M_{\tau'}:M_{\tau}]$). 
\end{definition}

\begin{lemma} \label{lem atv.atv}Let $\tau'$ be a ridge of $\YY^{\dagger}$. Suppose $r'\in C(\tau')(1)$, and let $r=\phi(r')$ and $\tau=\phi(\tau')$,
	then 
	\begin{equation}
	 \label{eq atv.atv}\alpha_{r',\tau'}=[E_{\tau'}:E_{\tau}]\alpha_{r,\tau}.
	\end{equation}
\end{lemma}
\begin{proof} This follows from the projection formula (\ref{proj formula}) and the fact that $\{r'\}=f^{-1}(r)\cap C(\tau')(1)$.
\end{proof}
\begin{definition}[Slopes] \label{defn slope} Let $(\Sigma,\alpha)$ be a $\Z$-PL tropical complex and let $F\in (M_{\Sigma})_{\Q}$ be a PL function affine on the polyhedra of $\Sigma$. For any facet-ridge pair $(\sigma,\tau)$, let $r\in C(\sigma)(1)\setminus C(\tau)(1)$ be the unique ray of $C(\sigma)$ not in $C(\tau)$ (such an $r$ is unique since $\Sigma$ is assumed simplicial). Then we have two cases: 
\begin{itemize}
\item $E_r$ is vertical: then we define
		$$\partial_{\sigma/\tau}F\coloneqq \frac{1}{\vol(\sigma^-)}\left(F(v_r)-\frac{\sum_{w\in \tau(0)}m_w\alpha_{w,\tau}F(w)}{\deg_b\tau}\right).$$ 
 	\item $E_r$ is horizontal: then we define
 		$$\partial_{\sigma/\tau}F\coloneqq \frac{1}{\vol(\sigma^-)}\left(\partial_{n_r}F-\frac{\sum_{s\in C(\tau)(1)} \alpha_{s,\tau}\partial_{n_s}F}{\deg_u\tau}\right).$$ 
 \end{itemize} 
\end{definition}
\begin{definition}[Discrete Laplacian]
	Let $(\Sigma,\alpha)$ be a $\Z$-PL tropical complex. Let $F:\Sigma\to \R$ be linear on the faces of $\Sigma$. Then the Laplacian of $F$ is defined as the cycle 
	\begin{equation}\label{definition laplacian}
		\Delta(F)\coloneqq \sum_{(\sigma,\tau);\ \tau\text{ vertical}}\partial_{\sigma/\tau}F[\tau],
	\end{equation}
	summing over facet-ridge pairs $(\sigma,\tau)$ of $\D(X^{\dagger})$ such that $\tau$ is \emph{vertical}.
	We call $F$ \emph{harmonic} near $x\in \Sigma$ in case for all ridges $\tau$ containing $x$, we have 
	\begin{equation}\label{harmonic}
		\sum_{\sigma}\partial_{\sigma/\tau}F=0,
	\end{equation}
	summing over facets $\sigma$ containing $\tau$.
\end{definition}
\begin{theorem}[Poincar\'e-Lelong Slope formula] Let $\XX^{\dagger}$ be a proper $\Q$-factorial toroidal $S$-scheme. Then for any $F\in M_{\D(\XX^\dagger)}$ we have
\begin{equation}
		\label{eq: pll} \sp_*\DD_F=\Delta(F)
\end{equation} \label{thm: pll}
\end{theorem}
\begin{proof} The proof is carried out in the computations of ~\ref{computationes}: there we compared the $[\tau]$-coefficients of both sides, and the definitions of the slopes are chosen such that equality holds.
\end{proof}
\begin{remark} For comparison we note that the slope formula of \cite[\S6.8]{GRW} deals with \emph{algebraically closed ground fields} and  principal divisors. \end{remark}
\begin{theorem}[Harmonic morphisms] \label{thm: harmonic morphisms}Let $\phi:(\Sigma',\alpha')\to(\Sigma,\alpha)$ be a balanced finite cover of $\Z$-PL tropical complexes (Definition \ref{defn zpl trop}). Then 
\begin{equation} \label{eq pullback lapl}
\phi^*\Delta(F)=\Delta(F\circ\phi)\end{equation}
In particular, $\phi$ is a harmonic morphism in the sense that harmonic functions pull back to harmonic functions.
\end{theorem}
\begin{proof} 
We have $$\phi^*\Delta_*F_\DD=\sum_{(\sigma',\tau')} [M_{\sigma'}:M_{\sigma}]\partial_{\sigma'/\tau'}[\tau'],$$ summing over facet-ridge pairs $(\sigma',\tau')$ of $\D(\YY^{\dagger})$ with $\tau'$ bounded. So it suffices to prove that $[M_{\sigma'}:M_{\sigma}]\partial_{\sigma'/\tau'}=\partial_{\sigma/\tau}$. By \ref{lemma z-pl mu on polyhedron} it follows that $[M_{\sigma'}:M_{\sigma}]=\vol(\sigma^-)/\vol(\sigma'^{-})$, and so by Definition of $\sigma_{\sigma/\tau}$ (see \ref{defn slope}) it suffices to prove that $a_{v',\tau'}/a_{v,\tau}$ is independent of the choice of vertex $v'\in \tau'(0)$: this is implied by the definition of a balanced morphism of $\Z$-PL tropical complexes, see \ref{defn zpl trop}.
\end{proof}
\begin{remark}
		The case of dual complexes follows more directly from the slope formula (Theorem~\ref{thm: pll}), Proposition~\ref{prop pullback commutes with specialisation} and Lemma \ref{lem pullback support}, since these give $$\phi^*\Delta(F)=\phi^*\sp_*\DD_F=\sp_*f^*\DD_F=\Delta(F_{f^*\DD_F})=\Delta(F\circ\phi).$$
\end{remark}
\begin{remark} Both Theorems \ref{thm: pll} and \ref{thm: harmonic morphisms} extend to the compactified skeleta (see \ref{retract to compact skeleton}) as follows: we have $\overline{\sp}_*\DD_F=\overline{\Delta}(F)$ where we let $\overline{\Delta}(F)\coloneqq \bigsqcup_{x\in F(\XX^{\dagger})^{\mathrm{hor}}}\Delta(F\vert_{\D(E_x^{\dagger})})\in Z^1(\overline{\D}(\XX^{\dagger}))$, and we have $\phi^*\overline{\Delta}(F)=\overline{\Delta}(F\circ \phi)$. We leave the details to the interested reader.
	\label{block compact pll} 
\end{remark}

\section{Complements on determinants of log-differentials}\label{sec:log diff}
\begin{block}[Overview] In this section we discuss some calculus of logarithmic differential forms for later use in Section \ref{sec:rh}. Theorem~\ref{thm: relative canonical vs relative log canonical} gives an elementary proof of a formula comparing canonical sheaves to log-canonical sheaves.
\end{block}
\begin{block}[Definition of logarithmic differentials]\label{block log differentials} Suppose $f^\dagger:X^\dagger=(X,\M_X)\to Y^\dagger=(Y,\M_Y)$ is a morphism of log-schemes. The $\O_X$-sheaf $\Omega_{X^\dagger/Y^\dagger}$ is defined as the universal log-derivation $(\mathrm{d},\mathrm{dlog})$ satisfying the rules $\mathrm{d}\,m=m\,\mathrm{dlog}(m)$ and $\mathrm{dlog}(n)=0$ for all sections $m\in \M_X$, $n\in f^{-1}\M_y$. In fact $\Omega_{X^\dagger/Y^\dagger}$ is the quotient of $\Omega_{X/Y}\oplus\left(\O_X\otimes \M_X^{\mathrm{gp}}\right)$ by the aforementioned relations \cite[Proof of IV.1.2.4]{ogus}. If $f$ is locally of finite type and the log-structures are {coherent}, then $\Omega_{X^\dagger/Y^\dagger}$ is a finitely generated $\O_X$-module. Similar as for usual schemes, formation of log-K\"ahler differentials is compatible with arbitrary base change in the category of log-schemes as well in the categories of fine log-schemes or \emph{fs} log-schemes \cite[IV.1.2.15]{ogus}. 
\end{block}

\begin{block}[log-cotangent sequences.] We have the following logarithmic analogues of the usual cotangent sequences for K\"ahler differentials: suppose $f^\dagger:X^\dagger\to Y^\dagger$ is a morphism of log-schemes over a log-scheme $S^\dagger$, then we have an exact sequence of $\O_X$-modules \cite[Proposition IV.2.3.1]{ogus} \begin{equation} \label{eqn: 1st cotangent}
        f^*\Omega_{Y^\dagger/S^\dagger}\lra \Omega_{X^\dagger/S^\dagger}\lra \Omega_{X^\dagger/Y^\dagger}\lra 0.
    \end{equation}

 If $f^\dagger$ is log-smooth (see \ref{defn: log reg}), then \eqref{eqn: 1st cotangent} is left-exact and locally split \cite[Theorem IV.3.2.3]{ogus}, and if additionally $X^\dagger,Y^\dagger$ are fine, then $\Omega_{X^\dagger/Y^\dagger}$ is locally free of finite rank. 
    \end{block}
\begin{block}
   A morphism of log-schemes $f^\dagger:X^\dagger\to Y^\dagger$ is called 
   a \emph{strict closed immersion} (resp. \emph{strict regular closed immersion}) if $f^\dagger$ is strict and $f$ is a closed immersion (resp. regular closed immersion). If $f^\dagger:X^\dagger\to Y^\dagger$ is a strict closed immersion of log-schemes over $S^\dagger$, defined by a coherent ideal sheaf $\mathscr{I}$, then we denote by $\mathscr{C}_{X/Y}\coloneqq \mathscr{I}/\mathscr{I}^2$ the conormal sheaf of $f$, and we have the cotangent exact sequence \cite[Theorem IV.2.3.2]{ogus} \begin{equation} \label{eqn: 2nd cotangent}
        \mathscr{C}_{X/Y}\overset{\delta}{\lra} f^*\Omega_{Y^\dagger/S^\dagger}\lra \Omega_{X^\dagger/S^\dagger}\lra 0,
    \end{equation}
    where $\delta$ sends a section $s\in\mathscr{C}_{X/Y}$ to the class of $(\mathrm{d}s,1)$.
\end{block}
\begin{lemma}[Hyodo-Kato]\label{lem HK}
  Let $f:X^\dagger\to Y^\dagger$ be a morphism locally of finite type of fine (resp. \emph{fs}) log-schemes. Then \'etale-locally on $X$ there exists a fine (resp. \emph{fs}) log-scheme $W^\dagger$ and a factorisation $X^\dagger\to W^\dagger\to Y^\dagger$ where $X^\dagger\to W^\dagger$ is a strict closed immersion and $W^\dagger\to Y^\dagger$ is log-smooth.
\end{lemma}
\begin{proof}
This follows from \cite[\S2.9.2,\S2.9.3]{HK} in the fine case and the proof also works in the \emph{fs} case as we explain now. If in \cite[2.9.2]{HK} the charts $P,Q$ are chosen \emph{fs}, then so are $Z_1^\dagger$ and $Z_2^\dagger$ since they admit charts by $P$ and $Q\oplus \N^r$ for some $r\in\N$. The proof of \cite[2.9.3]{HK} relies on \cite[4.10]{Kato}, and if in the proof of \cite[4.10]{Kato} the chart $h:P\to Q$ is \emph{fs} then so is the monoid $Q'=(h^{\gp})^{-1}(P)\subset Q^{\gp}$: indeed, if $q\in Q^{\gp}$ and $q^n\in Q'$ for some $n\in\N$ then $h(q)^n=h(q^n)\in P$ implies $q\in Q'$. 
\end{proof}

\begin{definition}[Log-l.c.i.]
    \label{defn: log lci}  A morphism of \emph{fs} log-schemes $f^\dagger:X^\dagger\to Y^\dagger$ is called \emph{log-l.c.i.} (short for \emph{log-locally of complete intersection}) if \'etale locally on $X$, the morphism $f^\dagger$ factors as a strict regular closed immersion followed by a log-smooth morphism. 
\end{definition}

\begin{lemma}\label{lem: log-regulars log-lci}
  Suppose $f^\dagger:X^\dagger\to Y^\dagger$ is a morphism locally of finite type of log-regular log-schemes. Then $f^\dagger$ is log-l.c.i.
\end{lemma}
  \begin{proof} Since the statement is \'etale-local on $X$, we may assume there exists a factorisation $X^\dagger\to W^\dagger\to Y^\dagger$ as in Lemma~\ref{lem HK}. Since $W^\dagger\to Y^\dagger$ is log-smooth, it follows that $W^\dagger$ is log-regular and $X^\dagger\to W^\dagger$ is a strict closed immersion of log-regular log-schemes, and from \cite[4.2]{K94} it follows that $X^\dagger\to W^\dagger$ is a strict regular closed immersion. \end{proof}

\begin{proposition} \label{propn : log-lci ses log-kahler} Suppose $f^\dagger:X^\dagger\lra Y^\dagger$ is a log-l.c.i. morphism (Definition~\ref{defn: log lci}) of \emph{fs} log-schemes over $S^\dagger$, whose underlying morphism $f:X\to Y$ is a dominant and generically smooth morphism of integral schemes. Assume that the log-structures of $X^\dagger$ and $Y^\dagger$ are generically trivial. Then the cotangent sequence is left-exact: 
\begin{equation}
        0\lra f^*\Omega_{Y^\dagger/S^\dagger}\lra \Omega_{X^\dagger/S^\dagger}\lra \Omega_{X^\dagger/Y^\dagger}\lra 0.\label{eq ll}
    \end{equation}\end{proposition}
\begin{proof} The statement is \'etale-local with respect to $X$, so we may assume that $f^\dagger$ factors as a strict regular closed immersion $i^\dagger:X^\dagger\to W^\dagger$ followed by a log-smooth morphism $h^\dagger:W^\dagger\lra Y^\dagger$. We claim that the morphism $$\delta:\mathscr{C}_{X/W}\lra i^*\Omega_{W^\dagger/Y^\dagger}$$ as in \eqref{eqn: 2nd cotangent} is injective. To verify the claim, first note that $\mathscr{C}_{X/W}$ is locally free of finite rank since $i$ is a regular closed immersion. As $h^\dagger$ is log-smooth, $f^*\Omega_{Y^\dagger/Z^\dagger}$ is locally free of finite rank too. It will therefore suffice to show that $\delta$ is injective at the generic point of $X$. But generically, the log-structures are trivial, and upon choosing bases of $\mathscr{C}_{X/W}$ and $i^*\Omega_{W/Y}$ one sees that $\delta$ is given by a Jacobian matrix of $f$, which by the assumption on generic smoothness has maximal rank, and this settles the claim. 

Since $h^\dagger$ is log-smooth the morphism $h^*\Omega_{Y^\dagger/S^\dagger}\to \Omega_{W^\dagger/S^\dagger}$ is split injective, implying the same is true for $f^*\Omega_{Y^\dagger/S^\dagger}\to i^*\Omega_{W^\dagger/S^\dagger}$. The conclusion now follows from a diagram chase in the following commuting diagram of exact complexes.

$$\begin{tikzcd}
             & 0 \arrow[r]                              & \mathscr{C}_{X/W} \arrow[r] \arrow[d]    & i^*\Omega_{W^\dagger/Y^\dagger} \\
             &                                          & i^*\Omega_{W^\dagger/S^\dagger} \arrow[d] \arrow[ru] &                     \\
             & f^*\Omega_{Y^\dagger/S^\dagger} \arrow[r] \arrow[ru] & \Omega_{X^\dagger/S^\dagger}  \arrow[d]              &                     \\
0 \arrow[ru] &                                          & 0                                        &                    
\end{tikzcd}$$
\end{proof}
\begin{remark} In the special case where $f^\dagger:X^\dagger\to Y^\dagger$ is an integral morphism, 
we can also interpret the proof of Proposition~\ref{propn : log-lci ses log-kahler} via the log-cotangent complex: By \cite[6.9]{Olsson} Olsson's logarithmic cotangent complex $\mathbb{L}_{X^\dagger/Y^\dagger}$ associated to $f^\dagger$ is concentrated in degrees $[-1,0]$, and is quasi-isomorphic to $$[\mathscr{C}_{X/W}\overset{\delta}{\lra} i^*\Omega_{W^\dagger/Y^\dagger}].$$ 
Since $\delta$ is injective the homology in degree $-1$ is trivial. Since $f:X^\dagger\to Y^\dagger$ is an integral morphism, we have by \cite[1.1.6]{Olsson} the distinguished transitivity triangle $$f^*\mathbb{L}_{Y^\dagger/S^\dagger}\to \mathbb{L}_{X^\dagger/Y^\dagger}\to \mathbb{L}_{X^\dagger/S^\dagger}\overset{+1}{\to},$$ and the result follows by examination of the associated long exact sequence. Perhaps there is also an approach using Gabber's log-cotangent complex \cite[\S8]{Olsson}, but we do not pursue the details.  \end{remark}


\begin{block}[Residue exact sequence.] Let $D$ be a prime Cartier divisor on an integral scheme $X$ and consider the log-scheme $X^\dagger=X(\log D)$. Abbreviate $\Omega_{X}=\Omega_{X/\Z}$ and $\Omega_{X^\dagger}=\Omega_{X^\dagger/\Z^\dagger}$, where $\Z$ is given the trivial log-structure, and denote by $j:D\to X$ the closed immersion. Then we claim that there is an exact sequence \begin{equation}
  0\lra \Omega_{X}\lra \Omega_{X^\dagger} \overset{\mathrm{res}}{\lra} j_*\mathcal{O}_{D}\lra 0,\label{eqn: residue exact sequence}
\end{equation}
 The first map in \eqref{eqn: residue exact sequence} is the canonical map, to see it is injective we argue as follows. Suppose $g\in\O_{X,x}$ is a local equation for $D$ near $x\in X$. Then we have the following pushout diagram.

  $$ \begin{tikzcd}
  \O_{X,x} \arrow[r, "\cdot g"] \arrow[d, "\cdot \mathrm{d}g"] & \arrow[d, "\cdot\mathrm{dlog}g"]\O_{X,x} \\
  \Omega_{X,x} \arrow[r]    & \Omega_{X^\dagger,x}                                     
 \end{tikzcd} $$
The upper horizontal arrow is injective, and therefore by abstract properties of pushouts the lower horizontal arrow is so too.

The second map of \eqref{eqn: residue exact sequence}, also called the \emph{Poincar\'e residue map} \cite[IV.1.2.14]{ogus}, is induced by the map $$\Omega_{X,x}\oplus \O_{X,x}(\mathrm{dlog}g)\lra (j_*\O_{D})_x=\O_{X,x}/(g): (\omega,h\cdot \mathrm{dlog}(g))\mapsto h\bmod{g},$$ which factors through the relation $\mathrm{d}g=g\,\mathrm{dlog}(g)$. By construction the residue map is surjective and $\mathrm{res}([(\omega,h\mathrm{dlog}(g))])=0$ only if $h=gk$ for some $k\in\O_{X,x}$, in that case the log-form $[(\omega,h\mathrm{dlog}(g))]=[(\omega+k\mathrm{d}g,0)]$ comes from $\Omega_{X,x}$. This settles the claim that \eqref{eqn: residue exact sequence} is an exact sequence.
 \end{block}

\begin{block}[Reminders on determinants.] \label{determinants}Let $X$ be a scheme. Recall that an $\mathcal{O}_X$-module $\mathscr{F}$ is called \emph{perfect} if $\mathscr{F}$ has finite projective dimension, or in other words, $\mathscr{F}$ locally admits a finite length projective resolution. If $X$ is a regular scheme then all locally finitely generated $\mathcal{O}_X$-modules are perfect \cite[066Z]{stacks}. 

Knudsen and Mumford \cite[Theorem 2]{KM} have shown the existence of a \emph{determinant functor}, unique up to unique isomorphism, which is a pair $(\det,i)$ consisting of a functor
$$\det:\mathbf{Perf}\text{-}\mathbf{Mod}_{\mathcal{O}_X}\text{-}\mathbf{Iso}\longrightarrow \mathbf{Inv}\text{-}\mathbf{Mod}_{\mathcal{O}_X}\text{-}\mathbf{Iso}$$
from the category of perfect $\mathcal{O}_X$-modules with isomorphisms to the category of invertible $\mathcal{O}_X$-modules with isomorphisms; together with ``adjunction'' isomorphisms $$i(\alpha,\beta):\det \mathscr{F}'\otimes \det \mathscr{F}'' \overset{\sim}{\longrightarrow} \det \mathscr{F}$$ for every short exact sequence of perfect $\mathcal{O}_X$-modules 
$$0\longrightarrow \mathscr{F}'\overset{\alpha}{\longrightarrow} \mathscr{F} \overset{\beta}{\longrightarrow} \mathscr{F}''\longrightarrow 0,$$ satisfying the following properties (a)-(g).

\begin{enumerate}[label=(\alph*)]
  \item Both $\det$ and $i$ commute with flat base change; 
  \item We have $\det 0=\mathcal{O}_X$;
  \item If $\mathscr{F}$ is a finite rank locally free $\mathcal{O}_X$-module then $\det \mathscr{F}$ is the top wedge power of $\mathscr{F}$.
  \item The adjunctions $i(0,\id_\mathscr{F})$ and $i(\id_\mathscr{F},0)$ are the obvious ones.
  \item If $\mathscr{F},\mathscr{F}',\mathscr{F}''$ are all locally free of finite ranks $r,r',r''$ respectively, then $i(\alpha,\beta)$ is the usual adjunction map determined by the rule $$(s_1'\wedge \cdots \wedge s'_{r'})\otimes (\beta(s_1)\wedge \cdots \wedge \beta(s_{r''}))\mapsto\alpha(s_1')\wedge \cdots \alpha(m'_{r'}) \wedge s_1\wedge \cdots s_{r''}$$ for any choice of sections $s'_1,\dots,s'_{r'}\in \mathscr{F}'$ and $s_1,\dots,s_{r''}\in \mathscr{F}$. 
  \item The adjunctions $i(\alpha,\beta)$ are ``functorial in isomorphisms of short exact sequences''.
  \item The adjunctions $i(\alpha,\beta)$ are ``functorial in short exact sequences of short exact sequences''.
\end{enumerate}

In general, $\det \mathscr{F}$ is constructed as an alternating product of determinants of the terms in a projective resolution, and carefully bookkeeping isomorphisms, see \cite{KM} for details.
\end{block}
\begin{definition}[log-canonical modules]\label{defn: log-canonical module} 

Let $f^\dagger:X^\dagger\to Y^\dagger$ be a morphism of log-schemes. In case the following conditions are met:
\begin{enumerate}[label=(\alph*)]
  \item the log-structures of $X^\dagger$ and $Y^\dagger$ are coherent;
  \item the morphism $f$ is locally of finite type;
  \item the $\O_X$-module $\Omega_{X^\dagger/Y^\dagger}$ is perfect;
 \end{enumerate} 
 then we say that \emph{the log-canonical module} of $f^\dagger$ is defined, and it is given by $$\omega_{X^\dagger/Y^\dagger}\coloneqq \det \Omega_{X^\dagger/Y^\dagger}.$$

Later in~\ref{construction: log-can modules for normal schemes} we will consider log-canonical modules in a more general setting.
\end{definition}
\begin{block} As a special case, when the log-structures are trivial, we recover the definition of canonical modules: let $f:X\to Y$ be a locally of finite type morphism of schemes such that $\Omega_{X/Y}$ is a perfect $\O_X$-module. Then the  \emph{canonical module} $\omega_{X/Y}=\det \Omega_{X/Y}$ is defined.\label{subs canonical modules}\end{block}

\begin{block} 
If $f^\dagger:X^\dagger\to Y^\dagger$ is a locally of finite type morphism of coherent log-schemes, then $\Omega_{X^\dagger/Y^\dagger}$ is a perfect $\O_X$-module if for instance $X$ is a regular scheme (by~\ref{determinants}) or if $X^\dagger$ and $Y^\dagger$ are both log-regular, as we explain now. If $X^\dagger$ and $Y^\dagger$ are log-regular, then by Lemma~\ref{lem: log-regulars log-lci} this implies that $f^\dagger$ is log-l.c.i. So by the claim in the proof of Proposition~\ref{propn : log-lci ses log-kahler} the sequence \eqref{eqn: 2nd cotangent} is exact and this yields a two-term locally free resolution of $\Omega_{X^\dagger/Y^\dagger}$.\end{block}{}

\begin{block}
Formation of $\omega_{X^\dagger/Y^\dagger}$ is compatible with flat base change in the following sense. If $\omega_{X^\dagger/Y^\dagger}$ is defined and $Y'^\dagger\to Y^\dagger$ is a morphism of log-schemes whose underlying scheme morphism is flat, then $\omega_{X'^\dagger/Y^\dagger}$ is defined and canonically isomorphic to $(X'^\dagger\to X^\dagger)^*\omega_{X^\dagger/Y^\dagger}$ where $X'^\dagger=X^\dagger\times_{Y^\dagger} Y'^\dagger$. The previous fact remains true in the categories of fine and \emph{fs} log-schemes (provided we switch to fine or \emph{fs} base change).\end{block}

\begin{block} \label{generic embedding sheaf} Let $X$ be an integral scheme and let $i:\eta\to X$ denote the inclusion of the generic point of $X$. Let $\FF$ be a coherent sheaf on $X$. Then $i^*\FF$ is a $K(X)$-vector space of dimension $\mathrm{rank}(\FF)$, and we write $\FF_{\eta}\coloneqq i_*i^*\FF$ for the associated constant sheaf on $X$. In the case where $\FF$ is torsion-free, we identify $\FF$ as a sub-$\O_X$-module of $\FF_{\eta}$ via the natural map $\FF\to i_*i^*\FF$.  

In particular, for any perfect $\O_X$-module $\FF$ we view $\det \FF$ as a sub-$\O_X$-module of $\wedge^{\mathrm{rk}(\FF)}\FF_\eta$. Any adjunction $i(\alpha,\beta)$ as in~\ref{determinants} is compatible with flat base change (property (a) of determinants), and so is generically compatible with $i(\alpha_\eta,\beta_\eta)$. \label{rmk: canonical as submodule ff}\end{block}



\begin{lemma} 
 \label{lemma: structure sheaf is sub of canonical} Let $f^\dagger:X^\dagger\to Y^\dagger$ be a morphism of log-schemes with generically trivial log-structures such that $\omega_{X^\dagger/Y^\dagger}$ is defined. Additionally assume that $f$ is a dominant generically \'etale morphism of integral schemes. Then $\omega_{X^\dagger/Y^\dagger,\eta}=\O_{X,\eta}$ and $\O_{X}\subset \omega_{X^\dagger/Y^\dagger}\subset \O_{X,\eta}$. So, $\omega_{X^\dagger/Y^\dagger}^{-1}$ is an ideal of $\O_X$.\end{lemma}
\begin{proof} The assumptions imply $\Omega_{X^\dagger/Y^\dagger,\eta}=\Omega_{X/Y,\eta}=0$ so that $\omega_{X^\dagger/Y^\dagger,\eta}=\O_{X,\eta}$. 
Choose a surjective map $\O_X^n\lra \Omega_{X^\dagger/Y^\dagger}$ for some $n\in\N$ and let $\mathscr{G}$ denote the kernel. Then $\mathscr{G}$ is free of rank $n$ and we have a sequence  $$0\lra \mathscr{G}\lra \O_X^n\lra \Omega_{X^\dagger/Y^\dagger}\lra 0.$$ By~\ref{determinants} (a) and (d) the associated adjunction isomorphism can then be viewed as an equality of submodules of $\O_{X,\eta}$ given by $\omega_{X^\dagger/Y^\dagger} = \left(\det \mathscr{G}\right)^{-1}$, and by~\ref{determinants} (c) we have $\det \mathscr{G}\subset \O_X$ as submodules of $\O_{X,\eta}$.\end{proof}

\begin{proposition}[log-adjunction] \label{propn: adj} Suppose $X^\dagger\overset{f^\dagger}{\lra} Y^\dagger\lra S^\dagger$ are morphisms locally of finite type of log-regular log-schemes with generically trivial log-structures. Further assume that the underlying morphism $f:X\to Y$ is a flat, dominant and generically smooth morphism of integral schemes. 
Then as $\O_{X}$-submodules of $\omega_{X/Y,\eta}$ we have the following equality, called \emph{log-adjunction}: \begin{equation}
  \omega_{X^\dagger/Y^\dagger} \otimes f^*\omega_{Y^\dagger/S^\dagger}=\omega_{X^\dagger/S^\dagger}.\label{eqn: log adj} 
\end{equation}\end{proposition}
  \begin{proof} This follows from Lemma~\ref{lem: log-regulars log-lci}, Proposition~\ref{propn : log-lci ses log-kahler} and the existence of the adjunction morphisms as in~\ref{determinants}. \end{proof}

\begin{block} Specialising Proposition~\ref{propn: adj} to the setting of trivial log-structures, we obtain the following. Let $f:X\to Y$ be a flat, dominant and generically smooth morphism of regular integral locally finite type $S$-schemes. Then as $\O_{X}$-submodules of $\omega_{X/Y,\eta}$ we have the following equality, called \emph{adjunction}: \begin{equation}
  \omega_{X/Y}\otimes f^*\omega_{Y/S}=\omega_{X/S},\label{eqn: adj}
\end{equation}
\end{block}

\begin{block}[Assumptions]\label{assumptions comparison can modules}
 For later reference consider the following assumptions. Let $f:X\to Y$ be a flat and generically smooth morphism of normal integral schemes, and $D_X$ and $D_Y$ are reduced effective divisors on $X$ and $Y$ respectively such that $D_X\subset f^*D_Y$ and suppose that $f\vert_{D_X}:D_X\to D_Y$ maps generic points to generic points. 

 Denote by $f^\dagger:X^\dagger=X(\log\,D_X)\lra Y^\dagger=Y(\log\,D_Y)$ the induced morphism of divisorial log-schemes.
\end{block}

\begin{block}[Construction] \label{construction: log-can modules for normal schemes}
Assumptions as in~\ref{assumptions comparison can modules}. In this section we construct the canonical sheaf $\omega_{X/Y}$ of $f$ and the log-canonical sheaf $\omega_{X^\dagger/Y^\dagger}$ of $f^\dagger$, in a fashion compatible with Definition~\ref{defn: log-canonical module}.
Since $Y$ is assumed normal there exists a regular open subscheme $V\subset Y$ such that $Y\setminus V$ has codimension at least $2$. We may suppose that $D_V\coloneqq (D_Y)\vert_V$ is a normal crossings divisor on $V$, because we can remove the singular locus of $D_V$ from $V$ and $\mathrm{codim}(Y\setminus V)\ge 2$ will still hold. 

Similarly, we can find a regular open subscheme $U\subset f^{-1}(V)$ so that $\mathrm{codim}(f^{-1}(V)\setminus U)\ge 2$ and $D_U\coloneqq (D_{X})\vert_U$ is normal crossings on $U$. By our assumption on the generic points of $D_X$ and $D_Y$, it follows that $\mathrm{codim}(X\setminus U)\ge 2$.

Now denote by $\iota:U\to X$ the open immersion and $U^\dagger\coloneqq U(\log (D_U))\to V^\dagger\coloneqq V(\log(D_V))$ for the induced morphism of log-schemes. Then $\omega_{U/V}$ and $\omega_{U^\dagger/V^\dagger}$ are well-defined and we define $$\omega_{X/Y}\coloneqq \iota^*\omega_{U/V}\text{ and } \omega_{X^\dagger/Y^\dagger}\coloneqq \iota^*\omega_{U^\dagger/V^\dagger}.$$

To see that this construction is independent of the choice of $V$ and $U$: it follows from \cite[3.3.1]{EHN} that the sheaf $\omega_{X/Y}$ (resp. the sheaf $\omega_{X^\dagger/Y^\dagger}$) is 
characterised as the unique reflexive sheaf whose restriction to $U$ coincides with $\omega_{U/V}$ (resp. $\omega_{U^\dagger/V^\dagger}$). This characterisation moreover implies that the construction is compatible with the earlier definitions of (log-)canonical sheaves, if these are applicable. Beware that a priori $\omega_{X/Y}$ and $\omega_{X^\dagger/Y^\dagger}$ are only defined as reflexive rank one sheaves, and we do not study the question in what situations these are invertible sheaves. 
\end{block}

\begin{theorem}[Comparison of canonical and log-canonical modules] \label{thm: relative canonical vs relative log canonical} Assumptions as~\ref{assumptions comparison can modules}. Then as $\mathcal{O}_X$-submodules of $\omega_{X/Y,\eta}$, we have 
 \begin{equation}
  \label{er} \omega_{X^\dagger/Y^\dagger}=\omega_{X/S}(D_X-f^*D_Y).
 \end{equation}
\end{theorem}
\begin{proof} By Construction~\ref{construction: log-can modules for normal schemes} it suffices to prove the equality in the case were $X$ and $Y$ are regular and $X^\dagger$ and $Y^\dagger$ are log-regular because the sheaves $\omega_{X/Y}$ and $\omega_{X^\dagger/Y^\dagger}$ are determined as the unique reflexive extensions of the sheaves $\omega_{U/V}$ and $\omega_{U^\dagger/V^\dagger}$, with $U,V$ as in~\ref{construction: log-can modules for normal schemes}.

By Proposition~\ref{propn : log-lci ses log-kahler} we have a commutative diagram
 $$\begin{tikzcd}0\arrow[r]& f^*\Omega_{Y} \arrow[r, ""] \arrow[d, ""] & \Omega_{X} \arrow[r] \arrow[d, ""] & \Omega_{X/Y} \arrow[d, "\lambda_{X/Y}"] \arrow[r] & 0 \\
 0\arrow[r]& f^*\Omega_{Y^\dagger} \arrow[r, ""]  & \Omega_{X^\dagger} \arrow[r] & \Omega_{X^\dagger/Y^\dagger} \arrow[r] & 0\end{tikzcd}$$
with exact rows. By the residue exact sequences \eqref{eqn: residue exact sequence} for $X^\dagger$ and $Y^\dagger$, the long exact sequence provided by the snake lemma has the form $$0\lra \mathrm{ker}\lambda_{X/Y} \lra f^*(j_Y)_*\O_{D_Y}\lra (j_X)_*\O_{D_X} \lra \mathrm{coker}\lambda_{X/Y}\lra 0,$$ where $j_{X}:D_X\to X$ and $j_Y:D_Y\to Y$ denote the closed embeddings of $D_X$ and $D_Y$ in $X$ and $Y$ respectively.
All terms in the previous exact sequence are finitely generated $\O_X$-modules, therefore these are perfect as $X$ is assumed regular. The adjunction isomorphisms then imply the equalities \begin{align*}
  \det (j_{X})_*\O_{D_X}\otimes \left(f^*\det (j_{Y})_*\O_{D_Y}\right)^{-1}&=\det\ker{\lambda_{X/Y}}\otimes \left(\det\coker{\lambda_{X/Y}}\right)^{-1}\\&=\omega_{X^\dagger/Y^\dagger}\otimes \omega_{X/Y}^{-1}
\end{align*} as submodules of $\omega_{X/Y,\eta}$.

To finish the proof we need to check that if $j:D\to X$ is the closed embedding of a reduced effective divisor on a regular scheme $X$, then $\det j_*\O_D=\O_X(D)$. This follows from the adjunction isomorphism associated to the short exact sequence $$0\lra \O_X(-D)\lra \O_X \lra j_*\O_D\lra 0.$$
\end{proof}

\begin{lemma} 
	Let $X^\dagger$ be a log-regular scheme over $S$ (equipped with the trivial log-structure) and let $E_x=\overline{\{x\}}$ be the log-stratum associated to a point $x\in F(X^{\dagger})$. 
	Let $D_x$ be as in \ref{lemma balanced cospec}. Then $$\left(\omega_{X/S}\otimes\O_{X}(D_X)\right)\vert_{E_x}\cong\omega_{E_x/S}\otimes \O_{E_x}(D_x).$$\label{lem specialisation log-can}
\end{lemma}
\begin{proof} 
By Lemma \ref{lemma balanced cospec}, $E_x^{\dagger}=E_x(\log D_x)$ is log-regular. By induction we can assume $x$ has height $1$. Then $E_x$ is a regular codimension $1$ subscheme in $X$. After possibly removing a codimension $2$ subscheme of $X$, we may assume that $X$ is regular (see Construction~\ref{construction: log-can modules for normal schemes}) and hence $E_x\to X$ is l.c.i., it then follows as in the proof of \ref{propn : log-lci ses log-kahler} that the (non-logarithmic) cotangent sequence \ref{eqn: 2nd cotangent} is left-exact: \begin{equation} 0\lra
        \mathscr{C}_{E_x/X}{\lra} \Omega_{X/S}\vert_{E_x}\lra \Omega_{E_x/S}\lra 0,
    \end{equation}
    By taking determinants and using \cite[\S9.1.36]{Liu} we obtain the desired adjunction formula: $$\omega_{\XX/S}\vert_{E_x}=\omega_{E_x/S}\otimes \CC_{E_x/X}\cong \omega_{E_x/S}\otimes\O(E_x\vert_{E_x})=\omega_{E_x/S}\otimes\O_{E_x}(D_X\vert_{E_x}-D_x).$$

\end{proof}
\begin{corollary} \label{cor sp}Suppose $X^\dagger$ is a generically smooth toroidal scheme over $S^{\dagger}$ and $x\in F(X^{\dagger})$.
	\begin{enumerate}[label=(\roman*)]
	 	\item If $x$ is vertical, then $\omega_{X^{\dagger}/S^{\dagger}}\vert_{E_x}\cong \omega_{E_x/\tilde{k}}\otimes \O_{E_x}(D_x)$. If $\tilde{k}$ is perfect this is $\omega_{E_x^{\dagger}/\tilde{k}}$.
	 	\item If $x$ is horizontal, then $E_x^{\dagger}$ is a toroidal scheme over $S^{\dagger}$ and $\omega_{X^{\dagger}/S^{\dagger}}\vert_{E_x}\cong \omega_{E_x/S}$. If $E_x$ is generically smooth (eg. $k$ is perfect) then this is $\omega_{E_x^{\dagger}/S^{\dagger}}$.
	 \end{enumerate} 
\end{corollary}
\begin{proof} By Theorem \ref{thm: relative canonical vs relative log canonical} we have the comparison formula $\omega_{\XX^{\dagger}/S^{\dagger}}\cong \omega_{\XX/S}\otimes\O_{\XX}(D_X-\pi^*D_S)$ where $\pi:X\to S$ is the structure map. If $x$ is vertical then $\pi^*D_S\vert_{E_x}=0$, and if $x$ is horizontal then $\pi^*D_S\vert_{E_x}=(E_x\to S)^*D_S$. The conclusions then follow from Lemma \ref{lem specialisation log-can} and \ref{thm: relative canonical vs relative log canonical}.
\end{proof}
\begin{block}[Trace section and relative canonical divisor]\label{subs: traces for models} Suppose $f^\dagger:X^\dagger\to Y^\dagger$ is a morphism of log-schemes such that $\omega_{X^\dagger/Y^\dagger}$ is defined, and additionally assume that $f$ is of relative dimension zero. Then by Lemma~\ref{lemma: structure sheaf is sub of canonical} the sheaf $\omega_{X^\dagger/Y^\dagger}$, viewed as an $\O_{X}$-submodule of $\omega_{X/Y,\eta}$, contains $\O_X$. We call the image of the unit section $1\in\Gamma(X,\O_{X})\to \Gamma(X,\omega_{X^\dagger/Y^\dagger})$ the \emph{trace section} $\tau_{X/Y}$. By construction, trace sections are compatible with localisation. 

The \emph{relative canonical divisor} is defined as $$K_{X^\dagger/Y^{\dagger}}\coloneqq \div(\tau_{X^\dagger/Y^{\dagger}}).$$ Note that $K_{X^{\dagger}/Y^{\dagger}}$ is an effective Cartier divisor. \label{rel can eff}
\end{block}

\section{Riemann-Hurwitz for skeleta}\label{sec:rh}
\begin{block}
In this section we apply the Poincar\'e-Lelong slope formula (Theorem \ref{thm: pll}) and our study of the relative log-canonical divisor of Section \ref{sec:log diff} to recover a Riemann-Hurwitz formula for covers of skeleta. We also study the potential theory of differential forms on skeleta.
\end{block}
\begin{block}[Notation] In this section, assume $k$ is perfect and let $f:\YY^\dagger\to\XX^{\dagger}$ be a finite cover of proper generically smooth toroidal $S$-schemes and suppose $f$ is \'etale above $\XX\setminus D_{\XX}$. Write $\phi:\Sigma'\to\Sigma$ for the associated cover of dual complexes. We also view $\phi$ as a simultaneous skeleton of $f^{\an}:(\YY_k)^{\an}\to(\XX_k)^{\an}$.
\end{block}

\begin{definition}[Tropical canonical divisors]
		\label{def trop rel can} 
		For a vertical ridge $\tau$ of $\Sigma$, we write $$\chi_{\Sigma}(\tau)\coloneqq 2-2g(E_\tau)-\val_{\Sigma}(\tau)=\deg \omega_{E_{\tau}^{\dagger}/\tilde{k}},$$
and we define the \emph{tropical canonical divisor} as $$K_{\Sigma}\coloneqq\sum_{\tau \text{ vertical}}-\frac{\chi_{\Sigma}(\tau)}{\vol(\tau)}[\tau].$$ We define the \emph{tropical relative canonical divisor} as $$K_{\Sigma'/\Sigma}=\sum_{\tau' \text{ vertical}}-\frac{\chi_{\Sigma'/\Sigma}(\tau')}{\vol(\tau')}[\tau'],$$
		where for brevity we write $\chi_{\Sigma/\Sigma'}(\tau')=\chi_{\Sigma'}(\tau)-[E_{\tau'}:E_{\tau}]\chi_{\Sigma}(\tau)$.
\end{definition}

\begin{block}[The different] The relative canonical divisor $K_{\YY^\dagger/\XX^{\dagger}}$ is effective Cartier (\ref{rel can eff}) and supported on $D_{\XX}$ (by the hypothesis that $f$ is \'etale above $\XX\setminus D_{\XX}$), and so $K_{\YY^{\dagger}/\XX^{\dagger}}$ is determined (see \ref{lem correspondence cartier and pl functions}) by a unique non-negative $\Z$-PL function
\begin{equation}
		\delta_{\YY^\dagger/\XX^{\dagger}}:\D(\YY^{\dagger})\to\R_{\ge0} 
\end{equation}
that we call the \emph{different} associated to $f:{\YY^\dagger\to\XX^{\dagger}}$.
\end{block}
\begin{lemma}\label{lem}
		For any $x'\in \sk(\YY^{\dagger})^{\div}$ we have $\delta_{\YY^\dagger/\XX^{\dagger}}(x')=\frac{1}{e(\H(x')/k)}\mathrm{length}_{\H(x')^{\circ}}\,\Omega^{\log}_{\H(x')^{\circ}/\H(f(x'))^{\circ}}$.
\label{propn comparison with analytic different}
\end{lemma}
\begin{proof} If $x'$ is the divisorial point associated to a component of $\YY_s$, then this follows from the fact that the trace section of \ref{subs: traces for models} is compatible with localisation and the fact that the different ideal $\mathrm{Fitt}^0\Omega^{\log}_{\H(x)^{\circ}/\H(f(x))^{\circ}}$ is equal to $(\omega_{\H(x)^{\circ}/\H(f(x))^{\circ}}^{\log})^{-1}$.

In general, we reduce to the previous setting using toroidal subdivisions as follows. The divisorial points of $\sk(\XX^{\dagger})$ coincide with the $\Q$-PL points of $\sk(\XX^{\dagger})$, see for instance \cite[\S3.2]{MN}. So there exists a proper barycentric subdivision $\XX_0\to \XX$ of such that $f(x')$ is the divisorial point associated to the exceptional component. By Lemma \ref{lemma pullback subdivision} the pullback subdivision $\YY_0\to \YY$ such that $x'$ is the divisorial point associated to a component of $(\YY_0)_s$. Since $g:\YY_0\to\YY$ is log-\'etale we have $g^*\omega_{\YY^{\dagger}/\XX^{\dagger}}=\omega_{(\YY_0)^{\dagger}/(\XX_0)^{\dagger}}$, which implies $\delta_{\YY^\dagger/\XX^{\dagger}}(x')=\delta_{(\YY_0)^\dagger/(\XX_0)^{\dagger}}(x')$.
\end{proof}

\begin{theorem}\label{thm:rh}  There exists a $\Z$-PL function $\delta:\sk(\YY^\dagger)\to \R_{\ge0}$, called the \emph{different}, such that for all $x'\in \sk(\YY^{\dagger})^{\div}$ we have $$\delta(x')=\frac{1}{e(\H(x')/k)}\mathrm{length}_{\H(x')^{\circ}}\,\Omega^{\log}_{\H(x')^{\circ}/\H(f(x'))^{\circ}}.$$ Moreover if $\YY^{\dagger},\XX^{\dagger}$ are $\Q$-factorial we have the Riemann-Hurwitz formula $$\Delta(\delta)=K_{\Sigma'/\Sigma}.$$
\end{theorem}
\begin{proof} The first statement follows from Lemma~\ref{propn comparison with analytic different}. For the second statement, note that 
	\label{restr can} by corollary \ref{cor sp} we have $\sp_*\omega_{\XX^{\dagger}/S^{\dagger}}=K_{\Sigma}$. By the adjunction formula \ref{propn: adj} and \ref{prop pullback commutes with specialisation} it then follows that $\sp_*K_{\YY^\dagger/\XX^{\dagger}}=K_{\Sigma'}-\phi^*K_{\Sigma}$.
		A computation, using \ref{lemma z-pl mu on polyhedron} and the pullback of cycles defined in \ref{cycles}, shows that the \emph{tropical adjunction formula} holds: $$K_{\Sigma'/\Sigma}=K_{\Sigma'}-\phi^*K_{\Sigma}.$$
	The conclusion now follows from the Poincar\'e-Lelong formula (Theorem~\ref{thm: pll}): $$\Delta(\delta)=\sp_*K_{\YY^\dagger/\XX^{\dagger}}=K_{\Sigma'/\Sigma}.$$ 
\end{proof}

\begin{remark}[Analytic different] In a series of works Temkin and collaborators~\cite{CTT,T16,T17,BT} have defined an \emph{analytic different function} $$\delta:(\YY)_k^{\an}\to\R_{\ge0}:x'\mapsto \log_{|\varpi|}\cont\Omega^{\log}_{\H(x')^{\circ}/\H(x)^{\circ}},$$ 
where for any $k^{\circ}$-module $Q$ we write $\cont Q=\inf_{\alpha}\cont Q_{\alpha}$ for the exponential length or \emph{content} \cite[\S3]{T16}, defined as infimum of contents of finitely presented subquotients.  If $x'\in \YY^{\an}$ is divisorial then $\delta_f(x')=\frac{1}{e(\H(x)/k)}\mathrm{length}_{\H(x)^{\circ}}\,\Omega^{\log}_{\H(x)^{\circ}/\H(f(x))^{\circ}}$. Note that we apply a logarithm to comply with our earlier notation, whereas Temkin e.a. usually work with multiplicative notation which is better suited for the non-discretely valued setting. 

In the upcoming work of H\"ubner-Temkin \cite{HT} it will be shown that over any nonarchimedean field $k$ (not necessarily discretely valued or algebraically closed) the analytic different function $\delta$ associated to a generically \'etale morphism of $k$-analytic spaces defines a continuous function which is additionally PL on PL-charts when $k$ is stable. If we admit continuity of $\delta$, defined via the aforementioned formula, then by Lemma \ref{lem} and the fact that $\sk(\XX^{\dagger})^{\div}$ is dense in $\sk(\XX^{\dagger})$ (see \cite[\S2.4.12]{MN}) we see that the constructions agree: $\delta\vert_{\sk(\XX^{\dagger})}=\delta_{\YY^{\dagger}/\XX^{\dagger}}$
\end{remark}
\begin{block}
Along the same lines of Theorem \ref{thm:rh} one can obtain the following result on the piecewise linearity of the ``absolute'' different $\lambda$. 
\end{block}
\begin{proposition}\label{lambda}
	There exists a $\Z$-affine function $\lambda:\sk(\XX^{\dagger})\to \R_{\ge0}$ such that for all $x\in \sk(\XX^{\dagger})^{\div}$ we have $\lambda(x)=\frac{1}{e(\H(x)/k)}\mathrm{length}_{\H(x')^{\circ}} \left(\Omega^{\log}_{\H(x)^{\circ}/k^{\circ}}\right)_{\mathrm{tor}}$.
\end{proposition}
\begin{proof}
	Similar as \ref{propn comparison with analytic different}, using the line bundle $\det \left[\left(\Omega_{\XX^{\dagger}/S^{\dagger}}\right)_{\mathrm{tor}}\right]$ instead.
\end{proof}
\begin{remark} Proposition \ref{lambda} this gives some evidence for an expectation of Temkin on the piecewise linearity of K\"ahler norms, see \cite[\S8.3.4]{T16}. However, we do not know if $\lambda(x)=\log_{|\varpi|}\cont \left(\Omega_{\H(x)^{\circ}/k^{\circ}}^{\log}\right)_{\tor}$ for all monomial points $x$. We also do not know how to describe $\Delta(\lambda)$.
	\label{rmk: kahler pl}
\end{remark}

\begin{block}[Potential theory of weight functions] \label{wt} We finish the paper by clarifying the relation between the different and the weight functions introduced by Musta{\c{t}}{\u{a}}-Nicaise \cite[\S4.3]{MN}.

 By \ref{cor sp} we have $\sp_*\omega_{\XX^{\dagger}/S^{\dagger}}^{\otimes m}=mK_{\Sigma}$. Suppose $\XX$ is $\Q$-factorial, then by \ref{thm: pll} it follows that if $\omega\in\Gamma(\XX_k,\omega_{\XX_k^{\otimes m}/k})\setminus \{0\}$ is a nonzero $m$-canonical form with $\div\omega$ supported on $D_X$. By the Poincar\'e-Lelong formula (Theorem \ref{thm: pll}) we have $$\Delta(F_{\div\omega})=mK_{\Sigma}.$$
By definition, the support function $F_{\div\omega}$ is equal to the \emph{weight function} $$\mathrm{wt}_{\omega}:\Sigma\to \R$$ on $\Sigma$  Now equip $\XX$ with the standard log-structure and let $\omega\in\Gamma(\XX_k,\omega_{\XX_k/k})\setminus \{0\}$. Then we obtain the formula $$\Delta(\mathrm{wt}_{\omega}\vert_{\Sigma})=mK_{\Sigma}-\rho_*\div\omega$$ where we define $$\rho_*\div\omega:=\sum_{(\sigma,\tau):\ \sigma \text{ unbounded},\ \tau \text{ vertical}}\partial_{\sigma/\tau}[\tau]$$ summing over facet-ridge pairs $(\sigma,\tau)$ with $\sigma$ unbounded and $\tau$ vertical. 
This generalises the results of \cite[\S3.2]{BN} to any dimension. 
\end{block}
\begin{proposition}[different as discrepancy of weight functions] Let $\omega\in \Gamma(\XX_k,\omega_{\XX_k/k}^{\otimes m})\setminus \{0\}$ be supported on $D_{\XX}$, then as functions on $\sk(\YY^{\dagger})$ we have
	$$\frac{1}{m}\delta_{\YY^{\dagger}/\XX^{\dagger}}=\wt_{f^*\omega}-\wt_\omega\circ \phi$$
\end{proposition}
\begin{proof} 
By the adjunction formula \ref{propn: adj} we have $$\omega_{\XX^{\dagger}/\YY^{\dagger}}\cong \underline{\mathrm{Hom}}(f^*\omega_{\YY^{\dagger}/S^{\dagger}},\omega_{\XX^{\dagger}/S^{\dagger}}),$$ and via this isomorphism the trace section $\tau_{\YY^{\dagger}/\XX^{\dagger}}$ (see \ref{subs: traces for models}) corresponds to the pull-back of $m$-canonical forms. The previous then implies the desired equality is true at all divisorial points of $\YY$ attached to components of $\YY_s$. By \ref{wt} and Theorem \ref{thm:rh} both left- and right-hand side define functions which are $\Z$-affine on $\sk(\YY^{\dagger})$, therefore they coincide.
\end{proof}

 \appendix
\section{Conventions on PL geometry}
\label{sec: polyhedral}
\begin{block} To fix notation, we collect here in this Appendix some well-known material on piecewise linear geometry. We first discuss cones and fans and their ``logarithmic'' description via monoids and monoidal spaces. Then we discuss $\Z$-PL spaces. 
\end{block}
\begin{block}[Cones]\label{block cones}
A \emph{convex rational polyhedral cone}, henceforth simply \emph{cone}, is defined as a pair $\sigma=(|\sigma|,N_{\sigma})$ consisting of a topological space $|\sigma|$ and a lattice $N_{\sigma}$ such that 
  $|\sigma|$ is the intersection of finitely many half-spaces of the form $$\{x\in (N_{\sigma})_\R:l(x)\ge0\}\quad\text{ for some }\quad l\in M_{\sigma}\coloneqq N_{\sigma}^*=\Hom(N_{\sigma},\Z)$$
  and require that $$|\sigma|^{\gp}=(N_{\sigma})_\R.$$
 Note that if $\sigma$ is a cone, then $|\sigma|$ defines a convex sub-$\R_{\ge0}$-monoid of $(N_\sigma)_\R$ which, as an $\R_{\ge0}$-monoid, is generated by finitely many elements of $N_\sigma$. We write $\langle\cdot,\cdot\rangle:M_{\sigma}\times N_{\sigma}\to\Z$ for the natural bilinear pairing; we call $M_{\sigma}$ the lattice of $\Z$-linear functions.

A morphism of cones $\sigma\to\sigma'$ is a group morphism $h:N_{\sigma}\to N_{\sigma'}$ with $h_{\R}(|\sigma|)\subset |\sigma'|$; note then $h_{\R}\vert_{|\sigma|}:|\sigma|\to|\sigma'|$ is a morphism of $\R_{\ge 0}$-monoids. This defines the category of cones $\mathbf{Cone}$.

 \label{gordan}
If $(\sigma,N_{\sigma})$ is a cone, then it is immediate that the monoid of integral points $\sigma\cap N_{\sigma}$ is integral, saturated and torsion-free. By Gordan's lemma $\sigma\cap N_\sigma$ is also finitely generated\footnote{proof: Write $N=N_{\sigma}$ and $\sigma=\sum_{i=1}^k \R_{\ge 0}m_i$ for some $m_1,\dots,m_k\in \sigma$. Denote by $\|\cdot\|_{N}$ the Euclidean norm on $N_\R$ induced by a $\Z$-basis of $N$. Then every element in $\sigma\cap N$ is at most a distance $\epsilon=\sum_{i=1}^k\|m_i\|_{N}$ removed from an element in the submonoid $\sum_{i=1}^k \N m_i$ of $\sigma\cap N$. It follows that $\{m_1,\dots,m_k\}$ and the collection of lattice points in the ball of radius $\epsilon$ around $0$ jointly generate $M_{\sigma}$, as desired.}.

 To simplify the notation, we will write $\sigma$ instead of $|\sigma|$ if the lattice $N_{\sigma}$ is clear from context. 
\end{block}
\begin{block}[Associated cone]\label{ass cone} Let $M$ be a fine torsion-free monoid. Then $M^{\gp}$ is a lattice and we view $M$ as an \emph{additive} submonoid of $M^{\gp}$. 
The $\R_{\ge0}$-submonoid of $M^{\gp}_\R$ generated by $M$ is denoted by $\R_{\ge 0}M$. We call the cone $(\R_{\ge0}M,M^{\gp})$ the \emph{cone associated to $M$}.

Note that $\R_{\ge 0}M=\R_{\ge0}M^{\mathrm{sat}}$, where $M^{\mathrm{sat}}$ denotes the saturation of $M$. Also, $M$ is sharp if and only if $\R_{\ge 0}M$ is strictly convex, that is $\R_{\ge 0}M$ contains no lines. In general there exists a set of elements of $M$, part of a system of generators, which generate $M^{\gp}_\Q$ as a $\Q$-vector space. This implies that $$(\R_{\ge 0}M)\cap M^{\gp}=M^{\mathrm{sat}}.$$

\end{block}
\begin{proposition}[fs monoids correspond to cones] \label{fs monoids correspond to cones}There exists an equivalence of categories $$\mathbf{Mon}^{\mathrm{fs}}\overset{\sim}{\lra} \mathbf{Cone}$$ between the categories of \emph{fs} monoids and cones, given by $M\mapsto (\R_{\ge 0}M,M^{\gp})$, with inverse $(\sigma,N_{\sigma})\mapsto \sigma\cap N_\sigma$. An \emph{fs} monoid $M$ is sharp if and only if its associated cone is strictly convex (i.e. contains no linear subspaces).
\end{proposition}
\begin{proof} This is immediate from~\ref{gordan} and~\ref{ass cone}.
\end{proof}
\begin{block}[Subcones and faces]\label{subcones} A \emph{subcone} of $\sigma=(|\sigma|,N_{\sigma})$ is a cone of the form $\sigma'=(|\sigma'|,N\cap\R|\sigma'|)$ for some $|\sigma'|\subset |\sigma|$. 
A subcone $\sigma'\subset \sigma$ is called a \emph{face}, also written $$\sigma'\le \sigma,$$ if additionally for each $x,y\in\sigma$ we have that $x+y\in\sigma'$ implies $x\in\sigma'$ and $y\in\sigma'$. The $n$-dimensional faces of $\sigma$ are denoted $\sigma(n)$. A \emph{ray} is a $1$-dimensional face, i.e. an element of $\sigma(1)$. A cone is called \emph{simplicial} if the integral ray generators are linearly independent.

We call a morphism of cones $h:\sigma'\to\sigma$ a \emph{immersion} if the cokernel of $h:N'\to N$ is torsion-free, equivalently $h(\sigma')$ defines a subcone of $\sigma$. 
 A \emph{face embedding} is an injective immersion with $h(\sigma')\le \sigma$.
\end{block}

 \begin{block}[Dual monoids]\label{block dual monoids}If $M$ is an abelian group, we denote by $M^*=\Hom_{\Z}(M,\Z)$ the \emph{dual group}. Recall that if $M$ is a lattice then the natural map $M\to M^{**}$ is an isomorphism. 
If $M$ is a monoid, then we denote by $M^{\vee}=\Hom(M,\N)$ the \emph{dual monoid} of $M$. It is immediate that $M^\vee$ is sharp. If $M$ is finitely generated, then $M^\vee$ is \emph{fs} by \cite[I.2.1.17(8)]{ogus}. Moreover, by \cite[I.2.2.3]{ogus} if $M$ is fine and sharp then the natural map $(M^\vee)^{\gp}\to (M^{\gp})^*$ is an isomorphism and the natural map $M\to M^{\vee\vee}$ factors through an isomorphism $M^{\sat}\cong M^{\vee\vee}$. So in particular for sharp \emph{fs} monoids the functor $M\to M^{\vee}$ is an involution. 
\end{block}
\begin{block}[Dual cone]\label{dual cone}
Suppose that $M$ is a fine torsion-free monoid and view $M$ as an additive submonoid of the lattice $M^{\gp}$. Then we denote by $$\sigma_M\coloneqq \R_{\ge0}\left(M^\vee\right)$$ the \emph{dual cone} associated to $M$; its associated lattice is $(M^\vee)^{\gp}=(M^{\gp})^*$. Suppose $M$ is sharp; by~\ref{block dual monoids} we have $(M^{\vee})^{\gp}_\R=\Hom_\Z(M^{\gp},\R)=\Hom_{\R}(M^{\gp}_\R,\R)$ and it follows $\sigma_M=\Hom(M,\R_{\ge0})=\mathrm{Hom}_{\R_{\ge0}}(\R_{\ge0}M,\R_{\ge 0}),$ that is $\sigma_M$ is the cone dual to the cone $\R_{\ge0}M$ associated to $M$ (as $\R_{\ge0}$-monoids).

Note that $\sigma_{M^\vee}=\R_{\ge0}M^{\vee\vee}\cong \R_{\ge0}M^{\sat}=\R_{\ge0}M$ is the cone associated to $M$. An application of Gordan's lemma shows that $M^{\sat}\cong\sigma_{M^\vee}\cap M^{\gp}$ is finitely generated. So, if $M$ is fine and sharp then the saturation $M^{\sat}$ is so too. 
 \end{block}
 \begin{block}\label{dual monoids viewed as monoids of functions}
 For each strictly convex cone $(\sigma,N_{\sigma})$, the monoid $\sigma\cap N_{\sigma}$ of integral points is \emph{fs} by~\ref{fs monoids correspond to cones}. We write $$\CC_{\sigma}\coloneqq(\sigma\cap N_\sigma)^\vee$$ for the \emph{monoid of non-negative $\Z$-affine functions on $\sigma$}, it is an \emph{fs} monoid by~\ref{block dual monoids}, and we have $\CC_{\sigma}^\gp=M_{\sigma}$ (the lattice of $\Z$-affine functions) and $\CC_{\sigma}=\{\alpha\in M_{\sigma}:\langle \alpha,\sigma\cap N_{\sigma}\rangle\subset \N\}$ justifying terminology. 
\end{block}
\begin{block}
  To sum up~\ref{dual cone} and~\ref{dual monoids viewed as monoids of functions}: for every sharp \emph{fs} monoid $M$ there exists a unique cone $\sigma$ such that $M=\CC_{\sigma}$ is the monoid of non-negative $\Z$-affine functions on $\sigma$, namely the dual cone $\sigma_M$.
   \end{block}

\begin{definition}[Cone complex]\label{def cone complex}
A \emph{convex rational polyhedral cone complex}, henceforth simply \emph{cone complex} is a pair $\Sigma=(|\Sigma|,\tau)$ consisting of a locally compact topological space $|\Sigma|$ equipped with a \emph{cone atlas} $\tau$, which is a set of cones satisfying the following conditions. 
\begin{enumerate}[label=(\roman*)]
\item The set $\{|\sigma|\}_{\sigma\in\tau}$ defines a covering of closed subsets of $|\Sigma|$.
\item Each face of a cone in $\tau$ is again a cone in $\tau$.
  \item For any two cones $\sigma_1,\sigma_2\in \tau$, there is a unique cone $\sigma_1\cap\sigma_2\in\tau$ such that $|\sigma_1\cap\sigma_2|=|\sigma_1|\cap|\sigma_2|$ and $\sigma_1\cap\sigma_2$ admits a face embedding in both $\sigma_1$ and $\sigma_2$.
\item Each $x\in \Sigma$ admits a neighbourhood of the form $|\sigma_1|\cup\dots\cup|\sigma_n|$ for some $\sigma_1,\dots,\sigma_n\in \tau$.
\end{enumerate}

To ease notation we will identify a cone complex $\Sigma$ with its geometric realisation $|\Sigma|$. The lattice points of $\Sigma$ are denoted by $\Sigma(\N)=\colim N_{\sigma}$. We write $M_{\Sigma}=\lim M_{\sigma}$ for the group of $\Z$-affine functions. 

We define a \emph{morphism of cone complexes} $\phi:\Sigma\to \Sigma'$ as a map $h:\tau\to\tau'$ together with a family $(h_\sigma:\sigma\to h(\sigma))_{\sigma\in\tau}$ of compatible cone morphisms, where compatibility means that whenever $\rho$ is a face of $\sigma$, the restriction $(h_\sigma)\vert_\rho$ is the composition of $h_\rho$ with the face embedding of $h(\rho)$ in $h(\sigma)$. Then $h$ induces a continuous map $|h|:|\Sigma|\to|\Sigma'|$ which we will often identify with $h$ for simplicity.  

We call $h$ a \emph{subdivision} if $|h|$ is injective and all $h_{\sigma}$ are immersions. 
We call $h$ a \emph{proper subdivision} if additionally $|h|$ is bijective and each cone $\sigma'$ is a finite union of cones in $\Sigma$. 

\end{definition}
\begin{remark} A definition equivalent to~\ref{def cone complex} is given in Definition~\ref{defn fan}, see Remark~\ref{fs fan and cone complex}.
\end{remark}
\begin{block}[Constructions with monoids.] Recall that $\mathbf{Mon}$ admits all limits and colimits, where limits and coproducts commute with the forgetful functor to $\mathbf{Set}$ and coequalisers are described via congruence relations, i.e. equivalence relations that are submonoids \cite[\S1]{ogus}.
\end{block}
\begin{block}[Ideals and faces.] \label{ideals}Let $M$ be monoid. A subset $I$ of $M$ is called an \emph{ideal} of $M$ if $IM\subset I$, that is $I$ is a sub-$M$-set of $M$. By convention $I=\emptyset$ is also an ideal, and we call $I$ proper if $I\ne M$. Note that every monoid $M$ admits a maximal proper ideal, namely $M\setminus M^\times$. A proper ideal $I$ of $M$ is called a \emph{prime ideal} if for all $m,m'\in M$ we have $mm'\in I$ only if $m\in I$ or $m'\in I$. Note that if $\p$ is a prime ideal then the complement $M\setminus \p$ is a submonoid. The submonoids which arise in this fashion are called \emph{faces} of $M$, and they are equivalently characterised as those submonoid $F\subset M$ for which $mm'\in F$ implies $m\in F$ and $m'\in F$. If $M$ is a fine monoid then every face of $M$ is also fine by \cite[I.2.1.17(3)]{ogus}. 
If $M$ is fine torsion-free then the map $F\mapsto \R_{\ge0}F$ yields a one-to-one correspondence between faces of $M$ and faces of $\R_{\ge0}M$.

Arbitrary unions of (prime) ideals are again (prime) ideals, and finite intersections of ideals are again ideals. If $\phi:M\to M'$ is a morphism of monoids then the inverse image of a prime ideal is a prime ideal.

\end{block}
\begin{block}[Quotients by ideals] \label{app ideal quotient} A quotient of a monoid $M$ by an ideal $I$ introduces an absorbing element or zero, i.e. an element $0$ such that $m\cdot 0=0$ for all $m\in M$. For instance, $(\N\cup\{+\infty\},+)$ has zero element $+\infty$. Any $M$-set $I$ (i.e. a set with action by $M$) can be given a $0$ by constructing the coproduct $I^0\coloneqq I\oplus 0$ in the category of $M$-sets. 
If $I$ is an ideal of $M$ then $M/I$ is defined as the cokernel of $I^0\to M^0$ in the category of $M$-sets. As a set $M^0/I^0=M\setminus I\oplus 0$, and $M^0/I^0$ is canonically a monoid with zero (products that lie in $I$ are declared equal to $0$). If $I$ is prime then $M/I$ is the face $M\setminus I$ with an additional zero. 
\end{block}
\begin{remark} The construction in \ref{app ideal quotient} suggests to work systematically with monoids with zero, this is possible after some small changes to make sure not to divide by zero (for instance one changes the notion of an integral monoid to require the cancellation law ``$ab=ac$ implies $b=c$ for all nonzero $a$''). We do not pursue this any further.
\end{remark}
\begin{block}[Spectrum of a monoid]

\label{spec monoid} The spectrum $\spec M$ of a monoid $M$ is the topological space obtained by equipping the collection of prime ideals of $M$ with the \emph{Zariski topology}: the closed sets are given by $Z(I)=\{\p\in \spec M:\ \p \supset I\}$ where $I$ is an ideal of $M$ -- equivalently the sets $D(f)=\{\p:\ f\not\in \p\}$ for $f\in M$, form a basis for the topology of $\spec M$. The spectrum yields a functor $$\spec:\mathbf{Mon}^{\mathrm{op}}\to\mathbf{Top}.$$

Let $M$ be a monoid and let $S$ be a submonoid. Then the localisation $S^{-1}M$ is the quotient of $S\times M$ by the congruence relation generated by the relations $(s,m)\sim (s',m')$ for all $s,s'\in S$ and $m,m'\in M$ such that $sm'=sm'$. For $f\in M$ we write $M_{f}=\{f^n: n\in \N\}^{-1}M$. We have a canonical homeomorphism $D(f)=\spec M_{f}$. For $\p \in \spec M$ we write $M_\p=(M\setminus \p)^{-1}M$, and as subsets of $\spec M$ the spectra $\spec M_\p$ and $\spec (M/\p\setminus\{0\})=\spec M\setminus \p$ are those primes $\mathfrak{q}$ contained in $\p$ and containing $\p$ respectively.

Let $f\in M$. Then there exists a smallest face containing $f$, namely $$\langle f\rangle\coloneqq\{m\in M:\exists m'\in M,\exists n\in\N: mm'=f^n\}.$$
Therefore $\spec M_f$ contains a unique closed point, namely $M_f\setminus \langle f\rangle$. The localisation $\langle f\rangle^{-1}M$ coincides with $M_f$.

Let $\dim M$ denote the Krull dimension of $\spec M$, then for each $\p\in \spec M$ we have that $\dim M_{\p}$ coincides with the \emph{height} $h(\p)$ of $\p$, defined as the supremum of the length of a strictly descending chain of prime ideals $$\p=\p_0\supset \p_1...\supset \p_h=\emptyset.$$ If $M$ is finitely generated then $h(\p)$ is finite and $h(\p)=\dim M_\p\le \mathrm{rank}_\Z(M^\sharp)^\gp$ with equality if $M$ is fine, by \cite[I.1.4.7]{ogus}. So if $M$ is fine then $\dim M=\dim (M^\sharp)^{\mathrm{gp}}$ is finite and so is $\spec M$.

If $M$ is fine and sharp then $\dim M=\dim \R_{\ge0}M$, and so if $F$ is a face of $M$ then the height $h(\p)$ of $\p=M\setminus F$ equals $\dim \R_{\ge0}M-\dim \R_{\ge0}F$.
\end{block}
\begin{block} \label{defn: mon space}
The category $\mathbf{MonSp}$ of \emph{monoidal spaces} is defined as follows. Objects are pairs $(T,\M_T)$ consisting of a topological space $T$ endowed with a sheaf of monoids $\M_T$. A morphism of monoidal spaces $(T,M_T)\to(T',M_{T'})$ is a pair $(f,f^\flat)$ consisting of a continuous map $f:T\to T'$ and a homomorphism $f^\flat:f^{-1}\M_{T'}\to M_{T}$ such that that induced stalk maps $$f^{\flat}_t:\M_{T',f(t)}\to\M_{T,t}$$ are \emph{local}, that is $\M_{T',f(t)}^\times \subset (f^{\flat}_t)^{-1}\M_{T,t}^\times$.
\end{block}
\begin{block}\label{local morphisms of monoids}
  Many of the constructions in $\mathbf{Mon}$ extend to sheaves of monoids on a topological space, including the construction of limits and colimits and the functors $(\cdot)^{\gp},(\cdot)^{\mathrm{int}},(\cdot)^{\mathrm{sat}},(\cdot)^{\sharp}$, by applying these constructions locally and then sheafifying the outcome. All these constructions commute with passing to stalks.

  For instance, if $(T,\M_T)$ is a monoidal space then we write $(T,\M_T)^\sharp=(T,\M_T^\sharp)$ for the \emph{sharpification}. We say $(T,\M_T)$ is \emph{sharp} if $\M_T=\M_T^\sharp$, equivalently all stalks $\M_{T,t}$ are sharp. Note that if $(f,f^\flat):(T,M_T)\to(T',M_{T'})$ is a morphism of sharp monoidal spaces then the requirement that $f^{\flat}_t$ is local means that $(f^{\flat}_t)^{-1}\{1\}=\{1\}$.\end{block}
\begin{block}\label{monoscheme} Let $M$ be a monoid and let $T=\spec M$. We define a sheaf of monoids $\M_T$ on $T$ via the rule $D(f)\mapsto M_f$ for all $f\in M$ and sheafifying the result. Since $D(f)=\spec M_f$ admits a unique closed point $\p=M\setminus \langle f\rangle$, it follows that $\M_T(D(f))=\M_{T,\p}={\langle f\rangle}^{-1}M=M_f$, so sheafifying was in fact not needed. In particular if $f=1$ is trivial we find that $\M_T(T)=M$. 

A computation similar as the one for schemes, see for instance \cite[II.1.2.2]{ogus} shows that the functor $$\mathbf{MonSp}\to \mathbf{Set}:(T,\M_T)\mapsto \Hom(M,\Gamma(T,\M_T))$$
  is representable by the monoidal space $(\spec M,\M_{\spec M})$. 

\end{block}

\begin{block}[Affine fans]
\label{affine fans}  
An \emph{affine fan} is defined as a monoidal space of the form $(\spec M,\M_{\spec M})^\sharp$, in particular these are always sharp. We have $\M_{\spec M}^\sharp=\M_{\spec M^{\sharp}}$ and by~\ref{monoscheme} it follows that the category of affine fans is anti-equivalent to the category of sharp monoids. In particular, affine fans admit fibered products. The affine fan $(\spec M,\M_{\spec M})^\sharp$ is called \emph{fs} if $M$ is \emph{fs}. 
\end{block}
\begin{definition}\label{defn fan}
  A \emph{fan} is defined as a sharp monoidal space which is locally isomorphic to an affine fan. A fan is called \emph{locally fs} (resp. \emph{fs}) if it is a union (resp. a finite union) of \emph{fs} affine fans. The category of fans (resp. locally fs fans) is written $\mathbf{Fan}$ (resp. $\mathbf{Fan}^{fs}$). The fibered product for affine fans glues to fibered products for fans \cite[II.1.3.5]{ogus}, and similarly for locally \emph{fs} fans. 

 A fan $F$ is called \emph{regular} at a point $x$ if $F$ is locally \emph{fs} at $x$ and $$\mathrm{rank}_\Z\M_{F,x}^\gp=\dim \M_{F,x}.$$ We write $F_{\mathrm{reg}}$ for the regular points of $F$.
Given a fan $F$ and a monoid $M$, we will write $$F(M)=\Hom(\spec M,F).$$ If $F$ is a fan we call $F(\N)$ the set of \emph{integral points} on $F$, and similarly we have the sets of \emph{rational points} $F(\Q_{\ge0})$ and \emph{real points} $F(\R_{\ge0})$. 
\end{definition}
\begin{block} Let $M$ be a sharp monoid, then the set $(\spec M)(\N)$ of integral points of $\spec M$ coincides with $M^\vee$. In particular, if $M$ is also fine then $(\spec M)(\N)$ is the integral points of the dual cone $\sigma_M$ (see~\ref{dual cone}) and $(\spec M)(\R_{\ge0})=\sigma_M$.
\end{block}

\begin{construction}[Cone complex associated to a fan $F$]\label{block: class fans vs kato fans} \label{cone complex of a fan}
To every locally \emph{fs} fan $F$ we associate a strictly convex cone complex $\Sigma(F)$ as follows.

For every $x,y\in F$ such that $x$ is a specialisation of $y$, that is $x\in\overline{\{y\}}$, we have a local morphism $\M_{F,y}\lra \M_{F,x}$, 
 and so we get an immersion of dual cones \begin{equation}
  \sigma_{\M_{F,x}}\lra \sigma_{\M_{F,y}}.\label{eqn: cone maps}
\end{equation}

The maps \eqref{eqn: cone maps} define face embeddings and hence glue to a well-defined cone complex $\Sigma(F)=(|\sigma(F)|,\{\sigma_{\M_{F,x}}\}_{x\in F})$ such that
\begin{equation}
  |\Sigma(F)|=\cup_{x\in F}\sigma_{\M_{F,x}}.\label{cone complex of a fan}
\end{equation}
In other words $\Sigma(F)=\mathrm{colim}_{x\in F}\sigma_{\M_{F,x}}$.


\end{construction}
\begin{block}\label{fs fan and cone complex}
  Conversely, from every strictly convex cone complex one can reconstruct a locally \emph{fs} fan and this gives an equivalence of categories -- we omit the details. 

Under this equivalence regular points of a locally \emph{fs} fan correspond to \emph{regular} cones in the corresponding cone complex, where a regular cone is defined as a simplicial cone (\ref{subcones}) for which the integral points are generated by ray vectors. 

The parallel definitions of subdivisions (\ref{def cone complex}) are follows.
\end{block}
\begin{definition}[Subdivision] \label{def subdivision}
A subdivision of locally \emph{fs} fans is defined as a morphism $\phi:F'\to F$ such that 
  \begin{enumerate}[label=(\alph*)]
    \item For all $x'\in F'$, the map $\M^{\gp}_{F,\phi(x)}\to \M^{\gp}_{F',x}$ surjects. 
    \item The map $F'(\N)\to F(\N)$ is injective 
  \end{enumerate}
  A subdivision $\phi$ is \emph{proper} if additionally $\phi$ has finite fibers and $F'(\N)\to F(\N)$ is bijective.
\end{definition}

\begin{proposition}[Existence regular subdivisions]\label{existence nice proper subdivision}
  Let $F$ be a locally \emph{fs} fan. Then there exists a proper subdivision $\phi:F'\to F$ such that $F'$ is regular and moreover we can require that the induced map $\phi^{-1}(F_{\mathrm{reg}})\lra F_{\mathrm{reg}}$ is an isomorphism. We will also call $\phi$ a \emph{regular} subdivision.
\end{proposition}
\begin{proof} See either \cite[\S6.6.31]{GR} or \cite[\S4.6.1]{W}.
\end{proof}
\begin{block}\label{minimal subdivision}
  Let $F$ be a $2$-dimensional \emph{fs} fan. Then we claim there exists a \emph{minimal} regular subdivision $F'\to F$. Indeed, if $\sigma$ is a $2$-dimensional cone then consider the convex hull $H$ of $|\sigma|\cap N_{\sigma}$. Consider the set $S$ of points in $N_{\sigma}\cap H$ which are integral ray generators. A computation shows that $S$ is the minimal generating set of $|\sigma|\cap N_{\sigma}$. The rays of these generators define a subdivision that contains any regular subdivision of $\sigma$. 
  The set $S$ can be computed explicitly using Hirzebruch-Jung continued fractions, see for instance \cite[\S1]{oda}.
\end{block}

\begin{block} 
For the remainder of this Appendix we recall the definition and some elementary properties of $\Z$-PL spaces. In the literature, these are also known as $\Z_{|k^\times|}-$ or $(\Z,|k^\times|)-$PL spaces. Since $k$ is discretely valued, we can fix an isomorphism $|k^\times|\cong \Z$ and we will omit $|k^\times|$ to lighten the notation. 

These spaces generalise the cone complexes defined in~\ref{def cone complex}, in fact we recover these by letting $b=0$ throughout Definitions~\ref{def polyhedral complex} and~\ref{def PL space} below.
\end{block}


\begin{definition}[Polyhedra] If $N$ is a lattice a \emph{$\Z$-affine function} on $N_\R$ as any function of the form $$\varphi:x\mapsto \langle a,x\rangle+b$$ for some $a\in N^*$ and $b\in\Z$. A \emph{$\Z$-PL polyhedron}, or short \emph{polyhedron} is a pair $\sigma=(|\sigma|,N_{\sigma})$ of a topological space $|\sigma|$ and a lattice $N_{\sigma}$ such that $|\sigma|$ is the intersection of finitely many half-spaces $$\{x\in (N_{\sigma})_{\R}:\varphi(x)\ge0\}\subset (N_{\sigma})_\R^n$$ for some $\Z$-affine functions $\varphi$ on $\N_{\sigma}$ and $(N_{\sigma})_{\R}=\R|\sigma|$. The last condition ensures $$C(\sigma)\coloneqq (\overline{\R_{\ge0}|\sigma|},N_{\sigma})$$ is a cone with the same lattice $N_{\sigma}$ -- here $\overline{\cdot}$ denotes the closure in $(N_{\sigma})_{\R}$. We call $C(\sigma)$ the \emph{cone} associated to $\sigma$, note $N_{\sigma}=N_{C(\sigma)}$, and $M_{\sigma}\coloneqq N_{\sigma}^*=M_{C(\sigma)}$ is the lattice of $\Z$-affine functions on $\sigma$.
The monoid of non-negative $\Z$-affine functions on $\sigma$ is $\CC_{\sigma}\coloneqq\CC_{C(\sigma)}$.

 A \emph{subpolyhedron} of $\sigma=(|\sigma|,N_{\sigma})$ is a polyhedron of the form $\sigma'=(|\sigma'|,N_{\sigma}\cap \R|\sigma'|)$ with $|\sigma'|\subset |\sigma|$. A \emph{face} of a polyhedron $\sigma$ is a subpolyhedron obtained as an intersection of $|\sigma|$ with the boundary of a half-space $H_l$ containing $\sigma$, and is possibly $\emptyset$ or $\sigma$. Given a polyhedron $\sigma$ we write $\sigma(n)$ for the set of $n$-dimensional faces of $\sigma$.
\label{def polyhedral complex}

A \emph{morphism of polyhedra} $\sigma'\to\sigma$ is a group morphism $h:N_{\sigma'}\to N_{\sigma}$ such that $h_\R(|\sigma'|)\subset|\sigma|$. We call $h$ an \emph{immersion} if $\mathrm{coker}(h)$ is torsion-free, equivalently $h(\sigma')$ defines a subpolyhedron. A \emph{face embedding} is an injective immersion with $h(\sigma')\le \sigma$. 

A \emph{polytope} is a polyhedron which is bounded (equivalently, compact). 


\end{definition}
\begin{lemma} \label{lemma: props Z-pl function on polyhedra} Let $\sigma$ be a polyhedron. Then $\CC_{\sigma}^{\gp}\subset M_{\sigma}$ with equality if and only if $C(\sigma)$ is strictly convex.
\end{lemma}
\begin{proof} It follows from the definitions $\CC_{\sigma}^{\gp}\subset M_\sigma$.  If $C(\sigma)$ is strictly convex, then we can choose a $\Z$-affine function $\varphi'\in M_{\sigma}$ such that $C(\sigma)\setminus\{0\}\subset\{\varphi'>0\}$.
 Let $n_1,\dots,n_k$ be primitive ray generators of $C(\sigma)$. Let $\varphi\in M_{\sigma}$, after rescaling we may assume $\varphi'(n_i)>\max\{-\varphi(n_i),0\}$ for all $i=1,\dots,k$. Since the function $\max\{-\varphi,0\}$ is convex, it follows $\varphi+\varphi'\in \CC_{\sigma}$ and so $\varphi=(\varphi+\varphi')-\varphi'\in \CC_{\sigma}^{\gp}$. Conversely, any element of $\CC_{\sigma}$ vanishes on linear spaces contained in $C(\sigma)$, so the same is true for elements of $\CC_{\sigma}^{\gp}$. So any element of $M_{\sigma}$ which is non-trivial on a non-trivial linear space contained in $C(\sigma)$ does not lie in $\CC_{\sigma}^{\gp}$.
\end{proof}

\begin{definition}[$\Z$-PL space]
\label{def PL space}
	
A \emph{$\Z$-PL polyhedral complex}, or simply \emph{$\Z$-PL complex} is a pair $\Sigma=(|\Sigma|,\tau)$ 
 consisting of a locally compact topological space $|\Sigma|$ and a \emph{polyhedral atlas} $\tau$, which is a set of polyhedra such that (i)-(iv) holds:
\begin{enumerate}[label=(\roman*)]
\item The set $\{|\sigma|\}_{\sigma\in\tau}$ defines a covering of closed subsets of $|\Sigma|$.
\item Each face of a polyhedron in $\tau$ is again a polyhedron in $\tau$.
  \item For any two polyhedra $\sigma_1,\sigma_2\in \tau$, there is a unique polyhedron $\sigma_1\cap\sigma_2\in\tau$ such that $|\sigma_1\cap\sigma_2|=|\sigma_1|\cap|\sigma_2|$ and $\sigma_1\cap\sigma_2$ admits a face embedding in both $\sigma_1$ and $\sigma_2$.
\item Each $x\in |\Sigma|$ admits a neighbourhood of the form $|\sigma_1|\cup\dots\cup|\sigma_n|$ for some $\sigma_1,\dots,\sigma_n\in \tau$.
\end{enumerate} 
We write $$M_{\Sigma}=\lim_{\sigma\in \Sigma}M_{\sigma}$$ 
for the group of functions on $\Sigma$ that are $\Z$-affine on the faces. Similarly define $\CC_{\Sigma}=\lim_{\sigma\in\Sigma}\CC_{\sigma}$ for the non-negative $\Z$-affine functions.

A morphism of polyhedral complexes $h:(\Sigma,\tau)\to(\Sigma',\tau')$ is a map $h:\tau\to\tau'$ together with a family $(h_\sigma:\sigma\to h(\sigma))_{\sigma\in\tau}$ of compatible polyhedron morphisms, where \emph{compatibility} means that whenever $\psi$ is a face of $\sigma$, the restriction $(h_\sigma)\vert_\psi$ is the composition of $h_\psi$ with the face embedding of $h(\psi)$ in $h(\sigma)$. For simplicity we often identify $h$ with its geometric realisation $|h|:|\Sigma|\to|\Sigma'|$. A morphism $h:\Sigma\to\Sigma'$ induces natural maps $h^*:M_{\Sigma'}\to M_{\Sigma}$ and $h^*:M_{\Sigma'}^+\to M_{\Sigma}^+$. We call $h$ a \emph{subdivision} if $|h|$ defines a homeomorphism onto its image and the $h_\sigma$ can be chosen to be injective. We call $h$ a \emph{proper subdivision} if additionally $|h|$ is bijective, each $h_{\sigma}$ is an immersion and each polyhedron $\sigma'$ is a finite union of polyhedra in $\Sigma$. 

Consider the equivalence relation on polyhedral complexes generated by proper subdivisions. We then define a \emph{$\Z$-PL space} as an equivalence class of polyhedral complexes, and a morphism of $\Z$-PL spaces is induced from a morphism of representative polyhedral complexes. By a polyhedron of a $\Z$-PL space $\Sigma$ we mean any polyhedron of a representative. To ease notation, we will often identify a $\Z$-PL space $\Sigma=(|\Sigma|,[\tau])$ with its support $|\Sigma|$, and suppress the equivalence class of polyhedral atlasses $[\tau]$. The \emph{group of $\Z$-PL-functions on $\sigma$} is defined as $$\mathrm{PL}_{\Sigma}\coloneqq \mathrm{colim}_{\Sigma'\to\Sigma}M_{\sigma'}$$ where the colimit is taken over all (not necessarily proper) subdivisions $\Sigma'\to\Sigma$; similarly define $\mathrm{PL}_{\Sigma}^+$. A morphism $\Sigma'\to\Sigma$ then induces natural homomorphisms $\mathrm{PL}_{\Sigma}\to \mathrm{PL}_{\Sigma}$ and $\mathrm{PL}_{\Sigma}^+\to \mathrm{PL}_{\Sigma}^+$. It follows from the definitions that a continuous function $\Sigma\to\R$ is $\Z$-PL (resp. $\N$-PL) if and only if it is $\Z$-PL (resp. $\N$-PL) on the polyhedra of some polyhedral atlas. 

One may similarly define $\Q$-PL and $\R$-PL-spaces. By a PL-space we mean a $\R$-PL space. 
\end{definition}

\begin{block}[Associated cone complex]
To any $\Z$-PL space $\Sigma$ we can functorially associate a cone complex $C(\Sigma)$ as follows: The functor which sends a polyhedron $(\sigma,N_{\sigma})$ to the cone $C(\sigma)=(\overline{\R_{\ge0}|\sigma|},N_{\sigma})$ assocatied to $\sigma$, preserves immersions and face embeddings, and so globalises to define a functor $C(\cdot)$ from the category of $\Z$-PL spaces to the category of cone complexes. 
\end{block}
\begin{block}[Recession cone complex]\label{block recession fan}
Let $\Sigma$ be a $\Z$-PL space, then we construct the \emph{recession cone complex} $\mathrm{rec}(\Sigma)$ as follows.
If $\sigma$ is a polyhedron, then $$|\mathrm{rec}(\sigma)|=\{v\in (N_{\sigma})_{\R}:\sigma+\R_{\ge0}v\subset \sigma\},$$ and the induced lattice is the restriction of $N_{\sigma}$. This construction is functorial and preserves face embeddings, hence glues to any $\Z$-PL space.
The cone complex $\mathrm{rec}(\Sigma)$ is trivial if and only if $\Sigma$ is bounded, i.e. all its polyhedra are bounded, and $\mathrm{rec}(\Sigma)$ is strictly convex if and only if $\Sigma$ is. 
\end{block}
\begin{lemma}
		Let $\Sigma$ be a polyhedral complex. Then the recession cone complex $\mathrm{rec}(\Sigma)$ is a subcomplex of $C(\Sigma)$.
\end{lemma}
\begin{proof}
		It suffices to consider the case of a polyhedron $\sigma$. Let $v\in \mathrm{rec}(\sigma)$. Then for all $w\in \sigma$ and $\lambda\in\R_{>0}$ we have $w+\lambda v\in \sigma$ and so $\lambda^{-1}w+v\in C(\sigma)$. Take the limit as $\lambda\to\infty$ and  it follows $v\in C(\sigma)$.
\end{proof}

\begin{block}[Stars and links] \label{stars and links} Let $\Sigma$ be a polyhedral complex. Given a polyhedron $\tau$ of $\Sigma$ we write $\mathrm{star}(\tau)=\cup_{\sigma \ge \tau}\sigma$ for the \emph{star} of $\tau$, and $$\mathrm{link}_{\Sigma}(\tau)=\overline{\mathrm{star}}_{\Sigma}(\tau)\setminus\mathrm{star}_{\Sigma}{\tau}$$ for the link of $\tau$, which is again a PL space. For any polyhedron $\rho$ of $\mathrm{link}_{\Sigma}(\tau)$ we write $\tau+\rho$ for the unique minimal polyhedron of $\overline{\mathrm{star}}_{\Sigma}(\tau)$ containing $\tau$ and $\sigma$.  

It is immediate to verify that $\star(\rec(\sigma))=\rec(\star(\sigma))$ and $\link(\rec(\sigma))=\rec(\link(\sigma))$.
\end{block}

\end{document}